\documentclass{amsart}[11pt]
\usepackage{subfigure}
\usepackage[margin=3.2cm]{geometry}
\usepackage{color}
\usepackage{graphicx}
\usepackage{enumerate}
\usepackage{xcolor}
\usepackage{mathtools,amsthm,amssymb, amsmath, amssymb, amsfonts, epsfig}
\usepackage{setspace, stmaryrd, verbatim, euro, enumerate}
\usepackage{bbm}
\usepackage{appendix}
\usepackage{multirow}
\usepackage{colortbl}
\usepackage{makecell}
\mathtoolsset{showonlyrefs}
\newtheorem{theorem}{Theorem}[section]
\newtheorem{corollary}[theorem]{Corollary}
\newtheorem{lemma}[theorem]{Lemma}
\newtheorem{proposition}[theorem]{Proposition}

\numberwithin{equation}{section}

\theoremstyle{definition}
\newtheorem{definition}[theorem]{Definition}
\newtheorem{assumption}[theorem]{Assumption}

\theoremstyle{plain}
\newtheorem{example}[theorem]{Example}
\newtheorem{remark}[theorem]{Remark}

\definecolor{winered}{rgb}{0.6,0,0}
\definecolor{ocean}{rgb}{0,0.1,0.6}
\definecolor{forest}{rgb}{0,0.5,0.1}
\definecolor{fill}{rgb}{.92,1,.92}
\definecolor{sunset}{rgb}{.8,0.3,0}
\definecolor{lightgray}{rgb}{.92,0.96,0.92}

\def\ocean#1{\textcolor{ocean}{#1}}
\def\forest#1{\textcolor{forest}{#1}}

\def\sunset#1{\textcolor{sunset}{#1}}

\usepackage[colorlinks=true, linkcolor=ocean,citecolor=winered]{hyperref}

\newcommand{\norm}[1]{\left\lVert#1\right\rVert}

\newcommand{\abs}[1]{\left\lvert#1\right\rvert}
\newcommand{\normbT}[1]{\left\lVert#1\right\rVert_\bT}

\newcommand\cA{\mathcal{A}}

\newcommand\cC{\mathcal{C}}
\newcommand\cD{\mathcal{D}}

\newcommand\cF{\mathcal{F}}
\newcommand\cG{\mathcal{G}}

\newcommand\cI{\mathcal{I}}
\newcommand\cJ{\mathcal{J}}
\newcommand\cK{\mathcal{K}}

\newcommand\cO{\mathcal{O}}

\newcommand\cS{\mathcal{S}}
\newcommand{\GUG}{\cS_\Upsilon^\Gamma}
\newcommand\cT{\mathcal{T}}
\newcommand\cU{\mathcal{U}}
\newcommand\cV{\mathcal{V}}
\newcommand\cW{\mathcal{W}}
\newcommand\cX{\mathcal{X}}

\newcommand{\Zm}{\boldsymbol Z}
\newcommand{\Ym}{\boldsymbol Y}
\newcommand{\Am}{\boldsymbol A}

\newcommand{\BS}{\mathrm{BS}}
\newcommand{\Ac}{\cA\cC}
\newcommand{\Aco}{\Ac_{0}}

\newcommand{\bb}{\mathfrak{b}}

\newcommand\bE{\mathbb{E}}

\newcommand\bN{\mathbb{N}}
\newcommand\bP{\mathbb{P}}

\newcommand\bR{\mathbb{R}}
\newcommand\bT{\mathbb{T}}

\newcommand{\D}{\mathrm{d}}
\newcommand{\E}{\mathrm{e}}

\newcommand{\half}{\frac{1}{2}}

\newcommand{\ep}{\varepsilon} 
\newcommand{\vphi}{\varphi} 
\newcommand{\notthis}[1]{}
\newcommand{\one}{\mathbbm{1}}
\newcommand{\overbar}[1]{\mkern 3mu\overline{\mkern-3mu#1\mkern-1mu}\mkern 1mu}

\newcommand{\dt}{\,\mathrm{d}t}
\newcommand{\ds}{\,\mathrm{d}s}
\newcommand{\du}{\,\mathrm{d}u}

\newcommand{\DD}{\mathrm{D}}
\newcommand{\II}{\mathrm{I}}

\newcommand{\LDP}{\mathrm{LDP}}
\newcommand{\MDP}{\mathrm{MDP}_\beta}
\newcommand{\rrho}{\overline{\rho}}
\usepackage{charter}

\begin{document}

\title{Large and moderate deviations for stochastic Volterra systems}
\date{\today}
\author{Antoine Jacquier}
\address{Department of Mathematics, Imperial College London, and Alan Turing Institute}
\email{a.jacquier@imperial.ac.uk}
\author{Alexandre Pannier}
\address{Department of Mathematics, Imperial College London}
\email{a.pannier17@imperial.ac.uk}
\thanks{}
\maketitle

%%%%%%%%%%%%%%%%%%%%%%%%%%%%%%%%%%%%%%%%%%%%%%%%%%%%%%%%%%%%%%%%%%%%%%%%%%%%%%%%%%%%%%%%%
\setcounter{tocdepth}{3}
%\tableofcontents

%%%%%%%%%%%%%%%%%%%%%%%%
\begin{abstract}
We provide a unified treatment of pathwise large and moderate deviations principles for a general class of multidimensional stochastic Volterra equations with singular kernels, not necessarily of convolution form. 
Our methodology is based on the weak convergence approach by Budhijara, Dupuis and Ellis~\cite{BD19, DE97}.
We show in particular how this framework encompasses most rough volatility models used in mathematical finance, yields pathwise moderate deviations for the first time and generalises many recent results in the literature.
\end{abstract}

%%%%%%%%%%%%%%%%%%%%%%%%%%
\section{Introduction}

This paper sheds new light on the asymptotic behaviour of the class of stochastic Volterra equations (SVEs) 
\begin{equation}
    X_t = X_0 + \int_0^t K(t,s)  b(s,X_s)\D s + \int_0^t K(t,s)  \sigma(s,X_s)\D W_s, \quad t\in[0,T],
    \label{eq:SVEintro}
\end{equation}
for some fixed time horizon~$T>0$,
where $X_0\in\bR^d$, $d\ge1$, $W$ is a multidimensional Brownian motion,
$K$ is a kernel that may be singular, and the coefficients are such that a unique pathwise solution exists. 
This class of models has been investigated in many fields, including nonlinear filtering~\cite{CD99} 
using fractional Brownian motion kernels,
pharmacokinetic models~\cite{Marie14} (Langevin equation driven by fractional Brownian motion),
fluid turbulence~\cite{Chevillard17},
and turbulence modelling in atmospheric winds or energy prices~\cite{BNPJ14,CHPP13}
using Brownian semistationary processes. 

Mathematical finance has however been the most dynamic area by far in terms of applications of SVEs,
and an in-depth study of~\eqref{eq:SVEintro} in the affine case with convolution kernels was recently carried out
by Abi Jaber, Larsson and Pulido~\cite{ALP17}.
Following previous analyses supporting non-Markovian systems \cite{Alos07,CV12a,CV12b,Comte98,Comte12,Fukasawa11}, the investigation of high-frequency data in~\cite{GJR14} revealed the roughness, in the sense of low H\"older regularity, of the observed time series of the instantaneous volatility of stock price processes. 
This suggested that fractional Brownian motion (fBm) with small Hurst parameter ($H\approx 0.1$)
is an accurate driver for its dynamics.
Since this seminal observation, more advanced results~\cite{EFR18} have proposed that the drift and the diffusion coefficients should be state dependent, 
giving rise to the widespread development of~\eqref{eq:SVEintro} in quantitative finance.

For option pricing purposes, the asymptotic results in~\cite{Alos07,BFGHS17,Fukasawa17} 
showed that the short-maturity behaviour of option prices is captured much more accurately
by these rough volatility models rather than by Markovian diffusions. 
Reconciling the stylised facts of the markets from both the statistical and the option pricing viewpoints is the tour de force that make these models so important today.
However, the loss in tractability compared to classical It\^o diffusions is not negligible. 
The solution to~\eqref{eq:SVEintro} is in general not a semimartingale nor a Markov process, 
preventing the use of It\^o calculus or Feynman-Kac type formulas. 
Path-dependent versions of the latter are available in some cases, in particular for affine rough volatility models~\cite{ALP17, Cuchiero18, ER19, GatheralKR}, 
but general results are scarce~\cite{VZ18}.
Rough path theory is a natural route but is not available for $H\le 1/4$, 
although a regularity structure approach was recently developed~\cite{BFGMS19}. 
In this context, one could turn to numerical methods to understand the dynamics or to price options but, despite new advances based on Monte-Carlo methods~\cite{BFG16,BLP15,MP18}, rough Donsker theorem~\cite{HJM17} or Fourier methods~\cite{ER19},
the roughness and memory of the process seriously complicate the task. 

Asymptotic methods have been used, both to provide clearer understanding of models in extreme parameter configurations 
and to act as proxies to numerical schemes.
Large Deviations Principles (LDP), in particular, have been widely explored in mathematical finance, 
and we refer the interested reader to~\cite{FGGJTBook} for an overview.
Let $\{X^\ep\}_{\ep>0}$ be a sequence of random variables in some Polish space~$\cX$, converging in probability to a deterministic limit~$\overbar{X}$ as~$\ep$ goes to zero. 
This sequence is said to satisfy an LDP with speed $\ep^{-1}$ and rate function $I:\cX\to[0,+\infty]$ if for all Borel subsets $B\subset\cX$, the inequalities
\begin{align*}
-\inf_{x\in B^\circ} I(x) \le 
\liminf_{\ep \downarrow 0} \ep \log \bP \big( X^\ep\in B \big) 
\le \limsup_{\ep \downarrow 0} \ep \log \bP \big( X^\ep\in B \big) 
\le -\inf_{x\in \overbar{B}} I(x)
\end{align*}
hold, and the level sets $\{x\in\cX: I(x)\le N\}$ of~$I$ are compact for all $N>0$. 
This rate function encompasses in a (relatively) concise formula first-order information about the asymptotic behaviour of complex dynamical systems. 
If~$X$ satisfies~\eqref{eq:SVEintro} and $\cX=\bR^d$, one can consider finite-dimensional LDP for $\{X_t\}_{t\ge0}$ (also called small-time LDP if the limit takes place as~$t$ goes to zero), 
or pathwise LDP for some rescaling of~$X$ with $\cX=\cC([0,T]:\bR^d)$. 
The former is easily recovered from the latter by a projection argument. 
Moderate deviations however are concerned with deviations of a lower order than large deviations, 
and thus apply to `less rare events'.
We indeed say that $\{X^\ep\}_{\ep>0}$ satisfies a moderate deviations principle (MDP) if $\{\eta^\ep\}_{\ep>0}$ satisfies an LDP 
with speed~$h_\ep^2$, where
\begin{equation*}
\eta^\ep := \frac{X^\ep-\overbar X}{\sqrt{\ep}  h_\ep}, \quad \text{for all } \ep>0,
\label{eq:MDPintro}
\end{equation*}
with $\lim_{\ep\downarrow0} h_\ep   = +\infty$ and $\lim_{\ep\downarrow0} \sqrt{\ep} h_\ep   =0$. Since the speed of convergence of $h_\ep$ is not fixed, an MDP essentially bridges the gap between the central limit regime where $h_\ep  =1$ and the LDP regime where $h_\ep  =\ep^{-1/2}$. 
An edifying example of the relevance of moderate deviations appears in~\cite{FGeP18} where, interested in option pricing asymptotics, the authors judiciously rescale the strikes with respect to time to expiry. Indeed, as time to expiry becomes smaller, the range of pertinent strikes naturally shrinks, 
and this `moderately-out-of-the-money' regime becomes more realistic.

Large deviations for SVEs were originally studied in~\cite{Nualart00,Rovira00} with regular kernels.
In the context of rough volatility, Forde and Zhang~\cite{FZ15} introduced the first finite-dimensional LDP where the log-volatility is modelled by a fractional Brownian motion, and refined versions followed in~\cite{BFGMS19,BFGHS17,FGaP18},
while pathwise LDP for similar models were studied in~\cite{Cellupica19, Gulisashvili18}. 
Departing from regular conditions on the behaviour of the coefficients led to specific requirements, 
and finite-dimensional large deviations for the fractional Heston model were carried out in~\cite{FGS19, GJRS18}, 
while more elaborate pathwise LDPs were derived for the rough Stein-Stein model with random starting point~\cite{HJL19},
for the rough Bergomi model~\cite{JPS18}, 
and small-time LDPs for the multi-factor rough Bergomi appeared in~\cite{LMS19}. We emphasise at this point that no pathwise MDP was previously known in the context of rough volatility.

The G\"artner-Ellis theorem~\cite[Theorem 2.3.6]{DZ10} is the main ingredient of a finite-dimensional LDP and depends on explicit computations of certain limits of the Laplace transform. 
This is only available though, when the process is either Gaussian~\cite{FZ15} or affine~\cite{FGS19}. 
Pathwise LDP on the other hand, have mainly been derived using the Freidlin-Wentzell approach~\cite{Freidlin84}:
starting from known large deviations for the driving (Gaussian) process~\cite[Theorem 3.4.5]{DS89},
they follow from a combination of approximations and continuous mapping, 
keeping track of the rate function.   
While this methodology is clear, 
it requires a case-by-case tailored path for each model, and in general leads to a cumbersome rate function. 
Furthermore, pathwise moderate deviations are so far out of reach in this approach, 
partially explaining the small number of related results compared to LDP.

A radically different method, introduced by Dupuis and Ellis in the monograph~\cite{DE97} and developed further by Budhiraja and Dupuis~\cite{BD19}, relies on the equivalence between the LDP and the Laplace principle.
The family $\{X^\ep\}_{\ep>0}$ is said to satisfy the Laplace principle with speed~$\ep^{-1}$ and rate function $I:\cX\to[0,+\infty]$ if for all continuous bounded maps $F:\cX\to\bR$,
\begin{equation}
    \lim_{\ep\downarrow0} - \ep \log \bE\left[\exp\left\{-\frac{F(X^\ep)}{\ep}\right\}\right] =  \inf_{x\in\cX}\big\{I(x)+F(x)\big\}.
    \label{eq:Laplace}
\end{equation}
This alternative, called the weak convergence approach, consists in proving a Laplace principle where the left-hand side pre-limit of~\eqref{eq:Laplace} can be represented as a variational principle for expectations of functionals of Brownian motion~\cite[Theorem 3.1]{boue98}:
\begin{lemma}[Bou\'e-Dupuis]
Let $W$ be an $\bR^m$-Brownian motion and $F$ be a bounded Borel-measurable function mapping $\cC([0,T]:\bR^m)$ into $\bR$. Then
\begin{equation}
- \log \bE \left[ \E^{-F(W)}\right] 
=\inf_{v\in\cA} \bE \left[\half \int_0^T \abs{v_s}^2\ds + F\bigg(W + \int_0^\cdot v_s \ds\bigg) \right],
\label{eq:boue}
\end{equation}
where 
\begin{equation}
\cA := \bigg\{v:\Omega\times [0,T]\to\bR^m \text{ progressively measurable, }\bE\left[ \int_0^T |v_t|^2\dt \right] < +\infty \bigg\}.
\label{eq:cA}
\end{equation}
\end{lemma}
The representation~\eqref{eq:boue} contains in a single formula the usual tools used in the proof of an LDP. The first term on the right-hand side comes from the relative entropy between the Wiener measure and the measure shifted by $\int_0^\cdot v_s\ds$ via Girsanov's theorem, under which $W+\int_0^\cdot v_s\ds$ is a Brownian motion. It can be interpreted as the cost of deviating from the original path and clearly indicates where the form of the rate function comes from. In essence, this representation replaces the non-linear analysis of the Freidlin-Wentzell approach with the linear theory of weak convergence. Instead of exponential estimates, only qualitative properties of the shifted process need to be established, such as strong existence and uniqueness and tightness.

The extensive literature on the topic, summarised in~\cite{BD19} and the references therein, demonstrates the strength of this generic approach which can be applied to a variety of models without appealing to their particular features. It has been used to derive LDPs, in the continuous-time case, for diffusions~\cite{Chiarini14}, multiscale systems~\cite{BDG18, DS12,Kostas13}, SVEs with singular kernels and Lipschitz continuous coefficients~\cite{Zhang08}, SDEs driven by infinite-dimensional Brownian motions~\cite{BDM08}, by Poisson random measures~\cite{BCD13} or both~\cite{BDM11}, including stochastic PDEs. 
Contrary to the Freidlin-Wentzell approach, this method has also proved efficient to obtain MDPs. 
SVEs with Lipschitz continuous kernels~\cite{Li17}, SDEs with jumps~\cite{BDG16} and slow-fast systems~\cite{Morse17} are a few relevant examples. 
The latter were then tailored to the setting of stochastic volatility models in~\cite{JK20}, 
developing the first application of the weak convergence approach in mathematical finance,
and extending the MDP results in~\cite{FGeP18} to a pathwise setting.  
A further appealing feature of moderate deviations is the simple form, often quadratic, of the rate function,
as opposed to that provided by large deviations, 
thereby opening the gates to the use of importance sampling and variance reduction techniques~\cite{DSW12,Rob10,Morse18}.

Building on this powerful approach, we provide a unified treatment of (finite-dimensional and pathwise) large and moderate deviations in the general framework~\eqref{eq:SVEintro} 
by showing the weak convergence of a perturbed system.
We relax the uniqueness requirement for the limiting Volterra equation, 
as in~\cite{CDMD15,Donati04} for the diffusion case, allowing us to consider coefficients that are not Lipschitz continuous and do not necessarily have sublinear growth.

The paper is organised in the following way:
Section~\ref{section:framework} introduces the framework and useful definitions. 
In section~\ref{section:main}, we present abstract criteria for the validity of an LDP, 
extending the results by Budhiraja and Dupuis~\cite[Theorem 4.1]{BD01}. 
Our main results, Theorem~\ref{thm:LDP} for LDP and Theorem~\ref{thm:MDP} for MDP, are then stated in the case of convolution kernels and extended to non-convolution kernels in Theorem~\ref{thm:nonconv}. 
In Section~\ref{section:rvol}, we show how these results apply to rough volatility models, 
and give precise formulae for 
the rough Stein-Stein, the (multi-factor) rough Bergomi and the rough Heston models. 
We finally gather technical proofs in the appendix.

%%%%%%%%%%%%%%%%%%%%%%%%%%%%%%%%%%%%%%%%%%%%%%%%%%%%%%%%%%%

\section{General framework}
\label{section:framework}
\subsection{Notations}
We consider a fixed time horizon $T>0$, and denote $\bT:=[0,T]$, $\bR_+ := [0,+\infty)$ and~$\overline{\bR}_+ := [0,+\infty]$.
For $d_1\ge1$, $d_2\ge1$, $\abs{\cdot}$ denotes the Euclidean norm in~$\bR^{d_1}$ and the Frobenius norm in~$\bR^{d_1\times d_2}$, and~$\llbracket d_1,d_2\rrbracket:= \{d_1,\cdots,d_2\}$.
For $p\ge1$,~$L^p$ stands short for~$L^p(\bT)$, and $\norm{\cdot}_2$ is the usual~$L^2$ norm.
Furthermore, for~$d\ge1$, $\cW^d := \cC(\bT:\bR^d)$ represents the space of continuous functions from~$\bT$ to~$\bR^d$, equipped with the supremum norm $\norm{\vphi}_\bT := \sup_{t\in\bT} \abs{\vphi_t}$ for any $\vphi\in\cW^d$. 
Finally, for any $d_1\ge1$, $d_2\ge1$, $M>0$,  $f:\bR^{d_1}\to\bR^{d_2}$,
we write $\norm{f}_M := \sup\{ \abs{f(x)}: \abs{x}\le M\}$.
Unless stated otherwise, constants will be denoted by~$C$ (with possible subscript) and may be different from one proof to another. Every statement involving~$\ep$ stands for all $\ep>0$ small enough.
A family of random variables will be called tight if the corresponding measures are tight~\cite[Appendix A]{DE97}.
We also use the classical convention that the infimum over an empty set is equal to infinity.
Finally, we recall the following definitions for clarity and notations:
\begin{definition}\label{def:modulusG}
Let $g$ be a function from~$\bR^d$ to~$\bR^n$.
\begin{itemize}
\item It has linear growth if there exists~$C_L>0$ such that 
$\abs{g(x)} \le C_L (1+\abs{x})$, for all~$x\in\bR^d$;
\item if it is uniformly continuous, it admits a continuous and increasing modulus of continuity $\rho_g:\bR_+\to\bR_+$, with $\rho_g(0) =0$ and
$\abs{g(x)-g(y)} \le \rho_g(\abs{x-y})$, for all $x,y\in\bR^d$;
\item it is locally~$\delta$-H\"older continuous with $\delta\in (0,1)$ if, 
for all~$M>0$, there exists~$C_M>0$ such that 
$\abs{g(x)-g(y)} \le C_M \abs{x-y}^\delta$, for all~$\abs{x}\vee\abs{y}\le M$.
\end{itemize}
\end{definition}

\subsection{Framework}
We consider small-noise convolution stochastic Volterra equations (SVE)
\begin{equation}
    X_t^\ep = X^\ep_0 + \int_0^t K(t-s)  b_\ep(s,X_s^\ep)\D s + \vartheta_\ep  \int_0^t K(t-s)  \sigma_\ep(s,X^\ep_s)\D W_s, \quad t\in\bT,
    \label{eq:Xep}
\end{equation}
taking values in~$\bR^d$ with $d\geq 1$, where $\ep>0$,
and $\vartheta_\ep >0$ tends to zero as~$\ep$ goes to zero.
For each $\ep>0$, $X_0^\ep\in\bR^d$, $b_{\ep}:\bT\times\bR^d\to\bR^d$ and $\sigma_{\ep}:\bT\times\bR^d\to\bR^{d\times m}$ are Borel-measurable functions, 
and~$W$ is an~$m$-dimensional Brownian motion on the filtered probability space~$(\Omega,\cF,\{\cF_t\}_{t\in\bT},\bP)$ satisfying the usual conditions. 
The kernel function~$K:\bT\to\bR^d\times\bR^d$, of convolution type, is allowed to be singular, 
thus encompassing fractional processes, in particular the recent literature on rough volatility~\cite{Alos07,BFG16,EFR18,GJR14}.
Components of the system are in general correlated, 
the correlation matrix being implicitly encoded in the diffusion coefficient~$\sigma_\ep$.
General existence and uniqueness results for such stochastic Volterra equations are so far out of reach, 
and our conditions below are sufficient and general enough for most applications.
In order to state them precisely, we first introduce several definitions and concepts:

\begin{definition}
For any $\ep>0$, a solution to~\eqref{eq:Xep} is an~$\bR^d$-valued progressively measurable stochastic process~$X^\ep$ satisfying~\eqref{eq:Xep} almost surely and such that
\begin{equation}
    \bP \left( \int_0^t \Big\{\abs{ K(t-s) b_\ep(s,X^\ep_s)} + \abs{ K(t-s) \sigma_\ep(s,X_s^\ep)}^2 \Big\} \D s <\infty, \text{  for all  }t\in \bT
\right) =1.
\label{eq:condsol}
\end{equation}
%We call it exact if it is pathwise unique.
\end{definition}
We shall assume that the (singular) convolution kernel satisfies the following condition, 
which is essentially a multivariate version of the  one given in~\cite[Condition (2.5)]{ALP17}:
\begin{assumption}
\label{assu:K}
The kernel~$K:\bT\to \bR^{d\times d}$ is an upper triangular matrix satisfying the following conditions:
$K\in L^2(\bT:\bR^{d\times d})$ and
there exists~$\gamma\in(0,2]$ such that, for~$h$ small enough,
\[
\int_0^h \abs{K(t)}^2\dt + \int_0^T \abs{K(t+h)-K(t)}^2 \dt=  \cO(h^\gamma).
\]
\end{assumption}
We refer to~\cite[Example 2.3]{ALP17} for a broad range of kernels that satisfy this assumption. 
Of particular interest in mathematical finance is the Riemann-Liouville kernel~$K(t)=t^{H-\half}$, for $H\in(0,\half)$ implying~$\gamma=2H$. 
Moreover if~$\widetilde K$ is locally Lipschitz and~$K$ satisfies Assumption~\ref{assu:K} 
then so does the product~$K \widetilde K$; this includes the gamma and power-law kernels which are related to the class of Brownian semistationary processes~\cite{BNS09}.
\begin{remark}
This setup covers in particular the following two useful forms for the kernel:
\begin{itemize}
    \item  $K = \mathrm{Diag}(k_1, \cdots, k_d)$ is a diagonal matrix, 
    where each~$k_i:\bT\to\bR$ satisfies Assumption~\ref{assu:K}.
    \item The drift and diffusion coefficients of any sub-system of~\eqref{eq:Xep} can be convoluted with different kernels.
    As an example, the one-dimensional SVE
   \[
    Y_t^{\ep} = Y_0^{\ep} + \int_0^t K_1(t-s) b_\ep^Y\left(s,Y_s^{\ep} \right)\ds + \int_0^t K_2(t-s) \sigma_\ep^Y\left(s,Y_s^{\ep} \right)\D W_s^{(1)}
   \]
   is the first component of~\eqref{eq:Xep} with $d=2$ and
\[
K = \begin{pmatrix} K_1 & K_2 \\ 0 &0 \end{pmatrix}, \quad 
b_\ep = \begin{pmatrix} b_\ep^Y \\ 0 \end{pmatrix}, \quad 
\sigma_\ep = \begin{pmatrix} 0&0\\  \sigma_\ep^Y & 0 \end{pmatrix}.
\]
\end{itemize}
One could in fact use general matrices to eliminate the auxiliary state, such as~$K = (K_1 \; K_2):\bT\to\bR^{1,2}$ in the example above. We stick to square matrices for consistency.
\end{remark}

Volterra systems appearing in the literature, and in particular in the mathematical finance one,
have a specific structure in the sense that only one component satisfies an SVE with (singular) kernel,
and can be dealt with independently of the other component.
This particular structure allows us to relax some conditions on the coefficients,
and we shall leverage on it whenever needed.
We make this more specific through the following two definitions:

\begin{definition}\label{def:autonFunction}
Let $\Upsilon\subset \llbracket 1,d \rrbracket$ and $\Gamma:\bR^{\abs{\Upsilon}}\to [0,\infty)$.
We define~$\GUG$ as the set of functions~$f$ for which
there exists a strictly positive constant~$C_{\Upsilon}$ such that, for all~$x\in\bR^d$, 
\begin{equation}
    \abs{f(x)} \le C_{\Upsilon} \Big(1 + \abs{x}_{\Upsilon^c} + \Gamma\big( x^{(\Upsilon)}\big) \Big) ,
\label{eq:UpsilonGrowth}
\end{equation}
where~$\abs{x}_{\Upsilon^c} := \sum_{i\in \Upsilon^c} \abs{x^{(i)}}$ 
and $x^{(\Upsilon)} := (x^{(i)})_{i\in\Upsilon}$.
\end{definition}

\begin{definition}\label{def:autonomous}
The process~$X^\ep$ admits an autonomous $\GUG$-subsystem 
$\{ X^{\ep,(l)}\}_{l\in \Upsilon}$ 
if for all $1\le i, j\le d$ such that $K_{ij} \neq 0$, $b_\ep^{(j)}(t,\cdot)$ and all the components of the row~$\sigma_\ep^{(j)}(t,\cdot)$ satisfy the following for small enough~$\ep$ and uniformly in~$t\in\bT$:
\begin{itemize}
    \item if $i\in \Upsilon$, they have linear growth and do not depend on~$X^{\ep, (k)}$ for $k\in\Upsilon^c$;
    \item if $i\in \Upsilon^c$, they belong to~$\GUG$.
\end{itemize}
\end{definition}

\begin{example}\label{ex:Bergomi}
The motivation for Definition~\ref{def:autonomous} is to be able to handle (rough) stochastic volatility models, 
ubiquitous in mathematical finance, where linear growth of all the coefficients may not hold.
Consider for example the rough Bergomi model~\cite{BFG16}
\begin{align*}
\left\{
\begin{array}{rl}
    X_t^{(1)} &= \displaystyle -\half \int_0^t \exp\left(X^{(2)}_s\right) \ds
     + \int_0^t \exp\left(\half X^{(2)}_s\right) \D B_s, \\
    X^{(2)}_t &= \displaystyle y_0  - at^{2H} + \int_0^t (t-s)^{H-\half} \D W_s,
\end{array}
\right.
\end{align*}
where $B$ and $W$ are Brownian motions with correlation~$\rho\in(-1,1)$.
After dropping the dependence in~$\ep$, this fits into the setup of~\eqref{eq:Xep} with~$d=3$, $y_0\in\bR$, $a>0$, $H\in(0,\half), \rrho=\sqrt{1-\rho^2}$ and
\[
K(t) = \begin{pmatrix} 1 & 0 & 0 \\ 0 & t^{2H-1} & t^{H-\half} \\ 0&0&0 \end{pmatrix},
\quad b\big(t,(x_1,x_2)\big) = \begin{pmatrix} - \E^{x_2}/2 \\ -a/(2H) \\ 0
\end{pmatrix},
\quad \sigma\big(t,(x_1,x_2)\big) = \begin{pmatrix} \rrho \, \E^{ x_2/2 } & \rho \, \E^{ x_2/2 } & 0 \\ 0&0&0 \\ 1 &0&0 
\end{pmatrix},
\]
where the third component is meaningless but allows us to handle the two different kernels.
Here~$X$ admits~$X^{(2)}$ as autonomous subsystem with $\Upsilon=\{2\}$ and
$\Gamma(x_2)= 1+\E^{x_2}$.
\label{ex:autonomous}
\end{example}

%%%%%%%%
The following set of assumptions, inspired by~\cite{Chiarini14}, completes our framework:
\begin{enumerate}[{\bfseries H1.}]
    \item $X^\ep_0$ converges to $x_0\in\bR^d$ as $\ep$ tends to zero.
    \item  For all~$\ep>0$ small enough, the coefficients~$b_\ep$ and~$\sigma_\ep$ are measurable maps on $\bT\times\bR^d$
    and converge pointwise to~$b$ and~$\sigma$ as~$\ep$ goes to zero. Moreover, $b(t,\cdot)$ and~$\sigma(t,\cdot)$ are continuous on~$\bR^d$, uniformly in~$t\in\bT$. 
    \item Either {\bf a)} or {\bf b)} holds:
    \begin{enumerate}[{\bf a)}]
        \item For all~$\ep>0$ small enough, $b_\ep$ and~$\sigma_\ep$ have linear growth  uniformly in~$\ep$ and in~$t\in\bT$.
        \item The process~$X^\ep$ admits an autonomous $\GUG$-subsystem.
    \end{enumerate}
    \item The SVE~\eqref{eq:Xep} is pathwise unique for small enough~$\ep>0$.
\end{enumerate}
{\bf H2} ensures that, on compact subsets of~$\bT\times\bR^d$, the convergence of~$b_\ep$ and~$\sigma_\ep$ is uniform and that~$b$ and~$\sigma$ are uniformly continuous.
{\bf H1, H2, H3a} are standard and easily verifiable. {\bfseries H3b} is unusual but includes a large number of functions; Assumption~\ref{assu:Gammabound} will complete it to indicate the role of~$\Gamma$ such as to include Example~\ref{ex:autonomous}. Moreover, the growth conditions from \textbf{H3} are uniform in~$\ep$ and therefore apply to the limits~$b$ and~$\sigma$. The main restrictions arise from {\bf H4}, although the latter is satisfied if, for instance, the coefficients $b_\ep$ and $\sigma_\ep$ are locally Lipschitz continuous for small enough $\ep>0$. 
This condition was relaxed in~\cite{MS15} to the one-dimensional case where~$K(t)=t^{-\alpha}$,
for $\alpha\in(0,\half)$ and~$\sigma(x)=x^\gamma$, for $\gamma\in(\frac{1}{2(1-\alpha)},1]$, 
which is clearly not Lispchitz continuous. 
Furthermore, to the best of our knowledge there currently exists no pathwise LDP for stochastic equations where pathwise uniqueness fails.

One can compare our setup with the SVE considered by Zhang in~\cite{Zhang08}, where a pathwise LDP was derived under the assumptions of Lipschitz continuity and linear growth of the coefficients. We relax both these assumptions.
Indeed, our framework covers H\"older-continuous diffusion coefficients, as mentioned in the previous paragraph, and functions with non-linear growth through the concept of autonomous~$\GUG$-subsystem, as in Example~\ref{ex:Bergomi}.

%%%%%%%%%%%%%%%%%%%%%%%%%%%%%%%%%%%%%%%%%%%%%%%%%%%%%%%%%%%%%%%%%%%%%%%%%%
%%%%%%%%%%%%%%%%%%%%%%%%%%%%%%%%%%%%%%%
%%%%%%%%%%%%%%%%%%%%%%%%%%%%%%%%%%%%%%%%%%
\section{Large and moderate deviations}
\label{section:main}
As discussed in the introduction, our goal is to provide pathwise large and moderate deviations 
for the general convolution stochastic Volterra system~\eqref{eq:Xep}, and then extend these to non-convolution kernels.
The classical Freidlin-Wentzell approach, used in~\cite{Freidlin84}, has limitations regarding the behaviour of the coefficients, 
and the rate function is often rather cumbersome to write.
We follow here instead the weak convergence approach developed by Dupuis and Ellis~\cite{DE97}.
We first introduce the reader to their abstract setting, 
and refine the large deviations result by Budhiraja and Dupuis~\cite{BD01} to our general setup.
We then show how this abstract framework applies to the small-noise stochastic Volterra system~\eqref{eq:Xep}, first proving pathwise large deviations, and then the moderate deviations counterpart.

%%%%%%%%%%%%%%%%%%%%%%%%%%%%%%%%%%%%%%%%%%%%
\subsection{Weak convergence approach: the abstract setting}\label{sec:WeakAbstract}
Given a family of Borel-measurable functions~$\{\cG^\ep\}_{\ep>0}$ from~$\cW^m$ to~$\cW^d$,
we enquire about the large deviations behaviour of the family of random variables
$ \{\cG^\ep(W)\}_{\ep>0}$ as~$\ep$ tends to zero, 
where~$W$ is a standard Brownian motion on the filtered probability space above.
For each $N>0$, the spaces of bounded deterministic and stochastic controls
\begin{align}
    &\cS_N:=\Big\{v\in L^2: \int_0^T |v_s|^2\ds \leq N \Big\} 
    \qquad\text{and}\qquad
    \cA_N:= \Big\{ v\in\cA: v\in \cS_N  \mbox{ almost surely}\Big\},
\label{eq:defAN}
\end{align}
with~$\cA$ introduced in~\eqref{eq:cA},
are equipped with the weak topology on~$L^2(\bT\times\Omega)$ such that they are closed and even compact (by Banach-Alaoglu-Bourbaki theorem).
Budhiraja and Dupuis~\cite{BD01} assume,
for any sequence~$\{v^\ep\}_{\ep>0}$ in~$\cA_N$ converging weakly to~$v\in\cA_N$, the existence of a limit in distribution of $\cG^\ep\left(W + \frac{1}{\ep} \int_0^\cdot v^\ep_s\ds \right)$ which is uniquely characterised by~$v$.
However, such uniqueness may fail when the coefficients of the system~\eqref{eq:SVEintro}
 (in particular the diffusion coefficient~$\sigma$) are not locally Lipschitz, 
 as is the case for the Feller diffusion for example
 (in this case without singular kernel, a dedicated analysis was carried out in~\cite{CDMD15,Donati04} using the Freidlin-Wentzell approach).
We relax here this uniqueness assumption by replacing the limiting trajectory by a perturbed version.

\begin{definition}\label{def:G0v}
For all~$N\in\bN$ and~$v\in\cA_N$ we define %$\cG^0_{v}$ as the limit set of $\cG^\ep(W + \ep^{-1}\int_0^\cdot v^\ep_s\ds)$ over all families $\{v^\ep\}_{\ep>0}$ in~$\cA$ converging in distribution to~$v$. That is
\begin{align*}
\cG^0_{v,N} := \bigg\{ \phi:\Omega\to\cW^d \;\Big\lvert 
\; & \text{there exist } \{\ep_n\}_{n\in\bN}\subset\bR_+ \text{ with } \lim_{n\uparrow\infty} \ep_n =0 \\
&\text{and a sequence } \{v^{\ep_n}\}_{n\in\bN}\subset\cA_N \text{  with } \lim_{n\uparrow\infty} v^{\ep_n}= v \text{ in distribution such that }\\
& \phi=\lim_{n\uparrow\infty} \cG^{\ep_n}\left(W + \ep_n^{-1}\int_0^\cdot v^{\ep_n}_s\ds \right) \text{ in distribution} \bigg\},
\end{align*}
and~$\cG^0_{v,N}$ is empty if~$v$ is not in~$\cA_N$.
For all~$v\in\cA$, we also denote~$\displaystyle\cG_v^0:= \bigcup_{N\in\bN} \cG_{v,N}^0$.
\end{definition}
%Note that if~$v$ is deterministic then $\cG_v^0$ should be as well.
Then we define the functional~$I:\cW^d\to\overline{\bR}_+$ given by
 \begin{equation}
        I(\phi):=\inf\bigg\{\half \int_0^T \abs{v_s}^2\ds: v\in  L^2 \text{ such that } \phi\in\cG^0_{v} \bigg \}.
    \label{eq:raterelax}
\end{equation}
\begin{definition}\label{def:UC}
We say that~$\phi\in\cW^d$ is {\em uniquely characterised} if there exists a sequence $\{v^n\}_{n\in\bN}\subset L^2$ such that
\begin{equation}
\cG^0_{v^n}=\{\phi\}
\qquad \text{and} \qquad 
\half \int_0^T \abs{v_s^n}^2\ds \le I(\phi) +\frac{1}{n},\quad \text{for all } n\in\bN.
\label{eq:condUB}
\end{equation}
In particular, if there exists $\widetilde v\in L^2$ which attains the infimum in~\eqref{eq:raterelax} and $\cG^0_{\widetilde v}=\{\phi\}$ then $\phi$ is uniquely characterised, because one can choose $v^n=\widetilde v$ for all~$n\in\bN$.
\end{definition}
In the display~\eqref{eq:raterelax} and  Definition~\ref{def:UC}, $\phi$, $v$, $v^n$ and~$\widetilde v$ are all deterministic.

\begin{assumption}
\label{assu:relax}
For any~$\delta>0$ and any $\phi\in\cW^d$ such that~$I(\phi)<+\infty$, there exists~$\phi^\delta$ uniquely characterised such that
$\normbT{\phi-\phi^\delta}\le \delta$ and $\abs{I(\phi)-I(\phi^\delta)}\le \delta$.
\end{assumption}
\begin{remark}
This assumption is reminiscent of~\cite[Proposition 3.3]{Donati04}, where the authors resolve the non-uniqueness issue in the diffusion case. A similar problem is also at the core of~\cite[Lemma 5.1]{BFW20} in an infinite-dimensional setting.
\end{remark}

Our abstract large deviations result is the following, extending~\cite[Theorem 4.4]{BD01},
at least when the underlying Hilbert space is~$L^2(\bT:\bR^m)$, to the non-uniqueness case.
\begin{theorem}
\label{thm:relax}
Assume that
\begin{enumerate}[(i)] 
    \item For all~$N>0$, all~$v\in\cA_N$ and all families~$\{v^\ep\}_{\ep>0}$ in~$\cA_N$ converging in distribution to~$v$ as~$\ep$ tends to zero, $\big\{\cG^\ep\left(W + \ep^{-1}\int_0^\cdot v^\ep_s\ds \right)\big\}_{\ep>0}$ is tight.
    \item The functional~$I$ defined by~\eqref{eq:raterelax} has compact level sets.
    \item Assumption~\ref{assu:relax} holds.
\end{enumerate}
Then the family~$\{\cG^\ep(W)\}_{\ep>0}$ satisfies the Laplace principle and, by equivalence, the Large Deviations Principle with rate function~$I$ and speed~$\ep^{-2}$.
\end{theorem}
\begin{remark}
Item (i) entails that, for all~$N>0$, all~$v\in\cA_N$ and all families~$\{v^\ep\}_{\ep>0}$ in~$\cA_N$ converging in distribution to~$v$, there exists a subsequence such that $\lim_n \cG^{\ep_n}\left(W + \ep_n^{-1}\int_0^\cdot v^{\ep_n}_s\ds \right)$ exists and by definition belongs to $\cG^0_{v,N}$. %The proof of this theorem does not change if one notices that every subsequence is also tight and therefore has a converging subsequence.
\end{remark}
\begin{remark}
In the large deviations literature, a rate function is sometimes called `good' if it has compact level sets. All the rate functions in the present paper satisfy this requirement (item (ii) above takes care of that) therefore we drop the adjective `good'. 
\end{remark}
We defer the proof to Appendix~\ref{app:proof:thm:relax}; 
the lower bound can be tackled as in~\cite{BD01}, and we therefore concentrate 
on the upper bound.
The idea is that the Laplace principle~\eqref{eq:Laplace} upper bound involves an infimum so deriving it only requires a $\delta$-optimal path. Hence a perturbation will also do the trick, provided one knows how to handle the control associated to it.
In~\cite[Theorem 4.4]{BD01}, unique characterisation of the limiting element in~(i) is granted, 
and the set~$\cG^0_v$ is a singleton that takes the form $\cG^0(\int_{0}^\cdot v_s \ds)$,
where they view~$\cG^0$ as a map.
In that case Assumption~\ref{assu:relax} is clearly satisfied since~$\phi^\delta$ can be taken as~$\phi$ itself.

%%%%%%%%%%%%%%%%%%%%%%%%%%%%%%%%%%%%%%%5
\subsection{Application to stochastic Volterra systems}\label{sec:Application}
We now show how the abstract setting developed above in Section~\ref{sec:WeakAbstract} 
applies to the small-noise stochastic Volterra system~\eqref{eq:Xep} and why pathwise uniqueness is so fundamental.
If {\bf H4} holds, define the functional~$\cG^\ep$ as the Borel-measurable map associating the multidimensional Brownian motion~$W$ 
to the solution of the stochastic Volterra equation~\eqref{eq:Xep}, that is: $\cG^\ep(W) = X^\ep$. 
For any control~$v\in\cA_N,\,N>0$ (introduced in~\eqref{eq:defAN}) and any~$\ep>0$,
the process
$\widetilde{W} := W + \vartheta_\ep^{-1}\int_0^\cdot v_s \ds$
is a $\widetilde\bP$-Brownian motion by Girsanov's theorem, where
\[
\frac{\D \widetilde\bP}{\D\bP} := \exp\left\{-\frac{1}{\vartheta_\ep } \sum_{i=1}^m \int_0^T v_s^{(i)} \D W_s^{(i)}
-\frac{1}{2\vartheta_\ep ^2}\int_0^T \abs{v_s}^2\ds\right\}.
\]
Hence the shifted version~$X^{\ep,v}:=\cG^\ep(\widetilde W)$ appearing in Theorem~\ref{thm:relax}(i) is the strong unique solution of~\eqref{eq:Xep} under~$\widetilde\bP$, with~$X^\ep$
and~$W$ replaced by~$X^{\ep,v}$ and~$\widetilde W$. Because~$\bP$ and~$\widetilde\bP$ are equivalent, $X^{\ep,v}$ is also the unique strong solution, under~$\bP$, of the controlled equation
\begin{equation}    \label{eq:Xepv}
    X_t^{\ep,v} = X^\ep_0 + \int_0^t K(t-s)  \Big[ b_\ep(s,X_s^{\ep,v}) + \sigma_\ep(s,X^
{\ep,v}_s)  v_s \Big]\D s
    + \vartheta_\ep  \int_0^t K(t-s)  \sigma_\ep(s,X^{\ep,v}_s)\D W_s.
\end{equation}
Under appropriate conditions, and using the notations set in~{\bf H1, H2}, 
we heuristically observe that taking~$\ep$ to zero,  
the system~\eqref{eq:Xepv} reduces to the deterministic Volterra equation
\begin{equation}
    \phi_t = x_0 + \int_0^t K(t-s) \Big[ b(s,\phi_s) + \sigma(s,\phi_s)  v_s \Big]\D s.
    \label{eq:LDPlimit}
\end{equation}
We will show later that the set~$\cG^0_v$ corresponds to the set of solutions of~\eqref{eq:LDPlimit}.

\begin{example}
\label{ex:relax}
To illustrate the need for a set~$\cG^0_v$ rather than a singleton, 
consider the Feller diffusion 
$$
X_t = x_0 + \kappa\int_{0}^{t}(\theta-X_s)\ds + \int_{0}^{t}\sqrt{X_s}\D W_s,
$$
for $t\in\bT$, with $x_0, \kappa, \theta>0$.
Letting $t\mapsto \ep t$ and denoting $X_t^\ep:=X_{\ep t}$ yields, by scaling,
$$
X^\ep_t = x_0 + \kappa\ep\int_{0}^{t}\left(\theta-X^\ep_s\right)\ds + \sqrt{\ep}\int_{0}^{t}\sqrt{X^\ep_s}\D W_s,
$$
which is exactly~\eqref{eq:Xep}
with $d=1$, $K\equiv1$ , $\vartheta_\ep=\sqrt{\ep}$ ,  $b_\ep(x) = \kappa\ep(\theta-x)$, $\sigma_\ep(x)=\sqrt{x}$.
For $v\in L^2$, taking limits as~$\ep$ tends to zero in the corresponding controlled equation~\eqref{eq:Xepv}
yields~\eqref{eq:LDPlimit}, or 
$$
\phi_t = x_0 + \int_0^t \sqrt{\phi_s} v_s\ds,\quad t\in\bT.
$$
Uniqueness of this Volterra equation does not hold in general because of the non-Lipschitz coefficient, 
and thus $\cG^0_v$ corresponds to the set of non-negative solutions.
Consider for example $x_0=1$, $\bT = [0,4]$ and the control
$$
v_t :=
\begin{cases} 
      -1, &\quad \text{if }   t\in[0,2) \\
       1, &\quad \text{if }   t\in[2,4].
   \end{cases}
$$
The function $\phi_t:= \frac{(t-2)^2}{4}$ is clearly a solution, 
but so is~$\vphi$ equal to~$\phi$ on~$[0,2]$ and null on $[2,4]$.
The square root function is indeed locally Lipschitz away from zero, 
and uniqueness can thus be guaranteed as long as the solution remains positive. 
The perturbation $\phi^\delta:=\phi+\delta t$ is now the unique solution to
\[
\phi^\delta_t = 1 + \int_0^t  \sqrt{\phi_s^\delta} v_s^\delta\ds,
\]
for all $t\in[0,4]$, where $v^\delta:= \dot\phi^\delta / \sqrt{\phi^\delta}$.
The infimum in~\eqref{eq:raterelax} is attained by $v^\delta$ and $\cG^0_{v^\delta}=\{\phi^\delta\}$, thus $\phi^\delta$ is uniquely characterised. Furthermore, \cite[Proposition 3.3]{Donati04} shows that $\phi^\delta$ satisfies Assumption~\ref{assu:relax}.
\end{example}
In~\cite{BBDW19}, the authors were also confronted to a limiting equation with multiple solutions.
Instead of perturbing the path~$\phi$, they perturb the control in a way that the resulting equation has a unique solution which is precisely~$\phi$, i.e.~$\cG^0_{v^\delta} = \{\phi\}$.
This approach may seem more natural; however, it is
not always obvious how to perturb the control ensuring uniqueness of the ODE, while our formulation
makes it more straightforward.
Before stating the main large and moderate deviations results for small-noise stochastic Volterra equations, 
we introduce the following assumption, monitoring the moments of the controlled equation:
\begin{assumption}
\label{assu:Gammabound}
Let~$X^{\ep,v}$ be the pathwise unique solution to~\eqref{eq:Xepv}. If {\bf H3a} holds then the present assumption is satisfied. 
If instead {\bf H3b} holds, then there exists $\ep_0>0$ such that, for any~$p\ge1$ and $N>0$,
    \begin{equation}
    \sup \Big\{ \bE \left[\big\lvert \Gamma\big( (X^{\ep,v}_t)^{(\Upsilon)} \big)\big\lvert^p \right] :  t\in\bT,  v\in\cA_N, \ep\in(0,\ep_0) \Big\} < \infty,
\label{eq:boundgsto}
\end{equation}
\begin{equation}
    \sup \Big\{ \big\lvert \Gamma\big( (\phi_t)^{(\Upsilon)} \big)\big\lvert : t\in\bT, v\in \cS_N, \phi\in\cG^0_v \Big\}< \infty.
\label{eq:boundgdet}
\end{equation}
\end{assumption}
\begin{remark}
In the following, {\bf H3b} will always be complemented by Assumption~\ref{assu:Gammabound}.
\end{remark}

\subsection{Large Deviations}
Armed with the abstract setting in Section~\ref{sec:WeakAbstract}, 
and its application to the stochastic Volterra system~\eqref{eq:Xep} in Section~\ref{sec:Application},
we can at last show large deviations for the latter:
\begin{theorem}[Large Deviations]
Under {\bf H1 - H4}, Assumptions \ref{assu:K}, \ref{assu:relax} and~\ref{assu:Gammabound}, the family~$\{X^\ep\}_{\ep>0}$, unique solution of~\eqref{eq:Xep}, satisfies a Large Deviations Principle with rate function~\eqref{eq:raterelax} and speed~$\vartheta_\ep^{-2}$, where~$\cG^0_v$ is the set of solutions of the limiting equation~\eqref{eq:LDPlimit}.
\label{thm:LDP}
\end{theorem}
\begin{remark}
We recall that Assumption~\ref{assu:relax} is automatically satisfied if the limiting equation~\eqref{eq:LDPlimit} has a unique solution. Also, it is only necessary to check Assumption~\ref{assu:Gammabound} if {\bf H3a} does not hold.
\end{remark}

\subsubsection{Technical preliminary results}
The proof will rely on the following results: 
Lemma~\ref{lemma:boundv} (proved in Section~\ref{app:proof:lemma:boundv}) 
shows the moment bound of the controlled process defined by~\eqref{eq:Xepv}, 
Lemma~\ref{lemma:tightLDP} (proved in Section~\ref{app:proof:lemma:tightLDP}) demonstrates the tightness and Lemma~\ref{lemma:goodnessLDP} (proved in Section~\ref{app:proof:lemma:goodnessLDP}) 
deals with the compactness of the level sets of the rate function. They will then allow the use of Theorem~\ref{thm:relax}.

%%%%%%%%%%%%%%%%%%%%%%%%%%%%%%%%%%%%%%%%%%%%%%%%%%%%%%%%%%%%%%%%%%%%%%%%%%%%%%%%%%%

\begin{lemma}[LDP Moment bound]
Under {\bf H1 - H4}, Assumptions~\ref{assu:K} and~\ref{assu:Gammabound}, for all~$p\ge 2$,~$N>0$, $v\in\cA_N$ and~$\ep>0$ small enough, there exists a constant~$\overline{c}>0$
independent of~$\ep,v,t$ 
such that
\begin{equation}
    \sup_{ t\in\bT} \bE \left[ \abs{X^{\ep,v}_t}^p \right] \le \overline{c}.
\label{eq:boundv}
\end{equation}
\label{lemma:boundv}
\end{lemma}

\begin{remark}
This bound also holds for any solution~$\phi$ of~\eqref{eq:LDPlimit} under the same assumptions, therefore there also exists~$\overline{c}_0>0$ such that~$\sup \{ \normbT{\phi}: v\in \cS_N \text{ such that } \phi\in\cG^0_v\}  \le \overline{c}_0$.
\label{rem:boundphi}
\end{remark}
%%%%%%%%%%%%%%%%%%%%%%%%%%%%%%%%%%%%%%%%%%%%%%%%%%%%%%%%%%%%%%%%%%%%%%%%%%%%%%%%%%%
The following lemma deals with~\ref{thm:LDP}(i) by showing tightness of~$X^{\ep,v^\ep}=\cG^{\ep}\left(W + \vartheta_{\ep}^{-1}\int_0^\cdot v^{\ep}_s\ds \right)$.
\begin{lemma}[LDP Tightness]
\label{lemma:tightLDP}
Consider {\bf H1 - H4}, Assumptions~\ref{assu:K} and~\ref{assu:Gammabound}.
If~$p> 2 \vee 2/\gamma$,~$N>0$ and~$\{v^\ep\}_{\ep>0}$ is a family in~$\cA_N$, then~$X^{\ep,v^\ep}$ admits a version which is H\"older continuous on~$\bT$ of any order~$\alpha < \gamma/2 -1/p$, uniformly for all~$ \ep>0$.
Denoting again this version by~$X^{\ep,v^\ep}$, one has for all~$\ep>0$ small enough
\begin{equation}
\label{eq:holderLDP}
     \bE \left[ \left( \sup_{0\le s<t\le T} \frac{\abs{X_t^{\ep,v^\ep}-X_s^{\ep,v^\ep}}}{\abs{t-s}^\alpha} \right)^p \right]
    \le \overbar C ,
\end{equation}
for all~$\alpha\in [0, \gamma/2 -1/p)$, where~$\overbar C$ is a constant independent of~$\ep,v^\ep,t$.
Moreover, the family of random variables~$\{ X^{\ep,v^{\ep}}\}_{\ep>0}$ is tight in~$\cW^d$.
\end{lemma}

\begin{remark}
\label{rem:tightphi}
This lemma entails that for all~$N>0,v\in\cS_N$, any solution to~\eqref{eq:LDPlimit} also has H\"older continuous paths of the same order.
\end{remark}

\begin{lemma}\label{lemma:G0v}
The set~$\cG_v^0$ from Definition~\ref{def:G0v} is characterised by
\[
\cG_v^0 = \left\{ \phi:\Omega\to\cW^d \,\Big\lvert\, \phi_t = x_0 + \int_0^t K(t-s) \Big[ b(s,\phi_s) + \sigma(s,\phi_s)  v_s \Big]\D s, \text{  for all  }t\in\bT\right\}.
\]
\end{lemma}

%%%%%%%%%%%%%%%%%%%%%%%%%%%%%%%%%%%%%%%%%%%%%%%%%%%%%%%%%%%%%%%%%%%
The following lemma proves Theorem~\ref{thm:relax}(ii) and its proof can be found in Appendix~\ref{app:proof:lemma:goodnessLDP}.
\begin{lemma}[LDP Compactness]
\label{lemma:goodnessLDP}
Under {\bf H2, H3}, Assumptions~\ref{assu:K} and~\ref{assu:Gammabound}, the functional~$I$ in~\eqref{eq:raterelax} has compact level sets.
\end{lemma}

%%%%%%%%%%%%%%%%%%%%%%%%%%%%%%%%%%%%%%%%%%%%%%%%%%%%%%%%%%%%%%%%%%%%%%
Leveraging on the above lemmas, the Large Deviations Principle (Theorem~\ref{thm:LDP}) is a direct consequence of Theorem~\ref{thm:relax}. 

\subsubsection{Proof of Lemma~\ref{lemma:G0v}}
%\forest{Fix~$N>0$. Consider a family~$\{v^\ep\}_{\ep>0}$ in~$\cA_N$ converging in distribution to~$v\in\cA_N$.
%\forest{We take an arbitrary subsequence~$\{v^{\ep_n}\}_{n\in\bN}$ and prove convergence along a subsequence thereof.} For~$ \ep>0$ small enough, the SVE~\eqref{eq:Xep} is pathwise unique by {\bf H4} and we showed that its controlled counterpart~\eqref{eq:Xepv} also has a unique strong solution~$X^{\ep,v^\ep}$.Lemma~\ref{lemma:tightLDP} shows that the family~$\{ X^{\ep_n,v^{\ep_n}} \}_{n\ge0}$ is tight in~$\cW^d$, as required by condition (i) of Theorem~\ref{thm:relax}.}
%\forest{Moreover, the trajectories of~$v^{\ep_n}$ belong to a compact space with respect to the weak topology so the family of controls is tight as a sequence of~$\cS_N$-valued random variables. Since these are both Polish spaces, the family~$\big\{X^{\ep_n,v^{\ep_n}},v^{\ep_n} \big\}_{n\ge0}$ is tight in~$\cW^d\times \cS_N$.}

%Hence there exists a subsequence, denoted hereafter~$\big\{ X^n,v^n \big\}_{n\ge0}$, that converges weakly to a~$\cW^d\times \cS_N$-valued random variable~$(X^0,v)$ defined on a possibly different probability space~$(\Omega^0,\cF^0,\bP^0)$ as~$n$ tends to~$+\infty$. 

For~$N\in\bN$ and~$v\in\cA_N$, we first need to identify~$\cG_{v,N}^0$, defined in~\eqref{def:G0v}. Consider a subsequence~$\{\ep_n\}_{n\in\bN}\subset\bR_+$ with ~$\lim_{n\uparrow\infty} \ep_n =0$
and a sequence~$\{v^{\ep_n}\}_{n\in\bN}\subset\cA_N$  such that~$\lim_{n\uparrow\infty} v^{\ep_n}= v$ in distribution, and assume that~$X^n:=X^{\ep_n,v^{\ep_n}}$ converges in distribution to some random variable~$\phi$ with values in~$\cW^d$. We also denote~$v^n,\ep_n,X^n_0, b_n, \sigma_n$ along this subsequence.

By Skorohod representation theorem we can work with almost sure convergence for the purpose of identifying the limit. Hence~$\big\{ X^n,v^n \big\}_{n\ge0}$ converges almost surely in the product topology on~$\cW^d\times \cS_N$, and the limit is the~$\cW^d\times \cS_N$-valued random variable~$(\phi,v)$.
%Or just apply the as convergence, or $L^2$ convergence (by uniform integrability) would be much quicker. Strategy: show that\[X^n_t - (x_0 + \int_0^t \big(b(\phi_s)+\sigma(\phi_s)v_s\big)\ds\]goes to zero. Since~$X^n\to\phi$ then $\phi$ satisfies this equation.
The convergence of the couple also takes place in distribution, so that we can follow the technique in~\cite{Chiarini14} to identify the limit.
For~$t\in\bT$, define~$\Phi_t:\cS_N \times\cW^d \to \bR$ as
\[
\Phi_t(f,\omega):=\abs{\omega_t - x_0 - \int_0^t K(t-s) \Big[b(s,\omega_s) + \sigma(s,\omega_s)f_s\Big] \D s}\wedge 1.
\]
Clearly,~$\Phi_t$ is bounded and we show that it is also continuous. Indeed, let~$\omega^n \to \omega$ in~$\cW^d$ and~$f^n\to f$ in~$\cS_N$ with respect to the weak topology. {\bf H2} implies the existence of continuous moduli of continuity~$\rho_b$ and~$\rho_{\sigma}$ for both coefficients on compact subsets
(see Definition~\ref{def:modulusG}). 
Since the paths~$\omega^n,n\ge1$ and $\omega$ are continuous, they are also uniformly bounded and hence these moduli are available.
Then, using Cauchy-Schwarz inequality and the fact that $\abs{x\wedge1 -y\wedge1} \le \abs{x-y}$ for all $x,y>0$, 
\begin{align*}
    \abs{\Phi_t(f,\omega) - \Phi_t(f^n,\omega^n)}
    & \le \abs{\omega_t-\omega^n_t} + \int_0^t \abs{K(t-s)}\abs{ b(s,\omega_s)-b(s,\omega_s^n)} \D s \\
    & \quad + \int_0^t \abs{K(t-s)} \abs{ \big( \sigma(s,\omega_s)-\sigma(s,\omega_s^n) \big) f^n_s 
    + \sigma(s,\omega_s) \big( f_s-f^n_s \big)
    } \D s \\
    & \le \normbT{\omega-\omega^n} + \normbT{\rho_b(\abs{\omega-\omega^n}) } \norm{K}_1
    + \normbT{\rho_{\sigma}(\abs{\omega-\omega^n})} \norm{K}_{2}  \norm{f^n}_{2} \\
    & \quad + \normbT{\sigma(\cdot,\omega)} \int_0^t \abs{K(t-s)} \abs{f_s-f^n_s}\ds .
\end{align*}
Since $K(t-\cdot)\in L^2$ and~$f_n$ tends to~$f$ weakly in~$L^2$ 
then the last integral converges to zero as~$n$ goes to infinity.
Moreover~$\lim_{n\uparrow\infty} \normbT{\omega-\omega^n}=0$,~$\norm{f^n}_2\le \sqrt N$ for all~$n\ge0$ and~$ \norm{K}_2 + \normbT{\sigma(\cdot,\omega)} < \infty$, which proves that~$\Phi_t$ is continuous, and therefore 
\[
\lim_{n\uparrow\infty} \bE\left[\Phi_t(v^n,X^n)\right] = \bE\left[\Phi_t(v,\phi)\right].
\]
We now prove that the left-hand side is actually equal to zero.
We start with the observation that, using BDG inequality,
\begin{align}
\label{eq:Phi}
    \bE\left[\Phi_t(v^n,X^n)\right]
    \le & \abs{X^n_0-x_0} + \int_0^t \abs{K(t-s)} \bE \big[ \abs{b_n(s,X^n_s)-b(s,X^n_s)} \big] \ds \nonumber \\
    & + \int_0^t \abs{K(t-s)} \bE\left[ \abs{\sigma_n(s,X^n_s)-\sigma(s,X^n_s)} \abs{v^n_s}\right] \ds  \nonumber \\
    & + \vartheta_{\ep_n}\bE \left[\int_0^t \abs{K(t-s) \sigma_n(s,X^n_s)}^2 \ds\right]^\half.
\end{align}
The bounds~\eqref{eq:Holdertrick} and~\eqref{eq:ineqg} show how to control the last term under {\bf H3a} and {\bf H3b} respectively, hence there exists $C_1>0$ independent of $t$ and $n$ such that
$\bE \left[  \int_0^t \abs{K(t-s) \sigma_n(s,X^n_s)}^2 \ds\right]^\half\le C_1$.

However the convergence of~$b_n,\sigma_n$ only occurs on compact subsets so we use a localisation argument. For all~$n\ge0$,~$M>0$ we introduce
\begin{equation}
A^M_n := \Big\{\omega\in\Omega: \normbT{X^n(\omega)} > M \Big\}.
\label{eq:AnMdef}
\end{equation}
The uniform (in~$n\in\bN$) H\"older regularity of~$X^n$, encompassed by~\eqref{eq:holderLDP}, entails the existence, for all~$p>2\vee 2/\gamma$, of~$C_2(p),C_3(p)>0$ independent of~$n$ such that
\[
\bE \left[\sup_{t\in\bT} \abs{X^n_t}^p \right] \le C_2(p) \big(\abs{X^n_0}^p + T^{p\alpha}\big) \le C_3(p),
\]
for some~$0<\alpha < \gamma/2 - 1/p$ and where~$X^n_0$ is uniformly bounded by~$2\abs{x_0}$ for~$n$ large enough. 
Markov's inequality then implies that
\begin{equation}
\lim_{M\uparrow\infty} \sup_{n\in\bN} \bP\big(A^M_n \big) 
\le \lim_{M\uparrow\infty} \sup_{n\in\bN} \frac{C_3(p)}{M^p}
=0.
\label{eq:AnMlimit}
\end{equation}
Moreover, for all~$n\in\bN$,~$\omega\in \Omega \setminus A^M_n$, and~$t\in\bT$,~$\abs{X^n_t (\omega)}$ is bounded by~$M$, which means
\[
\abs{b_n(t,X^n_t(\omega))-b(t,X^n_t(\omega))}
\le \norm{b_n(t,\cdot)-b(t,\cdot)}_M,
\]
which tends to zero uniformly on~$\bT$ as~$n$ goes to infinity (and likewise for~$\sigma_n$) from {\bf H2}.
Define now
\begin{align*}
I_n := & \int_0^t \abs{K(t-s)}  \Big( \abs{b_n(s,X^n_s)-b(s,X^n_s)} + \abs{\sigma_n(s,X^n_s)-\sigma(s,X^n_s)} \abs{v^n_s} \Big) \ds
\end{align*}
and observe that, using Jensen and Cauchy-Schwarz inequalities, the growth condition on the coefficients from {\bf H3} and the moment bounds on $X^n$ from~\eqref{eq:boundv}, there exists $C_4>0$ independent of~$n$ such that
\begin{align}
\bE \big[ \abs{I_n}^2 \big]
& \le 2 t \int_0^t \abs{K(t-s)}^2   \bE \big[\abs{b_n(s,X^n_s)-b(s,X^n_s)}^2 \big] \ds \nonumber \\
& \quad + 2N \int_0^t \abs{K(t-s)}^2 \bE \left[\abs{\sigma_n(s,X^n_s)-\sigma(s,X^n_s)}^2 \right] \ds \nonumber \\
& \le 2 \norm{K}_{2}^2   \sup_{s\le t} \bigg\{ t \bE \big[\abs{b_n(s,X^n_s)-b(s,X^n_s)}^2 \big] +  N \bE \left[\abs{\sigma_n(s,X^n_s)-\sigma(s,X^n_s)}^2 \right] \bigg\}\nonumber \\
& \le C_4.
\label{eq:C4}
\end{align}
Let us fix~$\epsilon>0$ and choose~$M_\epsilon>0$ large enough such that~$\sup_{n\in\bN} \bP\big(A^{M_\epsilon}_n \big) \le \epsilon^2  / C_4$; this choice is possible because of~\eqref{eq:AnMlimit}. Therefore, using the bound~\eqref{eq:C4} and Cauchy-Schwarz inequality to separate $I_n$ and~$\one_{A_n^{M_\ep}}$, one obtains
\begin{align*}
    \limsup_{n\uparrow\infty} \bE [I_n] 
    & = \limsup_{n\uparrow\infty} \bE \Big[ I_n \big( \one_{A_n^{M_\ep}} + \one_{\Omega \setminus A_n^{M_\ep}} \big) \Big] \\
    & \le \limsup_{n\uparrow\infty} \Big\{ \sqrt{ C_4 \bP\big (A^{M_\ep}_n \big)} + \norm{K}_1 \normbT{\norm{b_n-b}_{M_\ep}} + \sqrt N \norm{K}_2 \normbT{\norm{\sigma_n-\sigma}_{M_\ep}} \Big\}
    \le \epsilon.
\end{align*}
It follows from~\eqref{eq:Phi} that
\begin{align*}
    \lim_{n\uparrow\infty} \bE\big[\Phi_t(v^n,X^n)\big]
    & \le \lim_{n\uparrow\infty} \Big\{ \abs{X^n_0-x_0} + \bE [I_n] + \vartheta_{\ep_n} C_1 \Big\}
    \le \epsilon,
\end{align*}
hence $\lim_{n\uparrow\infty} \bE\left[\Phi_t(v^n,X^n)\right]=0$ since~$\epsilon>0$ was chosen arbitrarily.
The equality~$\bE\left[\Phi_t(v,\phi)\right]=0$ implies that~$\phi$ satisfies~\eqref{eq:LDPlimit} almost surely, for all~$t\in\bT$. Since~$\phi$ has continuous paths, it satisfies~\eqref{eq:LDPlimit} for all~$t\in\bT$, almost surely, which means $\cG^0_{v,N}$ consists of all the solutions of~\eqref{eq:LDPlimit}. Since this definition is independent of~$N$, it extends to~$\cG_v^0$, which yields the claim. \qed

%%%%%%%%%%%%%%%%%%%%%%%%%%%%%%%%%%%%%%%%%%%%%%%%%%%%%%%%%%%%%%%%%%%%%%%%%%%%%%%%%%%%%%%%%%%%%%%%%%%%%%%%%%%%%%%%%%%%%%%%%%%%%%%%%%%%%%%%%%%%%%%%%%%%%%%%%%%%%%%%%%%%%%%%

\subsection{Moderate Deviations}
Let~$h_\ep$ tend to infinity such that~$\vartheta_\ep h_\ep$ tends to zero as~$\ep$ goes to zero and define~$\overbar{X}$ to be the limit in law of $X^\ep$, which we identified in the previous subsection as a solution of the Volterra equation
\begin{equation}
\overbar{X}_t = x_0 + \int_0^t K(t-s) b(s,\overbar{X}_s) \D s.
\label{eq:Xbar}
\end{equation}
Then the MDP for~$\{X^\ep\}_{\ep>0}$ is equivalent to the LDP for the family~$\{\eta^\ep\}_{\ep>0}$ defined as
\[
\eta^\ep := \frac{X^\ep-\overbar{X}}{\vartheta_\ep h_\ep} = \frac{\cG^\ep(W)-\overbar{X}}{\vartheta_\ep h_\ep}=: \cT^\ep(W) ,
\]
where~$\cT^\ep:\cW^m \to \cW^d$ is a Borel-measurable map for each~$\ep>0$.
Therefore $\eta^{\ep}$ satisfies the following SVE for all $\ep>0$, and is its unique solution if {\bf H4} holds.
\begin{equation}
    \eta^{\ep}_t = \frac{X^\ep_0-x_0}{\vartheta_\ep h_\ep} 
    + \int_0^t K(t-s) \frac{ b_\ep \big(s,\overbar{X}_s + \vartheta_\ep h_\ep   \eta^{\ep}_s \big)-b(s,\overbar{X}_s)}{\vartheta_\ep h_\ep  } \D s 
     + \int_0^t K(t-s)  \frac{\sigma_\ep \big(s,\overbar{X}_s + \vartheta_\ep h_\ep   \eta^{\ep}_s \big)}{h_\ep  } \D W_s.
    \label{eq:etaep}
\end{equation}
Similarly to the LDP case we are interested in a certain shift of the driving Brownian motion, controlled by~$v\in\cA$. For all~$\ep>0$, let
\begin{equation}
\label{eq:etaepvdef}
    \eta^{\ep,v} :=  \cT^\ep \left(W + h_\ep   \int_0^\cdot v_s\ds \right)
    =  \frac{\cG^\ep\big( W + h_\ep   \int_0^\cdot v_s \ds \big)-\overbar{X}}{\vartheta_\ep h_\ep  }.
\end{equation}
For convenience we introduce the sequence~$\{\Theta^{\ep,v}\}_{\ep>0}$ defined, for all~$\ep>0$, $v\in\cA$, $t\in\bT$, by
\begin{align*}
    \Theta^{\ep,v}_t 
    :=& \, \cG^\ep\left( W + h_\ep   \int_0^\cdot v_s \ds \right)(t) \\
    =& \, X^\ep_0 + \int_0^t K(t-s)  \Big[ b_\ep(s,\Theta_s^{\ep,v}) + \vartheta_\ep  h_\ep   \sigma_\ep(s,\Theta^{\ep,v}_s)  v_s \Big]\D s 
    + \vartheta_\ep  \int_0^t K(t-s)  \sigma_\ep(s,\Theta^{\ep,v}_s)\D W_s.
\end{align*}
This sequence satisfies the bound~\eqref{eq:boundv} and converges weakly towards~$\overbar{X}$ since~$\vartheta_\ep  h_\ep   \sigma_\ep$ tends to zero as~$\ep$ tends to zero. Finally the process defined by~\eqref{eq:etaepvdef} satisfies
\begin{align}
\label{eq:etaepv}
    \eta^{\ep,v}_t =  \frac{X^\ep_0-x_0}{\vartheta_\ep h_\ep  } 
    & + \int_0^t K(t-s) \frac{ b_\ep \big(s,\Theta^{\ep,v}_s \big)-b(s,\overbar{X}_s)}{\vartheta_\ep h_\ep  } \D s \nonumber \\
    & + \int_0^t K(t-s)   \sigma_\ep \big(s,\Theta^{\ep,v}_s \big)  v_s \D s
    + \frac{1}{h_\ep  } \int_0^t K(t-s)  \sigma_\ep \big(s,\Theta^{\ep,v}_s \big) \D W_s.
\end{align}
For all $v\in\cA$, we define $\cT^0_v$ to be the solution of the limiting equation
\begin{equation}
    \psi_t
    = \int_0^t K(t-s) \Big[ \nabla b(s,\overbar{X}_s) \psi_s + \sigma(s,\overbar{X}_s) v_s
    \Big] \ds.
\label{eq:MDPlimit}
\end{equation}
The form of the limit equation is dramatically simpler than for the LDP and much easier to compute. Moreover $\cT^0_v$ is well defined because the linearity of the equation and Assumption~\ref{assu:K} grant uniqueness for free, provided~$\nabla b$ exists.
Hence we will need the following assumptions:
\begin{enumerate}[{\bfseries H1.}]
\setcounter{enumi}{4}
    \item For each $t\in\bT$, the function~$b(t,\cdot)$ is continuously differentiable and~$b$ is Lipschitz continuous.
    \item There exists~$\delta>0$ such that~$\sigma(t,\cdot)$ is locally~$\delta$-H\"older continuous, uniformly for all~$t\in\bT$.
    \item $\lim_{\ep\downarrow0} \big(\vartheta_\ep h_\ep  \big)^{-1} \abs{X^\ep_0-x_0} =0$. 
    \item There exist $\ep_0>0$, a sequence $\{\nu_\ep\}_{\ep>0}$ with $\lim_{\ep\downarrow 0}\nu_\ep(\vartheta_\ep h_\ep  )^{-1}=0$ and a function~$\Xi:\bR^d\to\bR$
    such that $\abs{b_\ep(t,x)-b(t,x)} \le \nu_\ep  \Xi(x)$ for all~$t\in\bT$, $\ep\in(0,\ep_0)$, where for all~$p\ge1$, $N>0$, 
    \begin{equation}
    \label{eq:Xibound}
    \sup \bigg\{ \bE\Big[ \abs{ \Xi \big(\Theta^{\ep,v}_t \big) }^p \Big], \ep\in(0,\ep_0), v\in\cA_N,  t\in\bT \bigg\} < \infty.
    \end{equation}
\end{enumerate}
\begin{remark}
{\bf H5} entails that~\eqref{eq:Xbar} has a unique solution and yields the bound $\normbT{\nabla b(\cdot,\overbar X)}<\infty$ by continuity. {\bf H7} implies {\bf H1}. We have already proved in Lemma~\ref{eq:boundv} that the moments of all orders of~$\Theta^{\ep,v}$ are bounded  hence~\eqref{eq:Xibound} is automatically satisfied if~$\Xi$ is of polynomial growth. This is however not sufficient for the applications we have in mind where~$\Xi$ is of exponential growth. 
\label{remark:assumdp}
\end{remark} 
The main theorem of this section is the following.
\begin{theorem}[Moderate Deviations]
Under {\bf H2 - H8}, Assumptions~\ref{assu:K} and~\ref{assu:Gammabound}, the family~$\{\eta^\ep\}_{\ep>0}$ satisfies a Large Deviations Principle (equivalently~$\{X^\ep\}_{\ep>0}$ satisfies a Moderate Deviations Principle) with speed~$h_\ep^2$ and rate function
\begin{equation}
\Lambda(\psi):= \inf\bigg\{ \half\int_0^T \abs{v_t}^2\dt: v\in  L^2, \psi=\cT^0_v \bigg\},
\label{eq:grfMDP}
\end{equation}
and~$\Lambda(\psi)=+\infty$ if this set is empty.
\label{thm:MDP}
\end{theorem}
The proof of the moderate deviations theorem follows a similar structure to that of Theorem~\ref{thm:LDP},
making use of Theorem~\ref{thm:relax}.
It will rely on moment bounds in Lemma~\ref{lemma:boundeta} 
(proved in Section~\ref{app:proof:lemma:boundeta}), 
tightness in Lemma~\ref{lemma:tightMDP} (proved in Section~\ref{app:proof:lemma:tightMDP}),
weak convergence in Lemma~\ref{lemma:weakcvgMDP} 
(proved in Section~\ref{app:proof:lemma:weakcvgMDP}),
and finally compactness of the level sets in Lemma~\ref{lemma:goodnessMDP}.

%%%%%%%%%%%%%%%%%%%%%%%%%%%%%%%%%%%%%%%%%%%%%%%%%%%%%%%%%%%%%%%%%%%
\begin{lemma}[MDP Moment bound]
\label{lemma:boundeta}
Under {\bf H2 - H5, H7, H8}, Assumptions~\ref{assu:K} and~\ref{assu:Gammabound}, for all~$p\ge 2$, $N>0$, $v\in\cA_N$ and~$\ep>0$ small enough, there exists~$\widehat{c}>0$ independent of~$\ep,v,t$ such that
\begin{equation}
     \sup_{t\in\bT} \bE \left[ \abs{\eta^{\ep,v}_t}^p \right] \le \widehat{c}.
\label{eq:boundveta}
\end{equation}
\end{lemma}

%%%%%%%%%%%%%%%%%%%%%%%%%%%%%%%%%%%%%%%%%%%%%%

%%%%%%%%%%%%%%%%%%%%%%%%%%%%%%%%%%%

\begin{lemma}[MDP tightness]
\label{lemma:tightMDP}
Let~$p> 2 \vee 2/\gamma$,~$N>0$ and a family~$\{v^\ep\}_{\ep>0}$ in~$\cA_N$.
Under {\bf H2 - H5, H7, H8}, Assumptions~\ref{assu:K} and~\ref{assu:Gammabound}, $\eta^{\ep,v^\ep}$ admits a version which is H\"older continuous on~$\bT$ of any order~$\alpha < \gamma/2 -1/p$, uniformly for all~$\ep>0$.
Denoting again this version by~$\eta^{\ep,v^\ep}$, one has for all~$\ep>0$ small enough,
\begin{equation}
\label{eq:holderMDP}
    \bE \left[ \left( \sup_{0\le s<t\le T} \frac{\big\lvert \eta_t^{\ep,v^\ep}-\eta_s^{\ep,v^\ep}\big\lvert}{\abs{t-s}^\alpha} \right)^p \right]
    \le \widehat C ,
\end{equation}
for all~$\alpha\in [0, \gamma/2 -1/p)$, where~$\widehat C$ is a constant independent of~$\ep,v^\ep,s,t$.
Moreover, the family~$\{ \eta^{\ep,v^{\ep}} \}_{\ep>0}$ is tight in $\cW^d$.
\end{lemma}
We recall that~$\eta^{\ep,v^\ep} = \cT^\ep \big(W+ h_\ep   \int_0^\cdot v_s\ds\big)$ and~$\psi=\cT^0_v$, hence the lemma above deals with Theorem~\ref{thm:relax} {\em (i)}. The following one identifies the limit set as the unique solution to~\eqref{eq:MDPlimit}. It is thus more precise than in the LDP case, and justifies the form of the rate function~\eqref{eq:grfMDP}.
\begin{lemma}[MDP weak convergence]
\label{lemma:weakcvgMDP}
Let~$N>0$, a family~$\{v^\ep\}_{\ep>0}$ such that, for all~$\ep>0$,~$v^\ep\in\cA_N$ and~$v^\ep$ converges in distribution to~$v\in\cA_N$, and~$\psi$ the unique solution of~\eqref{eq:MDPlimit}. 
Under {\bf H2 - H8}, Assumptions~\ref{assu:K} and~\ref{assu:Gammabound},~$\eta^{\ep,v^{\ep}}$ converges in distribution to~$\psi$ as~$\ep$ goes to zero. 
\end{lemma}

 Item {\em (ii)} is dealt with in the following lemma.

\begin{lemma}[MDP compactness]
\label{lemma:goodnessMDP}
Under {\bf H2, H3, H5}, Assumptions~\ref{assu:K} and~\ref{assu:Gammabound}, the functional~$\Lambda$ defined by~\eqref{eq:grfMDP} has compact level sets.
\end{lemma}
\begin{proof}
Noticing that~$\nabla b(t,\overbar X_t)$ and~$\sigma(t,\overbar X_t)$ are uniformly bounded on~$\bT$ by continuity, this lemma boils down to a particular case of Lemma~\ref{lemma:goodnessLDP}.
\end{proof}

Theorem~\ref{thm:relax}(iii) is immediate by uniqueness of~\eqref{eq:MDPlimit}, therefore all the conditions are met and Theorem~\ref{thm:MDP} follows as a direct application of Theorem~\ref{thm:relax}.

%%%%%%%%%%%%%%%%%%%%%%%%%%%%%%%%%%%%%%%%%%%%%%%%%%

\subsection{Extension to non-convolution kernels}\label{sec:NonConv}
The analysis undertaken in this paper is based, both for notational convenience and with a view towards application, on convolution kernels. 
Different assumptions were studied in the literature, in particular Decreusefond~\cite{Decreusefond02} considered the properties of the map $f\mapsto \int_0^\cdot K(\cdot,s)f(s)\ds$ in order to include the fractional Brownian motion in his setting.

\subsubsection{Setting}
We call a kernel a map $K:\bT^2\to\bR$
for which both $\int_0^t K(t,s)^2 \D s$ and $K(t,s)$
are finite for all $t\in\bT$ and $s\neq t$.
The associated space is defined as
\begin{align*}
\cK:=\Big\{ u:\bT\to\bR, \,  \{\cF_t\}\text{-progressively measurable, such that }\bE\int_0^t \big[K(t,s)u(s)\big]^2 \D s <\infty, \text{ for all }t\in\bT \Big\}.
\end{align*}
Hence, for all $u\in\cK$ the stochastic integral 
\[
\widetilde{M}^K_t(u):=\int_0^t K(t,s)u(s)\D W_s
\]
is well defined for all $t\in\bT$ in the It\^o sense.
For any $\alpha\in(0,1)$, we denote the Riemann-Liouville integral~$\II^\alpha$ 
and derivative~$\DD^\alpha$  as
\begin{equation}
        (\II^\alpha f)(t):=\frac{1}{\Gamma(\alpha)} \int_0^t (t-s)^{\alpha-1} f(s)\D s, \qquad
        (\DD^{\alpha}f)(t):=\frac{\D}{\D t}(\II^{1-\alpha}f)(t),
        \quad \text{for } f\in L^1, t\in\bT.
    \label{eq:RiemmanLiouville}
\end{equation}
Define $\cI_{\alpha,p} :=\II^\alpha(L^p)$ equipped with the norm 
    $\norm{f}_{\cI_{\alpha,p}}:=\norm{D^{\alpha}f}_{L^p}$. 
    If $\alpha>\frac{1}{p}$, then $\cI_{\alpha,p}\subset \cC_0^{\alpha-\frac{1}{p}}$, the space of $(\alpha-\frac{1}{p})$-H\"older continuous functions null at time $0$.
Let~$\mathrm{K}$ denote the linear map associated to $K(t,s)$ by
    \begin{equation}
    \mathrm{K}f(t):=\int_0^t K(t,s)f(s)\D s,
    \label{eq:mapK}
    \end{equation}
and introduce, for $x\in(0,\infty)\setminus\{2\}$,
\begin{equation}\label{eq:theta}
    \theta(x):= \frac{2x}{2-x}.
\end{equation}

Given the space inclusions above, the following assumption implies precise
H\"older regularity for the integral~\eqref{eq:mapK}:
\begin{assumption}
There exist $\chi\in(1,2)$ and $\gamma>1/\theta(\chi)$ for which~$\mathrm{K}$ is continuous from~$L^2$ 
to~$\cI_{\gamma+\half,2}$ and from~$L^\chi$ to~$\cI_{\gamma,\theta(\chi)}$.
\label{ass:dec}
\end{assumption}

\begin{example}
The operators associated to the following kernels satisfy Assumption~\ref{ass:dec}:
\begin{enumerate}[$\bullet$]
    \item The Riemann-Liouville kernel
    \[
    J_H(t,s)=\frac{(t-s)^{H-\half}_+}{\Gamma(H+\half)}, \qquad \text{with }H\in (0,1),
    \]
    satisfies this assumption with $\gamma=H$ and any $\chi<2$~\cite[Theorem 4.1]{Decreusefond02}.
    \item The fractional Brownian motion kernel
    \[
    K_H(t,s) = \frac{(t-s)^{H-\half}_+}{\Gamma(H+\half)} F\left(H-\half, \half - H, H+\half, 1- \frac{t}{s}\right),
    \]
    where $F$ is the Gauss hypergeometric function, also satisfies this assumption with the same parameters as above~\cite[Theorem 4.2]{Decreusefond02}.
\end{enumerate}
\label{ex:FracKernels}
\end{example}
%%%%%%%%%%%%%%%%%%%%%%%%%%%%%%%%%%%%%%%%%%%%%%%%%%%%%%%%

%%%%%%%%%%%%%%%%%%%%%%%%%%%%%%%%%%%%%%%%%%%%%%%%%%%%%%%%%%%%%%%%%%%%%%%%%%%%%
Decreusefond's main result yields the H\"older regularity of the stochastic Volterra integral~\cite[Theorem 3.1]{Decreusefond02}:
\begin{theorem}
Let Assumption~\ref{ass:dec} hold and $u\in \cK \cap L^{\theta(\chi)}(\Omega\times\bT)$.
Then $\widetilde{M}^K(u)$ has a measurable version $M^K(u)$ which is $\alpha$-H\"older continuous for all $\alpha<\gamma-1/\theta(\chi)$.
\label{th:dec}
\end{theorem}
From now on, we only consider the measurable version of the stochastic integral. 
Although this theorem was proved in a one-dimensional setting, it also covers multi-dimensional stochastic Volterra integrals by considering their components individually and summing them.

%%%%%%%%%%%%%%%%%%%%%%%%%%%%%%%%%%%%%%%%%%%
\subsubsection{Large and moderate deviations}
For each $\ep>0$ consider the stochastic Volterra equation
\begin{equation}
    X^\ep_t = X_0^\ep + \int_0^t K(t,s) b_\ep(s,X_s^\ep) \ds + \vartheta_\ep \int_0^t K(t,s) \sigma_\ep(s,X^\ep_s) \D W_s, \quad t\in\bT,
    \label{eq:SVEdec}
\end{equation}
which was studied in~\cite{CD00} without the $\ep$-dependence, and where the coefficients live in the same spaces as those from~\eqref{eq:Xep}. To complete the non-convolution setup we also need the following condition.
\begin{assumption}
\label{assu:kernelLr}
There exists $q>2$ such that 
\begin{equation}
\sup_{t\in\bT} \bigg\{ \int_0^t \abs{K(t,s)}^{-\theta(q)} \ds \bigg\} <\infty,
\label{eq:kernelLr}
\end{equation}
with~$\theta(q)$ introduced in~\eqref{eq:theta}.
\end{assumption}
Let $p>q>2$, then $2<-\theta(p)<-\theta(q)$ and H\"older's and Jensen's inequalities yield
\begin{align*}
 \left[ \int_0^t \abs{ K(t,s) f(s) }^2 \ds \right]^\frac{p}{2} 
   & \le \left[ \int_0^t \abs{K(t,s)}^{-\theta(p)} \ds \right]^\frac{p-2}{2}
    \int_0^t \abs{f(s)}^{p} \ds \\
   & \le t^\frac{p-q}{q} \left[\int_0^t \abs{K(t,s)}^{-\theta(q)} \ds \right]^\frac{p}{-\theta(q)}
    \int_0^t \abs{f(s)}^{p} \ds.
\end{align*}
This replaces the Gronwall-type inequality derived for convolution kernels in Lemma~\ref{lemma:gronwallK}.
Hence replacing Assumption~\ref{assu:K} by the condition~\eqref{eq:kernelLr} one recovers the moments bounds of Lemmata~\ref{lemma:boundv} and~\ref{lemma:boundeta} for the processes $\{X^{\ep,v}\}_{\ep>0}$ and~$\{\eta^{\ep,v}\}_{\ep>0}$ and for any $p\ge1$.
Setting in particular $p=\theta(\chi)$ from Theorem~\ref{th:dec}, then for any $v\in\cA_N,\,N>0$ and $\xi\in\{b_\ep,\sigma_\ep\}$ we have $\xi(X^{\ep,v})\in \cK \cap L^p(\Omega\times\bT)$ thanks to the growth conditions \textbf{H3}. Therefore Assumption~\ref{ass:dec}, Theorem~\ref{th:dec} and Assumption~\ref{assu:kernelLr} yield the almost sure H\"older regularity of the following processes defined on~$\bT$:
\[
\int_0^\cdot K(\cdot,s) \xi(X^{\ep,v}_s) \ds \qquad \text{and} \qquad \int_0^\cdot K(\cdot,s) \xi(X^{\ep,v}_s)\, \D W_s.
\]
Hence we recover the tightness of Lemmata~\ref{lemma:tightLDP} and~\ref{lemma:tightMDP} under this new set of assumptions. Notice that we can consider kernels consisting of both convolution and non-convolution components.
Finally we can extend the LDP and MDP results without further modifications:
\begin{theorem}
\label{thm:nonconv}
The conclusions of Theorems~\ref{thm:LDP} and~\ref{thm:MDP} stand for~$\{X^\ep\}_{\ep>0}$ defined by~\eqref{eq:SVEdec} when each component of the kernel~$K$ satisfies either Assumption~\ref{assu:K} or Assumptions~\ref{ass:dec} and~\ref{assu:kernelLr}.
\end{theorem}

%%%%%%%%%%%%%%%%%%%%%%%%%%%%%%%%%%%%%%%%%%%%%%%%%%%%%%%%%%%%%%%%%%%%%%%%%%%%%%%%%%%%%%%%%%%%%%%%%%%%%%%%%%%%%%%%%%%%%%%%%%%%%%%%%%%%%%%%%%%%%%%%%%%%%%%%%%%%%%%%%%%%%%%
%%%%%%%%%%%%%%%%%%%%%%%%%%%%%%%%%%%%%%%%%%%%%%%%%%%%%%%%%%%%%%%%%%%%%%%%%%%%%%%%%%%%%%

\section{Application to rough volatility}
\label{section:rvol}
We now show how our results (Theorems~\ref{thm:LDP},~\ref{thm:MDP} and~\ref{thm:nonconv}) 
apply to a large class of models recently developed in mathematical finance.
Originally proposed by Comte and Renault~\cite{Comte98} with financial econometrics applications in mind, 
rough volatility models were rediscovered later in the context of option pricing in~\cite{Alos07, BFG16, Fukasawa11, GJR14}, 
developed and extended widely,
and have now become the new standards of volatility modelling.
They usually take the following form:
\begin{equation}
\left\{
\begin{array}{rl}
    X_t &= \displaystyle - \half\int_0^t \Sigma(Y_s) \ds + \int_0^t \sqrt{\Sigma(Y_s)} \D B_s, \\
    Y_t &= \displaystyle y_0 + \int_0^t K_1(t-s) \bb(Y_s)\ds + \int_0^t K_2(t-s)\zeta(Y_s) \D W_s,
\end{array}
\right.
\label{eq:rvolmodel}
\end{equation}
where both~$X$ and~$Y$ are one-dimensional, $K_1,K_2\in L^2(\bT:\bR_+)$
and~$B$ and~$W$ are two standard Brownian motions with $\D\langle B,W\rangle_t =\rho \dt$, 
for some correlation parameter $\rho\in (-1,1)$.
We further define $\rrho:=\sqrt{1-\rho^2}$, and set $X_0=0$ without loss of generality.
Here~$X$ denotes the logarithm of a stock price process, and~$\sqrt{\Sigma(Y)}$ its instantaneous volatility.
We adopt a slight abuse of notation, as~$X$ previously denoted the multidimensional system,
but writing now~$X$ as the log-stock price is consistent with the mathematical finance literature and should not create any confusion.
We summarise in Table~\ref{tab:Summary} the most common rough volatility models 
used in mathematical finance, indicating where their asymptotic behaviours were covered, 
and where our framework not only encompasses those, but fills the gaps so far missing. As discussed below, our application to the rough Heston model is conditional on the latter to have a unique pathwise solution, a problem that remains open so far.
The detailed analysis of these cases is then provided in Section~\ref{sec:SmallTimeEx} 
in the small-time case, and in Section~\ref{sec:TailsEx} for their tail behaviours.

\begin{table}[ht!]
\begin{center}
 \begin{tabular}{|c|c|c|c|c|c|c|c|c|c|c|c|}
 \hline
 \multirow{4}{4em}{{\bf Models}} &
 \multicolumn{4}{c|}{} &  \multicolumn{2}{c|}{{\bf Rough}} &
 \multicolumn{2}{c|}{{\bf multi-factor}} &\multicolumn{3}{c|}{} 
 \\
 & \multicolumn{4}{c|}{{\bf Rough Stein-Stein}} & \multicolumn{2}{c|}{{\bf Bergomi}} &
 \multicolumn{2}{c|}{{\bf rough Bergomi}} &\multicolumn{3}{c|}{{\bf Rough Heston}} 
 \\
 & \multicolumn{2}{c}{Small-time} & \multicolumn{2}{c|}{Tail} & \multicolumn{2}{c|}{Small-time} & \multicolumn{2}{c|}{Small-time} & \multicolumn{2}{c}{\makecell{Small-time \\ \cite{FGS19}}} & \multicolumn{1}{c|}{Tail}    \\ 
 & \makecell{LDP \\ \cite{HJL19}} & MDP & \makecell{LDP \\ \cite{HJL19}} & MDP & \makecell{LDP \\ \cite{JPS18}}& MDP & \makecell{LDP \\ \cite{LMS19}} & MDP & LDP & MDP & LDP \\
  \hline
~$(X^\ep,Y^\ep)$ & \cellcolor{lightgray}\ocean{IF} & \cellcolor{lightgray}\ocean{IF}  & \cellcolor{lightgray}\ocean{IF} & -
 & \cellcolor{lightgray}\ocean{IF} & \cellcolor{lightgray}\ocean{IF}  & \cellcolor{lightgray}\ocean{IF} & \cellcolor{lightgray}\ocean{IF} 
 & \cellcolor{lightgray}\ocean{IF} & \cellcolor{lightgray}\ocean{IF} & \cellcolor{lightgray}\ocean{IF}  \\
 \hline
~$X^\ep$ & \sunset{OP} & \cellcolor{lightgray}\ocean{IF}   & \sunset{OP} & -
 & \sunset{OP} & \cellcolor{lightgray}\ocean{IF} & \cellcolor{lightgray}\sunset{OP} & \cellcolor{lightgray}\sunset{OP} 
 & \cellcolor{lightgray}\sunset{OP} & \cellcolor{lightgray}\ocean{IF} & \cellcolor{lightgray} \sunset{OP} \\
 \hline 
~$X_\ep$ & \sunset{OP} & \cellcolor{lightgray}\forest{CF}   & \sunset{OP} & - 
 & \sunset{OP} & \cellcolor{lightgray}\forest{CF}   &\cellcolor{lightgray}\sunset{OP} & \cellcolor{lightgray}\sunset{OP}
 & \sunset{OP} & \forest{CF} &  \cellcolor{lightgray}\sunset{OP} \\
 \hline
~$\widehat{\sigma}$ & \sunset{OP} & \cellcolor{lightgray}\forest{CF}   & \sunset{OP} & -
 & \sunset{OP} & \cellcolor{lightgray}\forest{CF}  & \cellcolor{lightgray}\sunset{OP} & \cellcolor{lightgray}\sunset{OP}
 & \sunset{OP} & \forest{CF} & \cellcolor{lightgray}\sunset{OP} \\
 \hline
	~$Y^\ep$ & \ocean{IF} & \cellcolor{lightgray}\ocean{IF}  & \ocean{IF} & \cellcolor{lightgray}\ocean{IF}
 & \ocean{IF} & \cellcolor{lightgray}\ocean{IF}  & \cellcolor{lightgray} \ocean{IF} & \cellcolor{lightgray}\ocean{IF} 
 & \cellcolor{lightgray}\ocean{IF} & \cellcolor{lightgray}\ocean{IF} & \cellcolor{lightgray}\ocean{IF}  \\
 \hline
~$Y_\ep$ & \forest{CF} & \cellcolor{lightgray}\forest{CF}  & \ocean{IF} & \cellcolor{lightgray}\ocean{IF}
 & \forest{CF} & \cellcolor{lightgray}\forest{CF}  & \forest{CF} & \cellcolor{lightgray}\forest{CF} 
 & \cellcolor{lightgray}\ocean{IF} & \cellcolor{lightgray}\forest{CF} &   \cellcolor{lightgray}\ocean{IF} \\
 \hline
\end{tabular}
\end{center}
\caption{Summary of rough volatility results and form of the rate functions (CF=closed-form; IF=integral form; OP=optimisation problem; shadowed cells are new contributions from this paper).
$\widehat{\sigma}$ corresponds to the implied volatility, defined precisely in Section~\ref{sec:IVsmalltime}.}
\label{tab:Summary}
\end{table}

%%%%%%%%%%%%%%%%%%%%%%%%%%%%%%%%%%%%%%%%%%%%%%%%%%%%%%%%%%%%%%%%%%%%%%%%%%%%%%%%%%%%%%%%%%%%%%%%%%%%%%

\subsection{Small-time rescaling (general)}\label{sec:SmallTimeGen}
In the small-time case, we need to assume some scaling behaviour for the kernel functions.
We say that a function $f:\bR\to\bR$ is homogeneous of degree~$\alpha\in\bR$ if 
$f( \lambda x)=\lambda^\alpha f(x)$ holds for all~$x,\lambda\in\bR$. 
\begin{assumption}\label{assu:kernelrvol}
$K_1$ and~$K_2$ are homogeneous of degrees~$\varpi\in (-\half, \half]$
and $H-\half \in(-\half, \half]$. 
\end{assumption}
Since~$K_1$ is homogeneous of degree~$\varpi$, then
\begin{equation*}
\begin{array}{rll}
\displaystyle \int_0^h K_1(t)^2 \dt
 & = \displaystyle \int_0^h K_1(1)^2 t^{2\varpi}\dt \le \frac{K_1(1)^2}{1+2\varpi} h^{1+2\varpi},
& \text{for any } h>0,\\
\displaystyle \int_0^T \big(K_1(t+h)-K_1(t)\big)^2\dt
 & = \displaystyle K_1(1)^2 \int_0^T \big((t+h)^\varpi - t^\varpi \big)^2 \dt
= \cO\left(h^{1+2\varpi}\right), & \text{for } h \text{ small enough},
\end{array}
\end{equation*}
and so Assumption~\ref{assu:K} is satisfied with $\gamma = 1+2\varpi \in (0,2]$,
and likewise for~$K_2$ with $\gamma= 2H$.
Under this assumption, the rescalings~$ X^\ep_t:= \ep^{H-\half} X_{\ep t}$ and $Y^\ep_t:= Y_{\ep t}$ turn~\eqref{eq:rvolmodel} into
\begin{align}
\label{eq:smalltime}
\left\{
\begin{array}{rl}
    X^\ep_t &= \displaystyle - \frac{\ep^{H+\half}}{2}\int_0^t \Sigma(Y^\ep_s) \ds + \ep^H \int_0^t \sqrt{\Sigma(Y^\ep_s)} \D B_s, \\
    Y^\ep_t &= \displaystyle y_0 + \ep^{1+\varpi} \int_0^t K_1(t-s) \bb(Y^\ep_s)\ds + \ep^H \int_0^t K_2(t-s) \zeta(Y^\ep_s) \D W_s,
\end{array}
\right.
\end{align}
so that we are precisely in the framework of~\eqref{eq:Xep} with $d=3$,  $\vartheta_\ep=\ep^H$,
\[
K(t) = \begin{pmatrix}
1 & 0 & 0\\
0 & K_1(t) & K_2(t) \\
0 & 0 & 0
\end{pmatrix}, \quad
b_{\ep}(t, (x,y)) 
= \begin{pmatrix}
 - \half\ep^{H+\half}\Sigma(y)\\
  \ep^{1+\varpi}\bb(y) \\
  0
\end{pmatrix},\quad
\sigma_{\ep}(t, (x,y)) = \begin{pmatrix}
\rho \sqrt{\Sigma(y)} & \rrho \sqrt{\Sigma(y)} & 0\\
0 & 0 & 0 \\
\zeta(y) & 0 & 0
\end{pmatrix},
\]
where, similarly to Example~\ref{ex:Bergomi}, the additional dimension allows to handle the two different kernels. Note that~$\sigma_\ep$ does not depend on~$\ep$ but encodes the correlation.
The controlled equation~\eqref{eq:Xepv} for the second component reads
\[
Y^{\ep,v}_t = y_0 + \ep^{1+\varpi} \int_0^t K_1(t-s) \bb(Y^{\ep,v}_s)\ds + \ep^H \int_0^t K_2(t-s) \zeta(Y^{\ep,v}_s) \D W_s + \int_0^t  K_2(t-s) \zeta(Y^{\ep,v}_s) v_s \ds,
\]
for each $t\in\bT,\ep>0$ and $v\in\cA$. 
Note that the dynamics of~$X^\ep$ do not feed back into~$Y^\ep$
 and that $\Sigma\in\cS^{\abs{\Sigma}}_{\{2\}}$ in the sense of Definition~\ref{def:autonFunction}.
The following assumption stands throughout this section:
\begin{assumption}[Small-time assumptions]\label{assu:rvolLDPSmallTime}\ 
\begin{itemize}
\item $K_2(t):= t^{H-\half}/\Gamma(H+\half)$;
\item $\Sigma$, $\zeta$ and~$\bb$ are continuous on~$\bR$;
\item $\bb$ and~$\zeta$ are of linear growth;
\item $\Sigma$ is either of linear growth or such that for all $p\ge1$, $N>0$ and~$\ep>0$ small enough,
\begin{equation}\label{eq:boundSigma}
\sup_{t\in\bT, v\in\cA_N}\bE \Big[\abs{\Sigma\big(Y^{\ep,v}_t \big)}^p\Big] < \infty; 
\end{equation}
\item the equation for~$Y^\ep$ in~\eqref{eq:smalltime} is pathwise unique for small enough $\ep>0$.
\end{itemize}
\end{assumption}
The choice of kernel~$K_2$ is a common setup in rough volatility models and allow for more explicit results.
These conditions ensure that {\bf H2} holds with limit coefficients $b\equiv (0,0,0)^\top$ and $\sigma=\sigma_\ep$. 
Furthermore, $Y^\ep$ is an autonomous subsystem in the sense of Definition~\ref{def:autonomous} and {\bf H3} and the bound~\eqref{eq:boundgsto} hold. 
An pathwise unique solution of the system~\eqref{eq:rvolmodel} exists since~$X^\ep$ is explicit from~$Y^\ep$, and {\bf H4} is satisfied.

%%%%%%%%%%%%%%%%%%%%%%%%%%%%%%%%%%%%%%%%%%%%%%%%%%%%%%%%%%%%%%%%%%%%%%%%%%%%%%%%%%%%%%%%%%%%%%%%%%%%%%%%%%%%%%%%%%%%%%%%%%%%%%%%%%%%

%%%%%%%%%%%%%%%%%%%%%%%%%%%%%%%%%%%%%%%%%%%%%%%%%%%%%%%%%%%%%%%%%%%%%%%%%%%%%%%%%%%%%%%%
\subsubsection{Large deviations} 
For each control~$v\in \cS_N$ with $N>0$, the limit equation~\eqref{eq:LDPlimit} of the volatility in the large deviations regime reads
\begin{equation}
\vphi_t = y_0 + \int_0^t \frac{(t-s)^{H-\half}}{\Gamma(H+\half)} \zeta(\vphi_s) v_s\ds, \qquad \text{for }t \in\bT.
\label{eq:LDPlimitrvol}
\end{equation}
From the uniform bound on~$\vphi$ derived in Remark~\ref{rem:boundphi} and the continuity of~$\Sigma$, we obtain that~$\abs{\Sigma(\vphi_t)}$ is uniformly bounded in~$t\in\bT$ and in~$v\in\cS_N$, hence~\eqref{eq:boundgdet} holds. Therefore Assumption~\ref{assu:Gammabound} and {\bf H1 - H4}
follow from Assumption~\ref{assu:rvolLDPSmallTime}.
Mimicking the fractional integral notation from Section~\ref{sec:NonConv}, we introduce for convenience the notations
$$
\II^{H+\half}_x(f) := x + \int_0^t \frac{(t-s)^{H-\half}}{\Gamma(H+\half)} f_s\ds
\qquad\text{and}\qquad
\II^{H+\half}_x(L^1):=\left\{\II^{H+\half}_x(f), f\in L^1\right\},
$$
and the fractional derivative~$\DD$ is defined in~\eqref{eq:RiemmanLiouville}.
From now on, to simplify the statements, 
we write $Z^\ep\sim\LDP(I, \ep^{-1})$ to express that the family of random variables~$\{Z^\ep\}_{\ep>0}$ satisfies an LDP with rate function~$I$ and speed~$\ep^{-1}$, 
as~$\ep$ tends to zero.
%%%%%%%%%%%%%%%%%%%%%%%%%%%%%%%%%%%%%%%%%%%%

\begin{proposition}[Large deviations]
\label{prop:rvolLDP}
Under Assumptions~\ref{assu:relax}, \ref{assu:kernelrvol} and~\ref{assu:rvolLDPSmallTime}, the following hold:
\begin{enumerate}[{\bf (L1)}]
    \item~$(X^\ep,Y^\ep)\sim\LDP\left(I,\ep^{-2H}\right)$, where~$I: \cW^2\to \overline{\bR}_+$ is given by
    \[
     I(\phi,\vphi) =  \inf\left\{ \half \int_0^T \left(u_t^2 + v_t^2\right) \dt:
    u,v\in L^2, 
    \phi_t = \int_0^t \sqrt{\Sigma(\vphi_s)} \big( \rrho u_s + \rho v_s\big) \ds,  \,
    \vphi = \II_{y_0}^{H+\half} (\zeta(\vphi) v) \right\};
    \]
    \item~$X^\ep \sim\LDP\left(I^X,\ep^{-2H}\right)$, where
    \[
    \displaystyle I^X(\phi) = \inf \big\{ I(\phi,\vphi): \vphi \in \II^{H+\half}_{y_0}(L^1) \big\},
    \]
    if $\phi \in \Ac_0$ and infinity otherwise;
    \item $\ep^{H-\half} X_\ep\sim\LDP\left(I^X_1,\ep^{-2H}\right)$ where 
    $I^X_1 (x) = \inf \big\{ I^{X}(\phi): \phi_1=x \big\}$ for all $x\in\bR$;
    \item $Y^\ep\sim\LDP\left(I^Y,\ep^{-2H}\right)$, where
    \[
    \displaystyle 
    I^Y(\vphi)=  \half \int_0^T \left( \frac{\DD^{H+\half}(\vphi-y_0)(t)}{\zeta(\vphi_t)}\right)^2 \one_{\zeta(\vphi_t)\neq0} \dt,
    \]
    if $\vphi\in\II_{y_0}^{H+\half}(L^1)$ and infinity otherwise;
    \item $Y_\ep\sim\LDP\left(I^Y_1,\ep^{-2H}\right)$, with
$I^Y_1(y)= \inf\big\{ I^Y(\vphi) : \vphi_1=y \big\}$ for $y \in \bR$.
    \end{enumerate}
\end{proposition}
While {\bf (L1)}, {\bf (L2)} and {\bf (L4)} deal with pathwise large deviations,
{\bf (L3)} and {\bf (L5)} are one-dimensional large deviations statements, 
about the marginal distributions of~$X$ and~$Y$.
In this small-time behaviour case, we recover the same scaling as in~\cite{FGS19, FZ15}.

\begin{proof}\ %[Proof of Proposition~\ref{prop:rvolLDP}]\ 
\begin{enumerate}[{\bf (L1)}]
    \item As discussed above, the assumptions of Theorem~\ref{thm:LDP} are satisfied, 
    so that the three-dimensional process~$(X^\ep,Y^\ep,Z^\ep)$, where $Z^\ep \equiv0$ for all $\ep>0$, satisfies an LDP with rate function
    \begin{align*}
        J(\phi,\vphi,\psi) =  \inf\bigg\{ & \half \int_0^T \left(u_t^2 + v_t^2 + w_t^2\right) \dt: \quad u,v,w\in L^2, \\
    &\phi_t = \int_0^t \sqrt{\Sigma(\vphi_s)} \big( \rrho u_s + \rho v_s\big) \ds, \quad \vphi = \II_{y_0}^{H+\half} \big( \zeta(\vphi) v \big), \quad \psi\equiv0 \bigg\}.
    \end{align*}
    The map $(X^\ep,Y^\ep,Z^\ep)\mapsto (X^\ep,Y^\ep)$ is continuous so that the contraction principle~\cite[Theorem 4.2.1]{DZ10} yields an LDP for~$(X^\ep,Y^\ep)$ with rate function
$\inf \big\{ J(\phi,\vphi,\psi): \psi\equiv0 \big\} = J(\phi,\vphi,0)$,
    which corresponds to~$I$.
    \item Since the map $(X^\ep,Y^\ep)\mapsto X^\ep$ is continuous, the claim follows from the contraction principle.
    \item Projecting the pathwise large deviations {\bf (L2)} onto the last coordinate point $t=1$
is equivalent to applying the contraction principle, and the claim follows immediately.
\item A direct application of Theorem~\ref{thm:LDP} yields an LDP with rate function
\begin{equation*}
    I^Y(\vphi)= \inf\bigg\{ \half \int_0^T  v_t^2 \dt: v\in L^2, \vphi = \II_{y_0}^{H+\half}(\zeta(\vphi) v)
    \bigg\};
\end{equation*}
Inverting it as above ends the proof and {\bf (L5)} follows from the contraction principle.
\end{enumerate}
\end{proof}
We observe that in some special cases one can reach a more explicit expression for~$I$. 
\begin{corollary}
Under the same assumptions as Proposition~\ref{prop:rvolLDP}, if~$\rho=0$ or~$\zeta(\vphi_t)\neq0$ almost everywhere, the rate function~$I$ can be written
\[
    I(\phi,\vphi)= \frac{1}{2\rrho^2} \int_0^T \left(\frac{\dot\phi_t^2}{\Sigma(\vphi_t)}\one_{\Sigma(\vphi_t)\neq0 } - \frac{2\rho \dot\phi_t  \DD^{H+\half} (\vphi-y_0)(t)}{\zeta(\vphi_t) \sqrt{\Sigma(\vphi_t)}} \one_{\zeta(\vphi_t)\Sigma(\vphi_t)\neq0 }  + 
    \left(\frac{\DD^{H+\half} (\vphi-y_0)(t)}{\zeta(\vphi_t)}\right)^2 \one_{\zeta(\vphi_t)\neq0 } \right)\dt,
\]
if~$\phi\in\Ac_0$ and $\vphi\in\II_{y_0}^{H+\half}(L^1)$, and infinity otherwise.
\label{coro:explicitrate}
\end{corollary}
\begin{proof}
We start from the definition of~$I$ given in Proposition~\ref{prop:rvolLDP}\textbf{(L1)}. For each~$\phi\in\Aco$ and $\vphi\in\II^{H+\half}_{y_0}(L^1)$, reverting the integral which defines~$\phi$ gives
\[
u_t = \frac{1}{\rrho} \left(\frac{\dot{\phi_t}}{\sqrt{\Sigma(\vphi_t)}}- \rho v_t \right)\one_{\Sigma(\vphi_t)\neq0}, 
\quad \text{for all } t\in\bT,
\]
because whenever~$\Sigma(\vphi_t)=0$, although~$u$ is not uniquely determined by~$\phi$, the optimal choice of control (the one minimising the cost) is~$u_t=0$, see~\cite[Remark 2.3]{CDMD15} for more details. 
In the uncorrelated case~$\rho=0$, the same reasoning for the equation that~$\vphi$ solves yields for all~$t\in\bT$
\[
\quad v_t = \frac{1}{\zeta(\vphi_t)}\DD^{H+\half} (\vphi-y_0)(t) \one_{\zeta(\vphi_t)\neq0}.
\]
Furthermore, in the special case where~$\zeta(\vphi_t)\neq0$ almost everywhere the equality above holds almost everywhere, which is sufficient for the optimisation problem (because of the correlation, $v$ may not be equal  to zero even if~$\zeta(\vphi)$ is).
Plugging these into the rate function yields the claim.
The last condition stands because if~$\phi\notin\Aco$ or~$\vphi\notin\II^{H+\half}_{y_0}(L^1)$ then they cannot satisfy the equations and therefore the infimum takes place over an empty set.
\end{proof}

\notthis{
\begin{remark}
\label{remark:FWcomp}
The form of $I$ above is reminiscent of the Freidlin-Wentzell rate function after application of the contraction principle from the Brownian motion LDP. The weak convergence approach allows to relax the assumptions of continuity (essentially Lipschitz continuity of the coefficients) to simple well-posedness of the equation.
\end{remark}
}

%%%%%%%%%%%%%%%%%%%%%%%%%%%%%%%%%%%%%%%%%%%%%%%%%%%%%%%%%%%%%%%%%%%%%%%%
\subsubsection{Moderate deviations}
We now show how our moderate deviations results apply to the rough volatility model~\eqref{eq:rvolmodel}.
Let~$h_\ep  =\ep^{-\beta}$ for any $\beta \in (0,H)$, and define the two-dimensional process
\[
\eta^\ep := \frac{1}{\vartheta_\ep h_\ep  } \big( X^\ep, Y^\ep-y_0\big) = \frac{1}{\ep^{H-\beta}} \big( X^\ep, Y^\ep-y_0\big).
\]
The case~$\beta=0$ corresponds to the Central Limit Theorem, whereas $\beta=H$ is the LDP regime, 
so that MDP precisely corresponds to some interpolation between the two. 
Regarding the assumptions note that~$y_0^\ep=y_0$ and~$b_\ep$ clearly tends to zero as~$\ep$ goes to zero, hence it is trivial that $b\equiv0$ is continuously differentiable and Lipschitz continuous, and thus {\bf H5} and {\bf H7} hold. 

Now let Assumption~\ref{assu:rvolLDPSmallTime} hold.
Denoting the $i$-th component of~$b$ with~$b^{(i)}$, one notices that $\big\lvert b_\ep^{(1)}(t,(x,y))-b^{(1)}(t,(x,y))\big\lvert\le \ep^{H+\half} \abs{\Sigma(y)}$ and $\big\lvert b_\ep^{(2)}(t,(x,y))-b^{(2)}(t,(x,y))\big\lvert\le \ep^{1+\varpi} C_L (1+\abs{y})$ by linear growth of~$\bb$.
For {\bf H8} to hold, one then requires that $\ep^{1+\varpi - (H-\beta)}$ and $\ep^{H+\half-(H-\beta)}$ both tend to zero as $\ep$ goes to zero.
Moreover the bound~\eqref{eq:boundSigma} implies~\eqref{eq:Xibound}.
\begin{assumption}[Moderate deviations assumptions]\label{assu:rvolMDPSmallT}\
\begin{itemize}
\item The parameters~$H,\varpi$ and~$\beta$ are such that~$(1+\varpi) \wedge (H+\half) > H-\beta$;
\item There exists~$\delta>0$ such that $\Sigma$ and~$\zeta$ are locally~$\delta$-H\"older continuous.
\end{itemize}
\end{assumption}
Notice that the first inequality is always satisfied if~$H\le\half$.
Therefore, Assumptions~\ref{assu:kernelrvol}, \ref{assu:rvolLDPSmallTime},~\ref{assu:rvolMDPSmallT} imply {\bf H1 - H8} and Assumptions~\ref{assu:K} and~\ref{assu:Gammabound}.
Similarly to the LDP case, and recalling the definition of MDP from the introduction, 
we write $Z^\ep \sim \MDP(\Lambda, l_\ep)$ if in fact $\ep^{\beta-H}(Z^\ep-\overbar{Z})\sim\LDP(\Lambda,l_\ep)$,
for any $l_\ep>0$ converging to zero as~$\ep$ tends to zero, where~$\overbar Z$ is the limit in distribution of~$Z^\ep$. We also denote the subset of~$\cW^d$ of absolutely continuous functions by~$\Ac$, 
and~$\Aco:= \{ \phi\in\Ac, \phi_0=0\}$, and refer to~\eqref{eq:RiemmanLiouville} for the definition of the Riemann-Liouville fractional derivative.

\begin{proposition}
\label{prop:rvolMDP}
Under Assumptions~\ref{assu:kernelrvol}, ~\ref{assu:rvolLDPSmallTime}, ~\ref{assu:rvolMDPSmallT} and the condition~$\zeta(y_0)\Sigma(y_0)\neq0$, the following moderate deviations hold:
\begin{enumerate}[{\bf (M1)}]
    \item $(X^\ep,Y^\ep)\sim\MDP\left(\Lambda,\ep^{-2\beta}\right)$, where~$\Lambda: \cW^2\to \overline{\bR}_+$ is given by
     \begin{align}
    \label{eq:rvolrateMDP}
    \Lambda(\phi,\vphi)= \frac{1}{2\rrho^2} \int_0^T \left( \frac{\dot\phi_t^2}{\Sigma(y_0)} - \frac{2\rho \dot\phi_t  \DD^{H+\half} (\vphi-y_0)(t)}{\zeta(y_0) \sqrt{\Sigma(y_0)}}   + 
    \left(\frac{\DD^{H+\half} (\vphi-y_0)(t)}{\zeta(y_0)}\right)^2  \right)  \dt,
\end{align}
if~$\phi\in\Ac_0$ and $\vphi\in\II_{y_0}^{H+\half}(L^1)$, and infinity otherwise;
    \item $X^\ep \sim\MDP\left(\Lambda^X,\ep^{-2\beta}\right)$, where
    $\displaystyle \Lambda^X (\phi) = \frac{1}{2 \Sigma(y_0)} \int_0^T  \dot\phi_t^2 \dt$,
    if $\phi\in\Ac_0$ and infinity otherwise;
    \item $\ep^{H-\half} X_\ep\sim\MDP\left(\Lambda_1^X, \ep^{-2\beta}\right)$, with 
    $\displaystyle\Lambda^X_1(x)=\frac{x^2}{2 \Sigma(y_0)}$, for $x\in\bR$;
    \item $Y^\ep\sim\MDP\left(\Lambda^Y,\ep^{-2\beta}\right)$, where~$\Lambda^Y: \cW\to \overline{\bR}_+$ is given by
    $$
    \Lambda^Y(\vphi)=  \half \int_0^T  \left(\frac{\DD^{H+\half}(\vphi-y_0)(t)}{\zeta(y_0)} \right)^2 \one_{\zeta(y_0)\neq0 } \dt,
    \qquad\text{if  }\vphi\in\II_{y_0}^{H+\half}(L^1) \text{ and infinity otherwise};
    $$
    \item $Y_\ep\sim\MDP(\Lambda^Y_1,\ep^{-2\beta})$,  where
     $\Lambda^Y_1(y)= \half y^2$
     for $y\in\bR$.
    \end{enumerate}
\end{proposition}

\begin{proof}\ 
\begin{enumerate}[{\bf (M1)}]
\item As discussed above, the assumptions of Theorem~\ref{thm:MDP} are satisfied, thus it yields an MDP with rate function
$$
    \Lambda(\phi,\vphi) = \left\{ \half \int_0^T \left(u_t^2 + v_t^2\right) \dt:
    u,v\in L^2, 
    \phi_t = \int_0^t \sqrt{\Sigma(y_0)} \big( \rrho u_s + \rho v_s\big) \ds, 
    \vphi = \II_{y_0}^{H+\half}(\zeta(y_0) v) \right\},
$$
    and inverting it as in the LDP case gives the claim.
    \item The contraction principle implies that an MDP for~$X^\ep$ holds with rate function
    $\Lambda^X (\phi) = \inf \big\{ \Lambda(\phi,\vphi): \vphi\in \II^{H+\half}_{y_0} (L^1) \big\}$. Let~$\psi\in\Ac_0$ such that~$\dot\psi:= \frac{\DD^{H+\half}(\vphi-y_0)}{\zeta(y_0)}$, then the rate function translates to
    \[
\Lambda^X (\phi) = \inf \bigg\{ \frac{1}{2\rrho^2} \int_0^T \bigg( \frac{\dot\phi_t^2}{\Sigma(y_0)}  - \frac{2\rho \dot\phi_t \dot\psi_t}{\sqrt{\Sigma(y_0)}}  + \dot\psi_t^2 \bigg) \dt, \quad \psi\in\Ac \bigg\},
\]
which can be solved as a variational problem as in~\cite[Corollary 2.4]{JK20}. The corresponding Euler-Lagrange equation reads $\ddot\psi=\rho \ddot\phi/\sqrt{\Sigma(y_0)}$ hence~$\dot\psi=\rho\dot\phi/\sqrt{\Sigma(y_0)}$ because~$\dot\psi_0=0$ by definition. Plugging into the above equation finishes the proof.
    \item The rate function is given by contraction principle as 
\[
\Lambda^X_1(x)=\inf\big\{ \Lambda^X (\phi): \phi\in\Ac_0, \phi_1=x\big\}= \inf \bigg\{ \frac{1}{2 \Sigma(y_0)} \int_0^T \dot\phi_t^2 \dt : \phi\in\Ac_0, \phi_1=x \bigg\}.
\]
Setting~$T=1$ the optimal path under the constraint~$\phi_1=x$ is~$\phi_t= x t$ by the Euler-Lagrange equation. Again, plugging it into the rate function ends the proof. 
    \item Theorem~\ref{thm:MDP} gives an MDP for~$Y^\ep$ with rate function
    \begin{equation*}
    \Lambda^Y(\vphi)= \inf\bigg\{ \half \int_0^T  v_t^2 \dt: v\in L^2, \vphi = \II_{y_0}^{H+\half}(\zeta(y_0) v)
    \bigg\}.
    \end{equation*}
    Inverting it yields {\bf (M4)}.
    \item By contraction principle one obtains $\Lambda^Y_1(y)= \inf\big\{ \Lambda^Y(\vphi), \vphi_1=y \big\}$. Then setting $\psi$ as in {\bf (M2)} it boils down to the same optimisation problem as for {\bf (M3)}.
\end{enumerate}
\end{proof}

As in the large deviations results, 
{\bf (M1)}, {\bf (M2)} and {\bf (M4)} correspond to pathwise statements, 
whereas {\bf (M3)} and {\bf (M5)} are finite-dimensional results about the marginal distributions.
For the log-stock price, {\bf (M3)} corresponds precisely to the moderately-out-of-the-money 
regime presented and justified in~\cite{FGeP18} (for diffusion volatility models), 
based on the observation that the range of observable strikes grows with maturity. Furthermore, one can always apply Theorem~\ref{thm:MDP} in the degenerate case~$\Sigma(y_0)\zeta(y_0)=0$ although the rate functions take slightly different forms.

%%%%%%%%%%%%%%%%%%%%%%%%%%%%%%%%%%%%%%%%%%%%%%%%%%%%%%%%%%%%%%%%%%%%%%%%%%%%%%%%%%%%%%%%%%%%%%%%%%%%%%%%%%%%ù
\subsubsection{Implied volatility asymptotics}
\label{sec:IVsmalltime}
We can easily deduce from the above results the asymptotic behaviour of the implied volatility, 
a standard norm for quoting option prices.
For each maturity $t\geq 0$ and log-moneyness $k\in\bR$, 
the implied volatility~$\widehat{\sigma}(t,k)$ is the unique non-negative solution to $C_{\BS} \big(t,k,\widehat{\sigma}(t,k)\big) = C(t,k)$, 
where~$C_{\BS}$ corresponds to the price of a European Call option under the Black-Scholes model, 
and~$C$ a given Call option price (for example in a rough volatility model). 
This notion is only well defined if the underlying stock price is a true martingale, 
which we have not assumed so far, and may require additional conditions on the coefficients.
This will be the case though in all our examples below, but for now, with the current level of generality, we assume it:

\begin{assumption}\label{assu:Martingale}
The process~$\exp(X^\ep)$ in~\eqref{eq:smalltime} is a true martingale for all small enough $\ep>0$.
\end{assumption}

Small-time implied volatility asymptotics can be derived from Properties~\ref{prop:rvolLDP} and~\ref{prop:rvolMDP} in a similar fashion. The explicit form of the MDP rate function allows a closed form expression.
\begin{corollary}
\label{coro:IV}
Let Assumption~\ref{assu:Martingale} hold.
\begin{itemize}
    \item[{\bf (LDP)}] Under the same assumptions as Proposition~\ref{prop:rvolLDP},
    \begin{equation}
\lim_{t\downarrow 0}\widehat{\sigma} \big(t,k t^{\half-H} \big)^2 
=\left\{
                \begin{array}{rl}
                  \displaystyle \frac{k^2}{2 \inf_{x\ge k}I^X_1(x)},  & \text{if } k>0, \\
                  \displaystyle \frac{k^2}{2 \inf_{x\le k}I^X_1(x)},  & \text{if } k<0.
                \end{array}
              \right.
\end{equation}
    \item[{\bf (MDP)}] Under the same assumptions as Proposition~\ref{prop:rvolMDP} and for any~$\beta\in(0,H)$, $k\neq0$,
\begin{equation}
\lim_{t\downarrow 0}  \widehat{\sigma} \big(t,k t^{\half-\beta} \big)^2 
= \Sigma(y_0).
\end{equation}
\end{itemize}
\end{corollary}

\begin{proof}\,

{\bf (LDP)} Consider the case $k>0$.
Proposition~\ref{prop:rvolLDP}{\bf (L3)} translates into
$$
    \lim_{t\downarrow 0} t^{2H}\log \bP( t^{H-\half} X_t \ge k)  =-\inf_{x\ge k}I^X_1(x).
$$
Meanwhile in the Black-Scholes model with constant volatility~$\sigma>0$ the log-price process satisfies
$X_t = -\frac{\sigma^2 t}{2} + \sigma B_t$ for all $t\in\bT$, and simple Gaussian computations yield
the large deviations behaviour
\[
\lim_{t\downarrow 0} t^{2H} \log \bP( X_t \ge k t^{\half-H})  = -\frac{k^2}{2\sigma^2}.
\]
The claim then follows directly from~\cite[Corollary 7.1]{GL14}, and by symmetry for the case $k<0$.

{\bf (MDP)} Following the same arguments as above, we obtain
\begin{equation*}
\lim_{t\downarrow 0}  \widehat{\sigma} \big(t,k t^{\half-\beta} \big)^2 
=\left\{
\begin{array}{ll}
\displaystyle \frac{k^2}{2 \inf_{x\ge k}\Lambda^X_1(x)}, & \text{if } k>0, \\
\displaystyle \frac{k^2}{2 \inf_{x\le k}\Lambda^X_1(x)}, & \text{if } k<0.
\end{array}
\right.
\end{equation*}
Plugging in the expression of~$\Lambda^X_1$ from {\bf (M3)} finishes the proof.
\end{proof}
This concludes the presentation of the general results for rough volatility models. The next sections display the diversity of the models found in the literature and how large and moderate deviations principles apply to them.

%%%%%%%%%%%%%%%%%%%%%%%%%%%%%%%%%%%%%%%%%%%%%%%%%%%%%%%%%%%%%%%%%%%%%%%%%%%%%%%%%%%%%%%%%%%%%%%%%%%%%%%%%%%%%%%%%%%%%%%%%%%%%%%%%%%%%%%%%%

\subsection{Small-time rescaling (examples)}
\label{sec:SmallTimeEx}

\subsubsection{Rough Stein-Stein}\label{sec:rSS}
The rough Stein-Stein, suggested in~\cite{HJL19} is an extension of the classical Stein-Stein volatility model~\cite{SteinStein} to the fractional setting.
It corresponds to~\eqref{eq:rvolmodel} with $K_1\equiv 1$ (hence $\varpi=0$), $K_2(t)= t^{H-\half}/\Gamma(H+\half)$,
$H\in(0,\half)$, $y_0>0$, $\Sigma(y)=y^2$, $\bb(y)= \kappa(\theta-y)$, $\kappa,\theta>0$ and $\zeta(y)\equiv \xi>0$.
The coefficients are Lipschitz continuous and well-behaved, hence Assumptions~\ref{assu:rvolLDPSmallTime} and~\ref{assu:rvolMDPSmallT} are easily checkable
and the limit equation~\eqref{eq:LDPlimitrvol} has a unique solution, hence Propositions~\ref{prop:rvolLDP} and~\ref{prop:rvolMDP} apply.
Note that because~$\zeta$ is a positive constant, Corollary~\ref{coro:explicitrate} gives the rate function~$I$ in integral form and one can solve {\bf (L5)} in closed-form using the Euler-Lagrange equation in a similar way as in the proof of Proposition~\ref{prop:rvolMDP}.
Furthermore, since~$Y$ is Gaussian its exponential moments are finite and Novikov's condition~\cite[Section 3.5.D]{KS98}
 ensures that Assumption~\ref{assu:Martingale} holds.
Therefore, Corollary~\ref{coro:IV} yields the small-time behaviour of the implied volatility.
Notice that the LDP and MDP for this model still hold when replacing the Riemann-Liouville kernel with the standard fractional Brownian motion by virtue of Theorem~\ref{thm:nonconv}. 
The pathwise LDP for this model was first derived in~\cite{HJL19} albeit with the different scaling $X^\ep_t := \ep^{H-\half+2\beta} X_{\ep t}$ and~$Y^\ep_t := \ep^{\beta} Y_{\ep t}$, for $\beta>0$.

%%%%%%%%%%%%%%%%%%%%%%%%%%%%%%%%%%%%%%%%%
\subsubsection{Rough Bergomi}\label{sec:rBergomi}
The rough Bergomi model as presented in~\cite{BFG16} reads
\begin{align*}
\left\{
\begin{array}{rl}
    X_t &= \displaystyle -\half \int_0^t V_s \ds + \int_0^t \sqrt{V_s} \D B_s, \\
    V_t &= \displaystyle V_0 \exp\left( \int_0^t \frac{(t-s)^{H-\half}}{\Gamma(H+\half)} \D W_s -at^{2H} \right),
\end{array}
\right.
\end{align*}
with $V_0>0$ and $a\in\bR$. A pathwise LDP for this model first appeared in~\cite{JPS18} using the Freidlin-Wentzell approach and a tailored proof. 
This case is quite intricate because the exponential does not satisfy the linear growth bound but we circumvented this issue by introducing the notion of autonomous system, illustrated in Example~\ref{ex:Bergomi} and completed by Assumption~\ref{assu:Gammabound} and~{\bf H3b}. 
Not only does this framework unifies the result of~\cite{JPS18} with other rough volatility models, but it also leads to a pathwise MDP.

With $Y:=\log(V)$, the system $(X,Y)$ fits into~\eqref{eq:rvolmodel} 
where $K_1(t)=t^{2H-1}$, $K_2(t)= t^{H-\half}/\Gamma(H+\half)$, for $H\in(0,\half)$, $\Sigma(y)=\exp(y)$, $\bb(y)=-a/(2H)$, $a>0$, $\zeta(y)\equiv 1$ and $y_0= \log(V_0)$.
The bound~\eqref{eq:boundSigma} is then satisfied since
$\int_0^t (t-s)^{H-\half} \D W_s$ is a Gaussian process with exponentional moments bounded uniformly in~$t\in\bT$, and for each~$v\in\cA_N$, $N>0$:
\[
\int_0^t (t-s)^{H-\half} v_s\ds \le N \frac{t^{2H}}{2H},
\]
almost surely, by Cauchy-Schwarz inequality.
Therefore~$\sup_{t\in\bT} \bE\big[ \exp\big( Y_t^{\ep,v}\big) \big]$ is finite, yielding the claim. 
Moreover, the volatility equation is explicit so we shall not be concerned with uniqueness and the rest of Assumptions~\ref{assu:rvolLDPSmallTime} and~\ref{assu:rvolMDPSmallT} is straightforward to check. This implies that Propositions~\ref{prop:rvolLDP} and~\ref{prop:rvolMDP} apply, and so does Corollary~\ref{coro:explicitrate}. Again, Theorem~\ref{thm:nonconv} guarantees that the LDP and MDP still hold when $K_2$ is replaced with the non-convolution fractional Brownian motion kernel.
Gassiat~\cite{Gassiat19} showed that, if~$\rho\le0$, then the stock price process is a true martingale,
ensuring that Assumption~\ref{assu:Martingale} holds, and implied volatility asymptotics thus follow from Corollary~\ref{coro:IV}.

%%%%%%%%%%%%%%%%%%%%%%%%%%%%%%%%%%%%%%%%%%
\subsubsection{Rough Heston}\label{sec:rHeston}
As introduced in~\cite{ER19} the rough Heston model fits into the framework of~\eqref{eq:rvolmodel} with $K_1(t)=K_2(t)= t^{H-\half}/\Gamma(H+\half)$, for $H\in(0,\half)$, $y_0>0$, $\Sigma(y)=y$, $\bb(y)=\kappa\big(\theta-y\big)$, $\kappa>0,\,\theta\ge0$ and $\zeta(y)=\xi \sqrt y$, $\xi\in\bR^d$. 
Linear growth and local H\"older continuity of the coefficients clearly hold.
The weak existence and uniqueness was proved in~\cite{ALP17}, however the square-root coefficient brings an issue for pathwise uniqueness of the SVE. 
We assume here that there exists a set~$\cU$ of coefficients $(H,\kappa,\theta,\xi,\rho,y_0)$ such that pathwise uniqueness indeed stands. 
The only known result so far is due to~\cite{YW70} in the smooth case $H=\half$. We also recall that pathwise uniqueness was proved for~$\zeta(y)= y^\gamma$ where~$\gamma>\frac{1}{2H+1}$ in~\cite{MS15}, but does not encompass the square root case.
Therefore Assumption~\ref{assu:rvolLDPSmallTime} holds in those two cases.
On a heuristic note remark that, even if pathwise uniqueness fails, there is a unique candidate for~$\cG^\ep$ since there exists a unique strong solution until the first hitting time of zero. The issue is it may not satisfy the SVE anymore after that time, but should be consistent for small-time LDP. 

Moreover, uniqueness of the limit equation~\eqref{eq:LDPlimitrvol} only holds up to first hitting time of zero. Hence we will make use of the uniqueness relaxation presented in Section~\ref{sec:WeakAbstract} and similar arguments as in Example~\ref{ex:relax} to prove that Assumption~\ref{assu:relax} holds. The suggested rate function~\eqref{eq:raterelax} reads now
\begin{align}
\label{eq:LambdaHeston}
    I(\phi,\vphi) = \inf\bigg\{ \half \int_0^t \left(u_t^2 + v_t^2\right) \dt:  u,v\in L^2, \,
     \phi_t= \int_0^t \sqrt{\vphi_s} \, \big( \rrho u_s + \rho v_s \big) \ds, 
    \, \vphi_t = \II_{y_0}^{H+\half}(\xi \sqrt{\vphi} v) \bigg\}.
\end{align}
We emphasise that $\vphi$ above solves the Volterra equation
\begin{equation}
    \vphi_t = y_0 + \xi \int_0^t \frac{(t-s)^{H-\half}}{\Gamma(H+\half)} \sqrt{\vphi_s} v_s \ds, \quad \text{for all } t\in\bT.
\label{eq:rHvphi}
\end{equation}
\begin{lemma}
\label{lemma:lebesguezero}
Let~$v\in L^2$ and~$\vphi$ satisfying the Volterra equation~\eqref{eq:rHvphi}
with $H\in(0,\half)$. Then
the set $\cD:= \{t\in[0,T]: \vphi_t>0 \}$ has Lebesgue measure $T$. 
\end{lemma}
\begin{proof}
We follow some arguments in the proof of~\cite[Theorem 3.6]{ALP17}. Let us drop the subscript in the kernel and write it~$K$ for clarity, and introduce its resolvent of the first kind~$L(\dt):=\frac{t^{-H-\half}}{\Gamma(\half-H)}\dt$~\cite[Definition 5.5.1]{Gripenberg90}.
Moreover, for $h>0$, define $\Delta_h K(t):= K(t+h)$ and for every measurable function $f$ on~$\bR_+$ and measure~$g$ on~$\bR_+$, $(f\ast g)(t) := \int_0^t f(t-s) \D g(s)$.
It is proved, in~\cite[Equation~(3.9)]{ALP17}, that $\Delta_h K\ast L$ is non-decreasing but in fact in this special (rough) case, it is strictly increasing. Indeed, the authors show that in the general case, for all $0\le s \le t \le T$,
\[
(\Delta_h K\ast L)(t) - (\Delta_h K\ast L)(s) = \int_0^h K(h-u) \big( L(s+\du) - L(t+\du) \big),
\]
which is positive because $K>0$ and $L$ is decreasing. Furthermore $K$ is decreasing and~$L>0$ thus 
\[
0 < (\Delta_h K\ast L)(t) < (K\ast L)(t) = 1,
\]
where the equality holds by definition.
Let $\vphi_t= y_0+ \int_0^t K(t-s) \xi\sqrt{\vphi_s} v_s \ds=: y_0+ (K\ast z)(t)$, where~$z$ is trivially a semimartingale, hence from~\cite[Equation (2.15)]{ALP17}:
\begin{align*}
    y_0 + (\Delta_h K \ast \D z)(t) 
    = \Big( 1 - (\Delta_h K \ast L)(t) \Big) y_0
    + \Big(\Delta_h K\ast L \Big)(0) \vphi_t 
    + \Big( \D (\Delta_h K \ast L) \ast\vphi\Big)(t),
\end{align*}
which is strictly positive because because~$y_0,\vphi\ge0$ and the two lines before. Now let us suppose there exists an interval~$[t,t+h]\subset\bT$ on which~$\vphi=0$. Then 
\[
\vphi_{t+h} =y_0 + \int_0^{t} K(t+h-s)\sqrt{\vphi_s} v_s \ds = y_0 + (\Delta_h K \ast \D z)(t) >0,
\]
which is a contradiction. Hence no such interval exists and the claim follows.
\end{proof}
\begin{remark}
This argument works for any rough kernel but not for the diffusion case~$H=\half$. We refer to~\cite[Proposition 3.3]{Donati04} for the latter. 
\end{remark}
The previous lemma allows to invert the integrals as showed in the proof of Corollary~\ref{coro:explicitrate} and yields a more explicit form for~$I$:
\begin{equation}
I(\phi,\vphi) = \int_0^T \frac{\one_{\vphi_t>0}}{2 \rrho^2 \vphi_t} \left(  \dot\phi_t^2
     - 2\rho \dot\phi_t \frac{\DD^{H+\half} (\vphi-y_0)(t)}{\xi} + \left( \frac{\DD^{H+\half} (\vphi-y_0)(t)}{\xi} \right)^2 \right) \dt, 
\label{eq:rateHestonexplicit}
\end{equation}
if the integral is well defined, $\phi\in \Aco$, $\vphi \in \II_{y_0}^{H+\half}(L^1)$,
and~$I=+\infty$ otherwise. 
We can now prove the following:
\begin{lemma}
Let~$H\in(0,\half)$, then the functional $I$ satisfies Assumption~\ref{assu:relax}.
\label{lemma:heston}
\end{lemma}
\begin{proof}
Note that any solution~$\vphi$ of~\eqref{eq:rHvphi} is non-negative. 
Let $(\phi,\vphi)$ be such that $I(\phi,\vphi)$ is finite. Then, for each~$\delta>0$, define $\vphi^\delta_t := \vphi_t + \delta t^{H+\half}$
such that~$\vphi^\delta$ is strictly positive. 
Therefore from definition~\eqref{eq:RiemmanLiouville} we have:
\begin{align*}
\DD^{H+\half}(\vphi^\delta-y_0)(t) - \DD^{H+\half}(\vphi-y_0)(t) 
& = \frac{1}{\Gamma(\half-H)} \frac{\D}{\D t} \int_0^t (t-s)^{-H-\half} (\vphi^\delta_s-\vphi_s) \ds \\
& = \frac{\delta}{\Gamma(\half-H)} \frac{\D}{\D t} \int_0^t  (t-s)^{-H-\half} s^{H+\half} \ds 
= \delta \, \frac{\pi \left(H+\half\right)}{\Gamma(\half-H) \cos(\pi H)},
\end{align*}
which entails convergence as~$\delta$ goes to zero, uniformly on $\bT$.
%where we used~\cite[3.196-3]{GR07} to compute the integral and~$\BB$ is the beta function defined in~\cite[8.380]{GR07}. 
Now define the control
\[
v_t := \frac{\DD^{H+\half}(\vphi-y_0)(t)}{\sqrt{\vphi_t}} \one_{\vphi_t>0},
\]
and $v$ belongs to~$L^2$ since $I(\phi, \varphi)$ is finite.
Then for each~$\delta>0$ the control~$v^\delta$ defined as
\[
v^\delta_t := \frac{\DD^{H+\half}(\vphi^\delta-y_0)(t)}{\sqrt{\vphi^\delta_t}} \le \frac{\DD^{H+\half}(\vphi^\delta-y_0)(t)}{\sqrt{\vphi_t}}
\]
is also in~$L^2$ because~$\one_{\vphi_t>0}=1$ almost everywhere.
Furthermore, for all $t\in\cD$,
$
\lim_{\delta\downarrow 0} \big(\vphi^\delta_t \big)^{-1} = (\vphi_t)^{-1}$ and therefore $
\lim_{\delta\downarrow 0} v^\delta_t = v_t$.
Let~$P(\vphi)$ denote the term between brackets in~\eqref{eq:rateHestonexplicit} divided by~$2\rrho^2$, which is non-negative by design since it corresponds to~$\vphi (u^2+v^2)$.
Therefore, by Lemma~\ref{lemma:lebesguezero},
\begin{align*}
    I(\phi,\vphi)-I(\phi,\vphi^\delta) 
    = \int_\cD \left(  \frac{1}{\vphi_t}- \frac{1}{\vphi^\delta_t}\right) P(\vphi)_t \dt + \int_\cD \frac{1}{\vphi^\delta_t} \Big( P(\vphi)_t - P(\vphi^\delta)_t \Big) \dt,
\end{align*}
where the first integrand is smaller than $ P(\vphi)_t/\vphi_t$ for all $t\in\cD$ and this upper bound belongs to~$L^1$ by assumptions. 
Hence the dominated convergence theorem implies that the first integral goes to zero.
From the calculations above we deduce that $P(\vphi)_t - P(\vphi^\delta)_t$ tends to zero uniformly as~$\delta$ goes to zero, hence the second integrand converges pointwise and, for~$\delta$ small enough, is dominated by $P(\vphi)/\vphi$. A second application of DCT yields convergence of the integral, and the claim follows.
\end{proof}
Therefore the large and moderate deviations from Propositions~\ref{prop:rvolLDP} and~\ref{prop:rvolMDP} apply if the coefficients belong to~$\cU$ and~$H\in(0,\half)$. Observe that Proposition~\ref{prop:rvolMDP}{\bf (M3)} agrees with~\cite[Section 3.5]{FGS19}, although the routes taken differ significantly.
El Euch and Rosenbaum~\cite[Appendix B]{ER18} showed that Assumption~\ref{assu:Martingale} is satisfied, 
and the implied volatility behaviour thus follows from Corollary~\ref{coro:IV}.

%%%%%%%%%%%%%%%%%%%%%%%%%%%%%%%%%%%%%%%%
%%%%%%%%%%%%%%%%%%%%%%%%%%%%%%%%%%%%%%%
\subsubsection{Multi-factor rough Bergomi}
Let~$W$ be an~$\bR^{m+1}$-Brownian motion,~$\Zm :=\left(Z^{(1)}, \cdots, Z^{(m)} \right)^\dagger$ where 
\[
Z^{(j)}_t := \int_0^t K^{(j)}(t-s)\D W^{(j)}_s,
\]
and we allow~$K^{(j)}\in L^2(\bT,\bR_+),\,j\in\llbracket 1, m\rrbracket,$ to be homogeneous of different degrees~$H_j-\half$ with~$H_j\in(0,1]$. Therefore, the variance of~$\Zm$ is proportional to~$\Am_t:=\left(t^{2H_1}/(2H_1),\dots, t^{2H_m}/(2H_m)\right)^\dagger$, for all $t\in\bT$.
Assume without loss of generality that the $H_j$ are ordered by increasing values, then we will design the rescaling at the speed~$\ep^{-2H_1}$. Denote~$m^\star:=\max\{j\in\llbracket 1,m\rrbracket: H_j =H_1\}$.
Let~$\cU$ and $\cV$ be~$m$-dimensional square matrices, and $\Ym$ an $m$-dimensional process defined for all $t\in\bT$ by
\[
\Ym_t := y_0 + \cU \Zm_t - \half \cV \Am_t, \quad y_0\in \bR^m.
\]
The log-price reads
\[
    X_t = -\frac{1}{2m} \int_0^t \sum_{i=1}^m  \exp\left(Y^{(i)}_s\right) \ds + \int_0^t \sqrt{\frac{1}{m}\sum_{i=1}^m  \exp\left( Y^{(i)}_s\right)} \D B_s,
\]
where~$B= \rrho W^{(m+1)} + \sum_{j=1}^{m} \rho_j W^{(j)}$, $\rrho^2 + \sum_{i=1}^m \rho_i^2 =1$.
The rescaling $X^\ep_t= \ep^{H_1-\half}X_{\ep t},\Ym^\ep_t := \Ym_{\ep t}$ yields
\begin{align*}
\left\{
\begin{array}{rl}
    &\displaystyle X_t^\ep = -\frac{\ep^{H_1+\half}}{2m} \int_0^t \sum_{i=1}^m \exp\left(Y^{\ep,(i)}_s\right) \ds + \ep^{H_1} \int_0^t \sqrt{\frac1m \sum_{i=1}^m \exp\left(Y^{\ep,(i)}_s\right)} \D B_s, \\
    &\displaystyle Y^{\ep,(i)}_t = y_0^{(i)} + \sum_{j=1}^m \left(\ep^{H_j} \cU_{ij} Z_t^{(j)} - \cV_{ij}\, \frac{(\ep t)^{2H_j}}{2H_j}\right)
    %= y_0^{(i)} + \ep^{H_1} \sum_{j=1}^m  \LL_{ij}^\ep Z_t^{(j)} - a_i (\ep t)^{2H_1}
    , \quad \text{for all } i\in\llbracket 1,m\rrbracket.
    \end{array}
\right.
\end{align*}
As we will shift each BM by~$\ep^{-H_1} \int v^{(j)}$, we notice that~$\ep^{H_j-H_1} \cU_{ij}$ goes to zero if~$H_j>H_1$, i.e. if $j>m^\star$. It means that the roughest component(s) (the one(s) with $H_1$) will outweigh the others, and only the former will make a contribution to the rate function.

Although similar to its one-dimensional counterpart, this model does not fit into the framework of~\eqref{eq:rvolmodel}.
Regarding the assumptions of Theorems~\ref{thm:LDP}, we only check {\bf H3b} and Assumption~\ref{assu:Gammabound} because the others are standard and similar to the one-dimensional case. Clearly $(Y^{\ep,(1)},\cdots,Y^{\ep,(m)})$ is an autonomous subsystem. As a Gaussian process,~$\cU\Zm$ has exponential moments of all orders and for all $N>0$, $j\in\llbracket 1,m\rrbracket$ and $v^{(j)}\in \cA_N$:
\[
\int_0^t K^{(j)}(t-s) v^{(j)}_s \ds \le \sqrt N \lVert K^{(j)} \lVert_2 \quad \text{almost surely},
\]
thus~$\exp\big(Y^{\ep,(i),v} \big)\in L^p(\Omega)$ for all~$p\ge1$. Therefore the bound~\eqref{eq:boundgsto} and Assumption~\ref{assu:Gammabound} are satisfied. This estimate also checks that~\eqref{eq:Xibound} and thus {\bf H8} stand. Since~$\vartheta_\ep=\ep^{H_1}$, we define for the moderate deviations regime $h_\ep=\ep^{\beta}, \beta\in(0,H_1)$.
\begin{corollary}
The pathwise LDP and MDP hold.
\begin{itemize}
    \item $(X^\ep,Y^\ep)\sim\LDP\left(I,\ep^{-2H_1}\right)$ where for all~$\phi\in\cW$ and~$\vphi\in\cW^m$:
\begin{align*}
\hspace{-1cm}& I(\phi,\vphi)=\inf \Bigg\{ \half \int_0^T \left(u_t^2 + \abs{v_t}^2\right) \dt : \quad u\in L^2(\bT,\bR), v\in L^2(\bT,\bR^m), \\
&\phi_t = \int_0^t \left(\frac1m \sum_{i=1}^m \E^{\vphi^{(i)}_s} \right)^\half \left(\rrho u_s + \sum_{j=1}^m \rho_j v_s^{(j)}  \right) \ds ,  \quad 
\vphi^{(i)}_t = y_0^{(i)} + \sum_{j=1}^{m^\star} \cU_{ij} \int_0^t K^{(j)}(t-s)  v_s^{(j)}  \ds
\Bigg\}.
\end{align*}
    \item $(X^\ep,Y^\ep)\sim\MDP\left(\Lambda,\ep^{-2\beta}\right)$ where for all~$\phi\in\cW$ and $\vphi\in\cW^m$:
\begin{align}\label{eq:multirBMDP}
  \hspace{-1cm} & \Lambda(\phi,\vphi)= \inf \Bigg\{ \half \int_0^T \left( u_t^2 + \abs{v_t}^2 \right) \dt :
    \quad u\in L^2(\bT,\bR), v\in L^2(\bT,\bR^m),  \\
&\phi_t = \int_0^t \left(\frac1m \sum_{i=1}^m \E^{y_0^{(i)}} \right)^\half \left(\rrho u_s + \sum_{j=1}^m \rho_j v_s^{(j)} \right) \ds, 
\quad \vphi^{(i)}_t = \sum_{j=1}^{m^\star}  \cU_{ij} \int_0^t  K^{(j)}(t-s) v_s^{(j)}  \ds
\Bigg\}.\nonumber
\end{align}
\end{itemize}
\end{corollary}
\begin{proof}
The LDP is a direct application of Theorem~\ref{thm:LDP} and the MDP of Theorem~\ref{thm:MDP}.
\end{proof}
One can also recover the LDP and MDP for~$(X^\ep)_{\ep>0}$ as well as the small-time LDP and MDP by contraction principle, as in Propositions~\ref{prop:rvolLDP} and~\ref{prop:rvolMDP}.

If~$\cU$ is lower triangular (i.e. $\cU_{ij}=0$ for all $i<j$), for instance if it arises from the Cholesky decomposition of a covariance matrix, then for all~$\phi\in\cW,$~$\vphi\in\cW^m$, one can derive the vector $v$ recursively, followed by~$u$. Note that if $m^\star<m$, $\vphi$ may not be attainable by the restrained number of controls $\{v^{(j)}, j\in\llbracket1, m^\star\rrbracket\}$.
\begin{example}
Consider the case $m=2$. Let $K^{(1)}(t)=K^{(2)}(t)=t^{H-\half} / \Gamma(H+\half)$ for $H\in(0,\half)$, hence~$m^\star=2$, and $\cU$ be lower triangular (i.e. $\cU_{12}=0$). Then in the moderate deviations setting $v$ and $u$ are explicit from~\eqref{eq:multirBMDP}:
\begin{align*}
&v^{(1)} = \frac{1}{\cU_{11}} \DD^{H+\half}\left(\vphi^{(1)}\right), 
\qquad v^{(2)} = \frac{1}{\cU_{22}} \DD^{H+\half}\left(\vphi^{(2)}\right) - \frac{\cU_{21}}{\cU_{22}} v^{(1)},\\
& u = \frac{1}{\rrho} \left[\sqrt{2}\,\dot\phi \left\{\exp\left(y_0^{(1)}\right) +\exp\left(y_0^{(2)}\right)\right\}^{-\half} -\rho_1 v^{(1)} - \rho_2 v^{(2)}
\right].
\end{align*}
\end{example}
\begin{remark}
We can similarly consider multidimensional versions of the other models presented in this chapter and derive large and moderate deviation principles. We only work out the computations for the multi-factor rough Bergomi model because it is the most relevant in the literature.
\end{remark}
%%%%%%%%%%%%%%%%%%%%%%%%%%%%%%%%%%%%%%%
%%%%%%%%%%%%%%%%%%%%%%%%%%%%%%%%%%%%%%%%

%%%%%%%%%%%%%%%%%%%%%%%%%%%%%%%%%%%%%%%%%%%%%%%%%%%%%%%%%%%%%%%%%%%%%%%%%%%%%%%%%%%%%%%%%%%%%%%%%%%%%%%%%%%%%%%%%%%%%%%%%%%%%%%%%%%%%%%%%%%%%%%%%%
\subsection{Tail rescaling}\label{sec:TailsEx}
We now investigate tail rescalings, which generally have the form $X^\ep=\ep X$, 
such that an LDP provides asymptotic estimates on $ \bP(X^\ep \ge 1) = \bP(X \ge \ep^{-1})$. 
The MDP for the whole system is not available in this case because~$\overline{Y}:=\lim_{\ep\downarrow0} Y^\ep \equiv0$ hence the limit equation for~$X^\ep$, arising from~\eqref{eq:MDPlimit}, would be independent of the control. Note that the theory does not break down but the rate function is trivial (equals zero at zero and~$+\infty$ everywhere else). Furthermore, the exponential function prevents the study of such a rescaling in the rough Bergomi model. 
%%%%%%%%%%%%%%%%%%%%%%%%%%%%%%%%%%%%%%%
\subsubsection{Rough Stein-Stein}
This model was defined in Section~\ref{sec:rSS}, but with the rescaling $Y^\ep_t:=\ep Y_{t}$ and~$X^\ep_t:= \ep^2 X_{t}$, the system becomes
\begin{align}
\label{eq:rSStail}
\left\{
\begin{array}{rl}
X^\ep_t &= \displaystyle - \half\int_0^t  (Y^\ep_s)^2 \ds + \ep \int_0^t Y^\ep_s \D B_s, \\
Y^\ep_t &= \displaystyle \ep y_0 + \int_0^t \kappa \left(\ep \theta - Y^\ep_s\right)\ds + \ep \int_0^t \xi \frac{(t-s)^{H-\half}}{\Gamma(H+\half)} \D W_s,
    \end{array}
    \right.
\end{align}
where the coefficients are identical to the small-time case. Although the rescaling is different, Assumption~\ref{assu:K} and {\bf H1 - H4} are easily satisfied in a similar way, the limit equation~\eqref{eq:LDPlimit} has a unique solution, and therefore Theorem~\ref{thm:LDP} applies. 
\begin{corollary}\label{cor:TailSS}
The following hold:
\begin{enumerate}[{\bf (L1)}]
    \item $(X^\ep,Y^\ep)\sim\LDP\left(I,\ep^{-2}\right)$ with
    \begin{equation*}
I(\phi,\vphi)= \half \int_0^T \left(u_t^2 + v_t^2\right) \dt, \quad \text{where}\quad
\left\{
\begin{array}{rl}
u & = \displaystyle \frac{1}{\rrho} \left( \frac{\dot\phi}{\vphi} + \half\vphi - \rho v \right)\one_{\vphi\neq0},\\
v & = \displaystyle \frac{1}{\xi} \left(\DD^{H+\half}(\vphi) + \kappa \II^{\half-H}(\vphi) \right),
\end{array}
\right.
\end{equation*}
if $\phi\in\Aco,\, \vphi\in \II^{H+\half}_0$ and infinity otherwise.
    \item $X^\ep\sim\LDP\left(I^X,\ep^{-2}\right)$ with
    $I^X(\phi)=\inf \big\{ I(\phi,\vphi): \vphi\in \II^{H+\half}_0 \big\}$ if~$\phi\in\Ac_0$ and infinity otherwise.
    \item For each $t\in\bT$, $\ep^2 X_t\sim\LDP\left(I^X_t,\ep^{-2}\right)$, where
    $I^X_t(x)=\inf \big\{ I^X(\phi): \phi_t=x \big\}$.
\end{enumerate}
\label{coro:tailSS}
\end{corollary}
\begin{proof}
For {\bf (L1)}, Theorem~\ref{thm:LDP} entails that the rate function is
\begin{align*}
I(\phi,\vphi)=& \inf \bigg\{ \half \int_0^T \left(u_t^2 + v_t^2\right) \dt : \quad u,v\in L^2, \\
& \phi_t = -\half\int_0^t \vphi_s^2\ds + \int_0^t  \vphi_s \big( \rrho u_s + \rho v_s \big) \ds , \,
\vphi_t = -\int_0^t \kappa \, \vphi_s\ds  + \int_0^t \xi\, \frac{(t-s)^{H-\half}}{\Gamma(H+\half)} v_s\ds
\bigg\}.
\end{align*}
Inverting the integrals as in Corollary~\ref{coro:explicitrate} to obtain the unique controls and using~$\DD^{H+\half} \II^1= \II^{\half-H}$ yields the claim. Similarly to the small-time case, {\bf (L2)} follows from the contraction principle, 
and one only needs to fix $t\in\bT$ to prove {\bf (L3)}.
\end{proof}

One can prove an LDP for $Y^\ep$ in a similar way; a more interesting problem is the moderate deviations setting. Recall that an MDP for the couple~$(X^\ep,Y^\ep)$ would have a trivial rate function because the limit equation of $X^\ep$ is independent of the control. 
However, since the diffusion coefficient of $Y^\ep$ is constant equal to~$\xi$, one can obtain an MDP for $Y^\ep$. More surprisingly, the limit equations in the large deviations~\eqref{eq:LDPlimit} and moderate deviations~\eqref{eq:MDPlimit} regimes coincide, which leads to identical rate functions.
Notice that~$\vartheta_\ep =\ep$ in this example, and therefore let~$h_\ep  =\ep^{-\beta}$ where~$\beta\in(0,1)$.  
\begin{corollary}
    $Y^\ep\sim\MDP\left(\Lambda^Y,\ep^{-2\beta}\right)$ where
    \begin{align*}
\Lambda^Y(\vphi)= \frac{1}{2\xi^2} \int_0^T\Big( \DD^{H+\half}(\vphi)(t) + \kappa \II^{\half-H} (\vphi)(t) \Big)^2 \dt,
\end{align*}
if~$\vphi\in \II^{H+\half}_0$ and infinity otherwise.
\end{corollary}
\begin{proof}
From~\eqref{eq:rSStail}, $b_\ep(y)=\kappa( \ep\theta -y)$ converges to $b(y)=-\kappa y$ and the diffusion coefficient is constant, hence {\bf H2 - H6} are easily satisfied. Moreover, $b_\ep - b \equiv \ep \kappa\theta$ and $\ep^{1-(H-\beta)}$ tends to zero therefore {\bf H7} and {\bf H8} also hold. Theorem~\ref{thm:MDP} thus yields an MDP with rate function
\begin{align*}
\Lambda^Y(\vphi)=& \inf \bigg\{ \half \int_0^T v_t^2 \dt : v\in L^2, 
\, \vphi_t = -\int_0^t  \kappa \vphi_s\ds  + \xi\int_0^t \frac{(t-s)^{H-\half}}{\Gamma(H+\half)} v_s\ds
\bigg\}.
\end{align*}
Inverting the integral yields the claim.
\end{proof}

%%%%%%%%%%%%%%%%%%%%%%%%%%%%%%%%%%%%%%%%%%%%%%%%%%%%%%%%
\subsubsection{Rough Heston}
After the rescaling~$Y^\ep_t:=\ep^2 Y_{t}$ and~$X^\ep_t:= \ep^2 X_{t}$, this model introduced in Section~\ref{sec:rHeston} takes the form
\begin{align*}
\left\{
\begin{array}{rl}
    X^\ep_t &= \displaystyle  - \half\int_0^t  Y^\ep_s \ds + \ep \int_0^t \sqrt{Y^\ep_s} \D B_s, \\
    Y^\ep_t &= \displaystyle \ep^2 y_0 +  \int_0^t \kappa \frac{(t-s)^{H-\half}}{\Gamma(H+\half)} \left(\ep^2 \theta - Y^\ep_s\right)\ds
     + \ep  \int_0^t \xi \frac{(t-s)^{H-\half}}{\Gamma(H+\half)} \sqrt{Y^\ep_s} \D W_s.
\end{array}
\right.
\end{align*}
Clearly {\bf H1 - H3} hold and we recall that~$\cU$ is the set of coefficients such that pathwise uniqueness, and hence {\bf H4}, hold.
We appeal to the uniqueness relaxation in the same way as the small-time case to prove the following result, which extends~\cite[Theorem 1.1]{CDMD15} to the rough case.
\begin{corollary}
\label{coro:tailheston}
If the rough Heston coefficients belong to $\cU$ and~$H\in(0,\half)$, then the following hold
\begin{enumerate}[{\bf (L1)}]
    \item $(X^\ep,Y^\ep)\sim\LDP\left(I,\ep^{-2}\right)$ where
\begin{equation*}
I(\phi,\vphi)= \half \int_0^T \left(u_t^2 + v_t^2\right) \dt, \quad \text{where} 
\left\{
\begin{array}{rl}
u & = \displaystyle \frac{1}{\rrho} \left( \frac{\dot\phi}{\sqrt\vphi} + \half \sqrt\vphi - \rho v \right)\one_{\vphi>0},\\
v & = \displaystyle \frac{1}{\xi\sqrt{\vphi}} \left(\DD^{H+\half}(\vphi) + \kappa \vphi \right) \one_{\vphi>0},
\end{array}
\right.
\end{equation*}
where $\phi\in\Ac_0$ and~$\vphi\in\II_0^{H+\half}(L^1)$ and infinity otherwise.
    \item $X^\ep\sim\LDP\left(I^X,\ep^{-2}\right)$ with
    $I^X(\phi)=\inf \big\{ I(\phi,\vphi): \vphi\in\II_0^{H+\half} \big\}$ if~$\phi\in\Ac_0$ and infinity otherwise.
    \item For each $t\in\bT$, $\ep^2 X_t\sim\LDP\left(I^X_t,\ep^{-2}\right)$, where
    $I^X_t(x)=\inf \big\{ I^X(\phi): \phi_t=x \big\}$.
\end{enumerate}

\end{corollary}
\begin{proof}
The proof is similar to the small-time case. The potential rate function for the couple is   
\begin{align*}
I(\phi,\vphi)=& \inf \bigg\{ \half \int_0^T \left(u_t^2 + v_t^2\right) \dt:\quad u,v\in L^2, \\
& \phi_t = -\half\int_0^t \vphi_s\ds + \int_0^t  \sqrt{\vphi_s} \big( \rrho u_s + \rho v_s \big) \ds, \quad
\vphi_t = \int_0^t \frac{(t-s)^{H-\half}}{\Gamma(H+\half)}\Big( -\kappa \vphi_s + \xi \sqrt{\vphi_s} v_s \Big) \ds
\bigg\}.
\end{align*}
The same arguments that were used to prove Lemmata~\ref{lemma:lebesguezero} and~\ref{lemma:heston} in the small-time case can be applied again here.
They entail that Assumption~\ref{assu:relax} holds and hence Theorem~\ref{thm:LDP} applies,
and the form of the rate function in {\bf (L1)} follows by inverting the relationships between $(u,v)$ and $(\phi, \varphi)$.
{\bf (L2)} and {\bf (L3)} follow from the same steps as in Corollary~\ref{coro:tailSS}.
\end{proof}

%%%%%%%%%%%%%%%%%%%%%%%%%%%%%%%%%%%%%%%%%%%%
\subsubsection{Implied volatility asymptotics}
We can also obtain implied volatility asymptotics since, by the same arguments as before, $\exp(X^\ep)$ is a martingale in both the rough Stein-Stein and rough Heston models.
\begin{corollary}
\label{coro:tailIV}
In both the rough Stein-Stein and rough Heston models, for each $t\in\bT$, the implied volatility~$\widehat \sigma$ satisfies
\[
\lim_{k\uparrow\infty} \frac{\widehat{\sigma}(t,k)^2 t}{k} = \frac{1}{2} \left( \inf_{y\ge1} I^X_t(y) \right)^{-1},
\]
where~$I^X_t$ is the respective rate function, given in Corollaries~\ref{cor:TailSS}{\bf (L3)} and~\ref{coro:tailheston}{\bf (L3)}.
\end{corollary}
\begin{proof}
Mapping $\ep^2$ to $1/k$ we have from Corollaries~\ref{coro:tailSS} and~\ref{coro:tailheston} respectively that, for each $t\in\bT$,
\[
\lim_{k\uparrow\infty} \frac{1}{k} \log \bP( X_t \ge k) = -\inf_{y\ge1} I^X_t(y).
\]
In the Black-Scholes model with constant volatility $\sigma>0$, we can directly compute
\[
\lim_{k\uparrow\infty}  \frac{1}{k^2}\log \bP(X_t\ge k) = -\frac{1}{2\sigma^2 t},
\]
and, similarly to the small-time case, the proof follows from~\cite[Corollary 7.1]{GL14}.
\end{proof}

%%%%%%%%%%%%%%%%%%%%%%%%%%%%%%%%%%%%%%%%%%%%%%%%%%%%%%%%%%%%%%%%%%%%%%%%%%%%%%%%%%%%%%%%%%%%%%%%%%%%%%%%%%%%%%%%%%%%%%%%%%%%%%%%%%%%%%%
\appendix
\section{Technical large deviations proofs}
\subsection{Abstract relaxation: Proof of Theorem~\ref{thm:relax}}
\label{app:proof:thm:relax}
The proof follows~\cite[Theorem 4.4]{BD01}. The lower bound proof stands as it is until the last series of inequalities. 
%Let and define the subsequence~$\{y_\ep\}_{\ep>0}$ as $y_\ep:=\inf_{\alpha < \ep} x_\alpha$, such that $\liminf_{\ep\downarrow0} x_\ep = \lim_{\ep\downarrow0}y_\ep$. By definition, $\{y_\ep\}$ has a limit in~$[-\infty,+\infty]$. Therefore if $\{y_\ep\}$ has a subsequence converging to some limit $y_0\in\bR$, then $\lim_{\ep\downarrow0}y_\ep = y_0$.
For any sequence $\{v^\ep\}_{\ep>0}$ in $\cA_N,N>0$ converging weakly to $v$ and any~$F\in\cC_b(\cW^d:\bR)$, define the function~$x:\bR_+\to\bR$ by~$x(\ep) := \bE\left[ F\circ\cG^{\ep} \left( W + \ep^{-1} \int_0^\cdot v^\ep_t \dt\right) \right]$. Theorem~\ref{thm:relax}(i) entails the existence of~$\phi\in\cG^0_{v,N}$ such that the subsequence~$y:\bR_+\to\bR$ defined as~$y(\ep):=\inf_{\alpha < \ep} x(\alpha)$ has a subsequence converging to~$\bE\left[ F(\phi)\right]$. By definition $\liminf_{\ep\downarrow0} x_\ep = \lim_{\ep\downarrow0}y_\ep$, which implies that $y$ has a limit in~$[-\infty,+\infty]$ and by uniqueness this limit must be~$\bE\left[ F(\phi)\right]$. Therefore we deduce:
\begin{align*}
    \liminf_{\ep\downarrow0} \bE &\left[ \half\int_0^T \abs{v^\ep_t}^2\dt + F\circ\cG^{\ep} \left( W + \ep^{-1} \int_0^\cdot v^\ep_t \dt\right) \right] \\
    & \ge \bE \left[ \half\int_0^T \abs{v_t}^2\dt + F(\phi) \right]
    \ge \inf \bigg\{ \half\int_0^T \abs{v_t}^2\dt + F(\phi) : v\in L^2 \text{ such that } \, \phi\in\cG^0_v\cap\cW^d\bigg \},
\end{align*}
which suggests the potential rate function~$I$ defined in~\eqref{eq:raterelax} and concludes the proof of the lower bound.

Then we prove the Laplace principle upper bound, for all~$F\in\cC_b(\cW^d:\bR)$:
\[
\limsup_{\ep\downarrow0} - \ep^2 \log \bE\left[\E^{-F\circ \cG^\ep(W)/{\ep^2}} \right] \le \inf_{\psi\in\cW^d} \{ I(\psi)+F(\psi) \}.
\]
Assume that the right-hand side is finite otherwise there is nothing to prove. Fix $\epsilon>0$ and let~$\phi\in\cW^d$ such that
\[
I(\phi)+F(\phi) \le \inf_{\psi\in\cW^d} \{ I(\psi)+F(\psi) \} + \epsilon.
\]
Since $F$ is continuous at~$\phi$, there exists $\delta\in(0,\epsilon)$ such that $\abs{F(\phi)-F(\vphi)}\le \epsilon$ for all~$\vphi\in\cW^d$ such that $\normbT{\phi-\vphi}\le \delta$.
If~$\phi$ is uniquely characterised then the proof is the same as in~\cite{BD01}. 
Otherwise, by Theorem~\ref{thm:relax}(iii), we can choose~$\phi^\delta$ uniquely characterised such that
$\normbT{\phi-\phi^\delta} \le \delta$
and
$\abs{I(\phi)-I(\phi^\delta)} \le \delta$,
which implies
$\abs{I(\phi) + F(\phi) - I(\phi^\delta) - F(\phi^\delta)} \le 2 \epsilon$.
Hence, combining inequalities we obtain
\[
I(\phi^\delta)+F(\phi^\delta) \le \inf_{\psi\in\cW^d} \{ I(\psi)+F(\psi)\} + 3\epsilon.
\]
Moreover, there exist~$\{ v^n\}_{n\in\bN}$ in~$L^2$ such that~\eqref{eq:condUB} is satisfied with~$\phi^\delta$ and $m\ge 1/\epsilon$ such that
\[
\cG^0_{v^m}=\{\phi^\delta\} \quad \text{and} \quad \half \int_0^T \abs{{v}^m_t}^2\dt \le I(\phi^\delta)+\frac{1}{m} \le I(\phi^\delta)+\epsilon,
\]
and therefore the remainder of the upper bound proof unfolds identically. 

Along the subsequence~$\{\ep_n\}_{n\ge0}$,~$\cG^{\ep_n}( W + \ep_n^{-1} \int_0^\cdot  v^m_t \dt)$ converges in distribution to~$\phi^\delta$ by item (i). Using the variational representation formula~\eqref{eq:boue} and the convergence we obtain
\begin{align*}
    \limsup_{n\uparrow\infty} - \ep_n^2 \log \bE\left[\exp\left\{-\frac{F\circ \cG^{\ep_n}(W)}{\ep_n^2}\right\} \right] 
    & = \limsup_{n\uparrow\infty} \inf_{v\in\cA} \bE \left[ \half\int_0^T \abs{v_t}^2\dt + F\circ\cG^{\ep_n} \left( W + \ep_n^{-1} \int_0^\cdot v_t \dt\right) \right] \\
    & \le \limsup_{n\uparrow\infty} \bE \left[ \half\int_0^T \abs{v^m_t}^2\dt + F\circ\cG^{\ep_n} \left( W + \ep_n^{-1} \int_0^\cdot v^m_t \dt\right) \right] \\
    & = \half \int_0^T \abs{v^m_t}^2 \dt + F(\phi^\delta) \\
    & \le I(\phi^\delta)+F(\phi^\delta)+ \epsilon
    \le \inf_{\psi\in\Omega} \{ I(\psi)+F(\psi)\} +4 \epsilon.
\end{align*}
Since~$\epsilon>0$ is arbitrary this concludes the proof.

%%%%%%%%%%%%%%%%%%%%%%%%%%%%%%%%%%%%%%%%%%%%%%%%%%%%
%%%%%%%%%%%%%%%%%%%%%%%%%%%%%%%%%%%%%%%%%%%%%%%%%%%%

\subsection{LDP moment bounds: Proof of Lemma~\ref{lemma:boundv}}
\label{app:proof:lemma:boundv}
Let us fix~$p\ge2$, $N>0$, $v\in\cA_N$, $\ep>0$ and $t\in\bT$.
Let~$\tau_n := \inf \{ t\ge0: \abs{X_t^{\ep,v}}\ge n \}\wedge T$ for all~$n\in\bN$. For clarity we write~$b_s^n:=b_\ep(s,X^{\ep,v}_s  \one_{s\le\tau_n})$ 
and~$\sigma_s^n:=\sigma_\ep(s,X^{\ep,v}_s  \one_{s\le\tau_n})$.
We start by assuming that all the coefficients satisfy the linear growth condition~{\bf H3a}. We fix~$n\in\bN$ and observe that, almost surely:
\begin{align}
\label{eq:ineqtaun}
    \abs{X^{\ep,v}_{t}}^p  \one_{t\le\tau_n}
    \le & 4^{p-1} \bigg[ \abs{X^\ep_0}^p + \abs{\int_0^t K(t-s)b^n_{s}\ds}^p
    + \abs{\int_0^t K(t-s)\sigma_{s}^n v_s \ds}^p  
    + \vartheta_\ep ^p \abs{\int_0^t K(t-s)\sigma_{s}^n \D W_s}^p
    \bigg] \nonumber \\
    =: & 4^{p-1} \Big[ \abs{X_0^\ep}^p + \text{I}_n
    + \text{II}_n + \text{III}_n \Big], 
\end{align}
because if~$t>\tau_n$ then the left-hand side is zero while the right-hand side is non-negative, and if~$t\le\tau_n$ then~$s\le\tau_n$ for all~$s\in[0,t]$ and the~$\tau_n$ dependence vanishes on both sides of the inequality.
For~$\ep$ small enough we can bound~$\abs{X_0^\ep}$ by~$2\abs{X_0}$ and~$\vartheta_\ep $ by~$1$ and we will do so repetitively in the sequel.
Using H\"older's and Jensen's inequalities, we obtain the following estimates almost surely:
\begin{equation}
    \text{I}_n
     \le \left[\int_0^t \abs{K(t-s)}^{\frac{4}{p}} \abs{b_s^n}^2 \ds \right]^{\frac{p}{2}} 
     \left[\int_0^t \abs{K(t-s)}^{2-\frac{4}{p}}  \ds \right]^{\frac{p}{2}}
     \le t^{\frac{p^2+4}{2(p-2)}} \norm{K}_{2}^{p-2} \int_0^t \abs{K(t-s)}^2 \abs{b_s^n}^p \ds,
\label{eq:Holdertrickb}
\end{equation}
and 
\begin{equation}
    \text{II}_n
    \le N^{\frac{p}{2}} \left[ \int_0^t \abs{K(t-s)}^{2-4/p} \abs{K(t-s)}^{\frac{4}{p}} \abs{\sigma_s^n}^2 \ds \right]^{\frac{p}{2}}
    \le N^{\frac{p}{2}} \norm{K}_{2}^{p-2} \int_0^t \abs{K(t-s)}^2 \abs{\sigma_s^n}^p \ds,
\label{eq:Holdertrick}
\end{equation}
where we also used that~$\int_0^T \abs{v_s}^2 \ds \le N$ almost surely.
Notice that for fixed~$t\in\bT$, $\{ \int_0^u K(t-s)\sigma_s^n\D W_s:  u\in[0,t]\}$ is a continuous local martingale and is bounded in~$L^2(\Omega)$. Hence, using Burkholder-Davis-Gundy (BDG) inequality and similar calculations as~\eqref{eq:Holdertrick} there exists~$C_p>0$ such that
\[
\bE[\text{III}_n] \le C_p \norm{K}_{2}^{p-2} \int_0^t \abs{K(t-s)}^2 \bE\abs{\sigma_s^n}^p \ds.
\]
From the linear growth condition on~$b_\ep$ and~$\sigma_\ep$ (uniform in~$\ep>0$) we deduce that there exists $C_1>0$ independent of~$\ep,v,n,t$ such that, for all~$n\in\bN$, $f^n_t:=\bE \left[ \abs{X^{\ep,v}_t}^p  \one_{t\le\tau_n} \right]$ satisfies the inequality
\begin{equation}
f_t^n \le C_1 + C_1 \int_0^t \abs{K(t-s)}^2  f_s^n\ds.
\label{eq:gronwall}
\end{equation}
The following lemma (Lemma~\ref{lemma:gronwallK}) yields a uniform bound in both~$n\in\bN$ and~$t\in\bT$ for~$f^n_t$. Taking the limit as~$n$ goes to infinity and using Fatou's lemma concludes the first part of the proof.
\begin{lemma}
\label{lemma:gronwallK}
Let~$f:\bT \to \bR_+$ and~$K$ a kernel satisfying Assumption~\ref{assu:K}. 
If there exists~$c_1,c_2\ge0$ such that
\[
f(t) \le c_1 + c_2 \int_0^t \abs{K(t-s)}^2 f(s)\ds,
\qquad\text{for all }t\in\bT,
\]
then~$f$ is uniformly bounded on~$\bT$ by a constant depending only on~$c_1,c_2,\norm{K}_2, T$. 
If~$c_1=0$ then~$f=0$.
\end{lemma}
\begin{proof}[Proof of Lemma~\ref{lemma:gronwallK}]
By definition
$\abs{K(t-s)}^2 = \sum_{i,j=1}^d \abs{K_{ij}(t-s)}^2$,
and~$\widetilde K(t,s) := c_2 \abs{K(t-s)}^2 \mathbbm{1}_{s\le t}$ is a Volterra kernel in the sense of~\cite[Definition~9.2.1]{Gripenberg90}. Following similar arguments as in the proof of~\cite[Lemma 3.1]{ALP17}, the generalised Gronwall lemma~\cite[Theorem 9.8.2]{Gripenberg90}
yields the bound
\[
f(t) \le c_1 - c_1\int_0^t \widetilde R(s)\ds \le c_1 - c_1\int_0^T \widetilde R(s)\ds,
\]
where~$\widetilde R$ is the (non-positive) resolvent of second kind of~$-\widetilde K$~\cite[Equation (2.11)]{ALP17},
proving the lemma.
\end{proof}

If only {\bf H3b} and Assumption~\ref{assu:Gammabound} hold with an autonomous sub-system~$\Upsilon$ (see Definition~\ref{def:autonomous}), then by the previous calculations for all~$l\in\Upsilon$, the components~$(X^{\ep,v})^{(l)}$ satisfy the bound~\eqref{eq:boundv} because their coefficients have linear growth. Then we turn our attention to the components~$(X^{\ep,v})^{(i)}$,~$i\notin\Upsilon$. Using~\eqref{eq:boundgsto} and H\"older's inequality as in~\eqref{eq:Holdertrick}, we obtain that for all $1\le j \le d$ such that~$K_{ij}\neq0$ and for all~$1\le k \le m$:
\begin{align}
    \bE & \left[ \left( \int_0^t \abs{K_{ij}(t-s) \sigma_\ep^{(jk)}\big(s,X^{\ep,v}_{s}\one_{s\le\tau_n}\big)}^2 \ds \right)^{p/2} \right] \nonumber\\
    & \le C_\Upsilon^p \bE \left[ \left( \int_0^t \abs{K_{ij}(t-s)}^2 \Big(1+ \Big\lvert X^{\ep,v}_{s}\one_{s\le\tau_n}\Big\lvert_{\Upsilon^c} + \abs{\Gamma \big( (X^{\ep,v}_{s}\one_{s\le\tau_n})^{(\Upsilon)} \big)} \Big)^2 \ds \right)^{p/2} \right] \nonumber\\
    & \le C_\Upsilon^p \bE \left[ \int_0^t \abs{K_{ij}(t-s)}^2 \Big(1+\abs{X^{\ep,v}_{s}\one_{s\le\tau_n}}_{\Upsilon^c} +  \abs{\Gamma \big( (X^{\ep,v}_{s}\one_{s\le\tau_n})^{(\Upsilon)} \big)} \Big)^p \ds \right] 
    \left(\int_0^t \abs{K_{ij}(t-s)}^2 \ds
    \right)^{p/2-1} \nonumber\\
    & \le 3^{p-1} C_\Upsilon^p \norm{K}_2^{p-2}
     \int_0^t \abs{K_{ij}(t-s)}^2 \bE \bigg[ 1+\Big\lvert X^{\ep,v}_{s}\one_{s\le\tau_n} \Big\lvert_{\Upsilon^c}^p + \abs{\Gamma \big( (X^{\ep,v}_{s}\one_{s\le\tau_n})^{(\Upsilon)}\big)}^p \bigg] \ds  \nonumber\\
    & \le C_2 + C_2 \int_0^t \abs{K_{ij}(t-s)}^2 \bE \left[ \abs{X^{\ep,v}_{s}\one_{s\le\tau_n}}_{\Upsilon^c}^p\right] \ds,
\label{eq:ineqg}
\end{align}
for some~$C_2>0$. Applying the same calculations to the other terms and summing all the coefficients we fall back on~\eqref{eq:gronwall}. Taking the limit and applying Fatou's lemma again conclude the proof.

%%%%%%%%%%%%%%%%%%%%%%%%%%%%%%%%%%%%%%%%%%%%%%%%%%%%
%%%%%%%%%%%%%%%%%%%%%%%%%%%%%%%%%%%%%%%%%%%%%%%%%%%%

\subsection{LDP tightness: Proof of Lemma~\ref{lemma:tightLDP}}
\label{app:proof:lemma:tightLDP}
Let us fix~$p> 2 \vee 2/\gamma$, $N>0$, a family $\{v^\ep\}_{\ep>0}$ in~$\cA_N$ and $\ep>0$.
For clarity we will write~$b_u:=b_\ep(u,X^{\ep,v}_u)$ and~$\sigma_u:=\sigma_\ep(u,X^{\ep,v}_u)$ for all~$u\in\bT$. Then, for 
all~$0\le s<t\le T$, using Cauchy-Schwarz and BDG inequalities as in the previous proof we obtain:
\begin{align*}
    \bE \left[ \abs{ X^{\ep,v^{\ep}}_t -X^{\ep,v^{\ep}}_s}^p \right] 
    \le &  \, 6^{p-1} \bE \left[ \abs{ \int_0^s \big(K(t-u) - K(s-u) \big) b_u \du}^p \right]  \\
    & + 6^{p-1} \bE \left[\abs{ \int_s^t K(t-u) b_u \du }^p \right] \\
    & + 6^{p-1} N^{p/2} \bE \left[ \left( \int_0^s \abs{ \big(K(t-u) - K(s-u) \big) \sigma_u }^2 \du
    \right)^{p/2} \right] \\
    & + 6^{p-1} N^{p/2} \bE \left[ \left( \int_s^t \abs{ K(t-u) \sigma_u }^2 \du
    \right)^{p/2} \right] \\
    & + 6^{p-1} \vartheta_\ep^p  C_p \bE \left[ \left( \int_0^s \abs{ \big(K(t-u) - K(s-u) \big) \sigma_u }^2 \du
    \right)^{p/2} \right] \\
    & + 6^{p-1} \vartheta_\ep^p  C_p \bE \left[ \left( \int_s^t \abs{ K(t-u) \sigma_u }^2 \du
    \right)^{p/2} \right].
\end{align*}
In a first step we assume that all the coefficients satisfy the linear growth condition {\bf H3a}.
Analogous calculations to the proof of Lemma~\ref{lemma:boundv}, 
bounds on~$\sup\limits_{t\in\bT, \ep>0} \bE\left[\big\lvert X^{\ep,v^\ep}_t\big\lvert^{p}\right]$, linear growth from~{\bfseries H3a} and Assumption~\ref{assu:K} lead to 
\begin{align*}
    \bE  \left[\abs{X^{\ep,v^\ep}_t- X^{\ep,v^\ep}_s}^p \right]
    \le & \,C_1 \left( \abs{ \int_0^s \big(K(t-u) - K(s-u) \big)^2 \du }^{p/2} + \abs{\int_s^t K(t-u)^2 \du }^{p/2}   \right) \\
    = & \, C_1 \left( \abs{ \int_0^s \big(K(u+t-s) - K(u) \big)^2 \du }^{p/2} + \abs{\int_0^{t-s} K(u)^2 \du }^{p/2}   \right) \\
    \le & \, C_2 (t-s)^{\gamma p/2},
\end{align*}
for some $C_1,C_2>0$ independent of~$\ep,v^\ep,s,t$.
Again, if there are components such that only {\bf H3b} holds with Assumption~\ref{assu:Gammabound} then following the example of~\eqref{eq:ineqg} yields the same result.
Then Kolmogorov continuity theorem asserts that ~$X^{\ep,v^\ep}$ admits a version which is H\"older continuous on~$\bT$ of any order~$\alpha < \gamma/2 -1/p$, uniformly in~$ \ep>0$ because~$C_2$ does not depend on~$\ep$, and which satisfies~\eqref{eq:holderLDP}. Furthermore, Aldous theorem~\cite[Theorem 16.10]{Billingsey99} states that the sequence~$\{ X^{\ep,v^\ep}\}_{\ep>0}$ is tight.

%%%%%%%%%%%%%%%%%%%%%%%%%%%%%%%%%%%%%%%%%%%%%%%%%%%
%%%%%%%%%%%%%%%%%%%%%%%%%%%%%%%%%%%%%%%%%%%%%%%%%%%

\subsection{LDP compactness: Proof of Lemma~\ref{lemma:goodnessLDP}}
\label{app:proof:lemma:goodnessLDP}
We prove that for all~$N>0$, the sublevel sets
\[
L_N:=\{\phi\in\cW^d: I(\phi)\le N\}
\]
of the map~$I:\cW^d\to\bR$ given by~\eqref{eq:raterelax} or more precisely by
\begin{equation}
I(\phi)=\inf\bigg\{\half \int_0^T \abs{v_s}^2\ds:  v\in  L^2,  \phi_t= x_0 + \int_0^t K(t-s)\Big[b(s,\phi_s) + \sigma(s,\phi_s) v_s \Big] \D s   \bigg \}
\label{eq:rfcompactness}
\end{equation}
are compact. Fix~$N>0$ and consider an arbitrary sequence~$\cJ:=\{\phi^n\}_{n\in\bN}\subset L_N$; we will show that there exists a converging subsequence the limit of which belongs to~$L_N$. Interestingly enough, the proof parallels, in a deterministic context, the proofs of bound, H\"older continuity and convergence of~$X^{\ep,v}$.

{\bfseries Relative compactness.}
According to Arzelà-Ascoli's theorem, the family~$\cJ$  is relatively compact in~$\cW^d$ if and only if~$\{\phi^n_t\}$ is bounded uniformly in~$n\in\bN$ and in~$t\in\bT$ and~$\cJ$ is equicontinuous.
Moreover, for all~$n\in\bN$ and all~$t\in\bT$, there exists~$v^n \in L^2$ such that~$\half\int_0^T \abs{v^n_t}^2\dt \le N$ and~$\phi^n\in\cG^0_{v^n}$, which means~$v^n\in\cS_{2N}$ and
\[
\phi_t^n= x_0 + \int_0^t K(t-s)\Big[b(s,\phi^n_s) + \sigma(s,\phi^n_s) v^n_s \Big] \D s.
\]
Hence Remarks~\ref{rem:boundphi} and~\ref{rem:tightphi} grant the uniform bound and equicontinuity respectively. Therefore~$\cJ$ is relatively compact which entails that $L_N$ is relatively compact for any~$N>0$.

{\bfseries Closure.}
Let~$\{\phi^n\}_{n\in\bN}$ be a converging sequence of~$L_N$ and denote its limit by~$\phi\in\cW^d$.
The controls~$v^n$ associated to~$\phi^n$ through~\eqref{eq:rfcompactness} belong to~$\cS_{2N}$ which is a compact space with respect to the weak topology. Hence there exists a subsequence~$\{n_k\}_{k\in\bN}$ such that~$v^{n_k}$ converges weakly in~$ L^2$ to a limit~$v\in \cS_{2N}$ and~$\lim_{k\uparrow\infty} \phi^{n_k}=\phi$. Now let us prove that~$\phi\in L_N$. For clarity we replace~$n_k$ by~$n$ from now on. 
The convergence as~$n$ goes to~$+\infty$ and the continuity of the paths entail
\[
\sup_{n\in\bN} \sup_{t\in\bT} \big( \abs{\phi^n_t} + \abs{\phi_t} \big) <+\infty,
\]
such that the paths lie in compact subsets of~$\bR^d$ and {\bf H2} asserts that uniform continuity of the coefficients~$b$ and~$\sigma$ hold. Therefore they admit continuous moduli of continuity that we respectively name~$\rho_b$ and~$\rho_\sigma$.
Using Cauchy-Schwarz inequality and {\bf H3} we get that for all~$t\in\bT$:
\begin{align*}
    & \abs{ \int_0^t K(t-s) b(s,\phi^n_s) \D s  - \int_0^t K(t-s) b(s,\phi_s) \D s }
    \le \norm{K}_{ L^1} \normbT{\rho_b \big( \abs{\phi^n_\cdot-\phi_\cdot} \big) }, \\ 
    & \abs{ \int_0^t K(t-s) \sigma(s,\phi^n_s) v^n_s \D s  - \int_0^t K(t-s) \sigma(s,\phi_s) v_s \D s } \\
    & \quad \le \int_0^t \abs{ K(t-s) \big(\sigma(s,\phi^n_s)-\sigma(s,\phi_s) \big) v^n_s } \D s 
    + \int_0^t \abs{ K(t-s) \sigma(s,\phi_s) (v^n_s -v_s)} \D s \\
    & \quad \le \norm{K}_2 \norm{v^n}_2 \normbT{ \rho_\sigma( \abs{\phi^n_\cdot - \phi_\cdot}) } +   \normbT{\sigma(\phi)}\int_0^t \abs{ K(t-s) (v^n_s -v_s)} \D s,
\end{align*}
and both estimates converge towards zero as~$n$ tends to infinity. Therefore, for all~$t\in\bT$
\begin{align*}
\phi_t = \lim_{n\uparrow\infty} \phi^n_t 
&=   \lim_{n\uparrow\infty} \bigg( x_0 + \int_0^t K(t-s) \Big[ b(s,\phi^n_s) + \sigma(s,\phi^n_s) v^n_s \Big] \D s \bigg) \\
&= x_0 + \int_0^t K(t-s) \Big[ b(s,\phi_s) + \sigma(s,\phi_s) v_s \Big] \D s,
\end{align*}
so that~$\phi\in L_N$ since~$v\in \cS_{2N}$, which concludes the proof of the closure and therefore of the compactness of~$L_N$.

%%%%%%%%%%%%%%%%%%%%%%%%%%%%%%%%%%%%%%%%%%%%%%%%%%%%%
%%%%%%%%%%%%%%%%%%%%%%%%%%%%%%%%%%%%%%%%%%%%%%%%%%%%%
\section{Technical moderate deviations proofs}
%%%%%%%%%%%%%%%%%%%%%%%%%%%%%%%%%%%%%%%%%%%%%%%%%%%%%
%%%%%%%%%%%%%%%%%%%%%%%%%%%%%%%%%%%%%%%%%%%%%%%%%%%%%

\subsection{MDP moment bounds: Proof of Lemma~\ref{lemma:boundeta}}
\label{app:proof:lemma:boundeta}
Let~$p\ge2$, $N>0$, $v\in\cA_N$, $\ep>0$ and $t\in\bT$. 
Starting from~\eqref{eq:etaepv}, we use Cauchy-Schwarz and BDG inequalities to obtain
\begin{align}
\label{eq:etasplit}
    \bE \left[ \abs{\eta^{\ep,v}_t }^p   \right]
    \le & \, 5^{p-1} \frac{\abs{X^\ep_0-x_0}^p}{\big(\vartheta_\ep h_\ep   \big)^p} \\
    & + 5^{p-1}  
    \bE \abs{ \int_0^t K(t-s) \frac{ b_\ep \big(s,\Theta^{\ep,v}_s \big)
    -b \big(s,\Theta^{\ep,v}_s \big)}{\vartheta_\ep h_\ep  } \ds}^p \nonumber \\
    & + 5^{p-1} \bE \abs{ \int_0^t K(t-s) \frac{ b \big(s,\Theta^{\ep,v}_s \big)-b(s,\overbar{X}_s )}{\vartheta_\ep h_\ep  } \D s }^p \nonumber \\
    & + 5^{p-1} N^{p/2}  \bE\left[ \left( \int_0^t \abs{ K(t-s)   \sigma_\ep \big(s,\Theta^{\ep,v}_{s}  \big) }^2 \D s \right)^{p/2} \right]  \nonumber \\
    & + \frac{5^{p-1} C_p}{h_\ep  ^p} \bE   \left[ \left( \int_0^t \abs{ K(t-s)  \sigma_\ep \big(s,\Theta^{\ep,v}_{s}  \big)}^2 \ds \right)^{p/2}\right] . \nonumber
\end{align}
The first term converges by {\bf H7} and is thus bounded.
Notice that {\bf H8} entails 
\begin{equation*}
    \bE \abs{ \int_0^t K(t-s) \frac{ b_\ep \big(s,\Theta^{\ep,v}_s  \big)
    -b \big(s,\Theta^{\ep,v}_s  \big)}{\vartheta_\ep h_\ep  } \ds}^p
    \le \left(\frac{\nu_\ep}{\vartheta_\ep  h_\ep  } \right)^p 
    \bE \abs{ \int_0^t K(t-s) \Xi\big(\Theta^{\ep,v}_{s} \big) \ds}^p,
\end{equation*}
which is bounded because~$\nu_\ep (\vartheta_\ep h_\ep  )^{-1}$ tends to zero and, using Cauchy-Schwarz and Jensen's inequalities in the same way as~\eqref{eq:Holdertrickb} and the bound~\eqref{eq:Xibound},
\[
\bE\abs{ \int_0^t K(t-s) \Xi\big( \Theta^{\ep,v}_s\big) \ds}^p \le
\norm{K}^p_2  t^{p/2-1} \sup_{s\le T} \bE \Big[ \abs{\Xi\big( \Theta^{\ep,v}_{s}\big)}^{p} \Big] \le C_1,
\]
where~$C_1$ is a positive constant that does not depend on~$\ep$.
%%%%%%%%%%%%%
Since $b$ is globally Lipschitz continuous, there exists $C_b>0$ such that for all~$s\in\bT$:
\begin{equation}
\label{eq:Lipschitz}
\abs{ b \big(s,\overbar{X}_s + \vartheta_\ep h_\ep   \eta^{\ep,v}_s  \big)-b(s,\overbar{X}_s)} \le C_b\,  \vartheta_\ep h_\ep   \abs{\eta^{\ep,v}_s}\, .
\end{equation}
Therefore, using Cauchy-Schwarz and Jensen's inequalities again
\begin{align*}
\bE \abs{ \int_0^t K(t-s) \frac{ b \big(s,\Theta^{\ep,v}_s \big)-b(s,\overbar{X}_s)}{\vartheta_\ep h_\ep  } \D s }^p
&\le \bE \abs{ \int_0^t K(t-s) C_b\, \abs{\eta^{\ep,v}_s}\,  \ds}^p \\
&\le C_b^p t^\frac{p^2+4}{2(p-2)} \norm{K}_2^{p-2} \, \int_0^t \abs{K(t-s)}^2 \bE \big[ \abs{\eta^{\ep,v}_s}^p\, \big] \ds.
\end{align*}
%%%%%%%%%%%%%%%%%%%%%%%%%%%%%%%%%%%%%%%%%%%
The last two terms of~\eqref{eq:etasplit} are also uniformly bounded in~$n$ and~$\ep$ because, similarly to~\eqref{eq:Holdertrick},
\begin{align}
\label{eq:boundnoise}
    &\bE\left[ \left( \int_0^t \abs{ K(t-s)  \sigma_\ep \big(s,\Theta^{\ep,v}_{s} \big) }^2 \D s \right)^{p/2} \right] \nonumber \\
    &\le \bE \left[ \left(\int_0^t \abs{K(t-s)}^2 \abs{ \sigma_\ep \big(s,\Theta^{\ep,v}_{s} \big) }^{p} \ds
    \right) \left( \int_0^t \abs{K(t-s)}^2 \ds \right)^\frac{p-2}{2}
    \right] \nonumber \\
    & \le \norm{K}_2^{p}  C_L\big( 1+ \sup_{t\in\bT} \bE\left[ \abs{\Theta^{\ep,v}_{t}}^{p} \right] \big),
\end{align}
by H\"older's inequality and the linear growth condition {\bf H3a}. If the latter fails we rely on {\bf H3b}, Assumption~\ref{assu:Gammabound} and the same calculations as in~\eqref{eq:ineqg} to obtain a similar bound.

Overall this results in the existence of a constant~$C_2>0$
independent of~$\ep,v,t$ such that
\[
\bE \left[\abs{\eta^{\ep,v}_t}^p  \right] \le C_2 + C_2 \int_0^t \abs{K(t-s)}^2 \bE \left[\abs{\eta^{\ep,v}_s}^p  \right]\ds,
\qquad\text{for all }t\in\bT,
\]
and Lemma~\ref{lemma:gronwallK} yields the bound uniform in~$\ep$ and~$t$.

%%%%%%%%%%%%%%%%%%%%%%%%%%%%%%%%%%%%%%%%%%%%%%%%%%%%%
%%%%%%%%%%%%%%%%%%%%%%%%%%%%%%%%%%%%%%%%%%%%%%%%%%%%%
\subsection{MDP tightness: Proof of Lemma~\ref{lemma:tightMDP}}
\label{app:proof:lemma:tightMDP}
Let~$p> 2 \vee 2/\gamma$, $N>0$, a sequence~$\{v^\ep\}_{\ep>0}$ in~$\cA_N$, $\ep>0$, and $0\le s < t\le T$. We proceed as in Lemma~\ref{lemma:tightLDP}; starting from~\eqref{eq:etaepv}, applying consecutively~{\bf H8}, then~\eqref{eq:Lipschitz}, Cauchy-Schwarz and BDG inequalities we obtain
\begin{align*}
    \bE & \left[\abs{\eta^{\ep,v^\ep}_t- \eta^{\ep,v^\ep}_s}^p \right]\\
    \le   & 8^{p-1} \bE
    \abs{ \frac{\nu_\ep}{\vartheta_\ep  h_\ep  } \int_0^s \big(K(t-u) - K(s-u) \big) \Xi\big(  \Theta^{\ep,v^\ep}_u \big) \du}^p \\
    & + 8^{p-1} \bE \abs{ \frac{\nu_\ep}{\vartheta_\ep  h_\ep  } \int_s^t K(t-u) \Xi\big(  \Theta^{\ep,v^\ep}_u \big) \du }^p   \\
    & + 8^{p-1} \bE \abs{ \int_0^s \big(K(t-u)-K(s-u)\big) C_b\,\abs{\eta^{\ep,v^\ep}_u} \du }^p \\
    & + 8^{p-1} \bE \abs{ \int_s^t K(t-u) C_b\, \abs{\eta^{\ep,v^\ep}_u}  \du }^p \\
    & + 8^{p-1} N^{p/2}  \bE\left[ \left( \int_0^s \abs{ \big(K(t-u)-K(s-u)\big)   \sigma_\ep \big(u,\Theta^{\ep,v^\ep}_u \big) }^2  \du \right)^{p/2} \right] \\
    & + 8^{p-1} N^{p/2}  \bE\left[ \left( \int_s^t \abs{ K(t-u)   \sigma_\ep \big(u,\Theta^{\ep,v^\ep}_u \big) }^2  \du \right)^{p/2} \right] \\
    & + \frac{8^{p-1} C_p}{h_\ep  ^p} \bE   \left[ \left( \int_0^s \abs{ \big(K(t-u)-K(s-u)\big)  \sigma_\ep \big(u,\Theta^{\ep,v^\ep}_u \big)}^2 \du \right)^{p/2}\right] \\
    & + \frac{8^{p-1} C_p}{h_\ep  ^p} \bE \left[ \left( \int_s^t \abs{ K(t-u)  \sigma_\ep \big(u,\Theta^{\ep,v^\ep}_u \big)}^2 \du \right)^{p/2}\right].
\end{align*}
In the first four terms, Cauchy-Schwarz inequality allows to separate the kernels from the random variables. 
For the last four terms, analogous calculations to~\eqref{eq:Holdertrick} achieve a similar separation of kernels and random variables. 
Then linear growth or~\eqref{assu:Gammabound} and bounds on~$\sup_{t\in\bT,\ep>0} \bE\left[\big\lvert \Xi\big( \Theta^{\ep,v^\ep}_t \big) \big\lvert^{p}\right]$ and~$\sup_{t\in\bT,\ep>0} \bE\left[\big\lvert\eta^{\ep,v^\ep}_t \big\lvert^{p}\right]$ lead to the existence of $C_1>0$ independent of~$t$ and~$\ep$ such that
\begin{align*}
    \bE \left[\abs{\eta^{\ep,v^\ep}_t- \eta^{\ep,v^\ep}_s}^p \right]
    & \le C_1 \left( \abs{ \int_0^s \big(K(t-u) - K(s-u) \big)^2 \du }^{p/2} + \abs{\int_s^t K(t-u)^2 \du }^{p/2}   \right) \\
    & = C_1 \left( \abs{ \int_0^s \big(K(u+t-s) - K(u) \big)^2 \du }^{p/2} + \abs{\int_0^{t-s} K(u)^2 \du }^{p/2}   \right).
\end{align*}
Hence Assumption~\ref{assu:K} yields the existence of a constant $C_2>0$ such that
\begin{align*}
    \bE \left[\abs{\eta^{\ep,v^\ep}_t- \eta^{\ep,v^\ep}_s}^p \right]
    \le C_2 (t-s)^{\gamma p/2}.
\end{align*}
Then Kolmogorov continuity theorem asserts that  ~$\eta^{\ep,v^\ep}$ admits a version which is H\"older continuous on~$\bT$ of any order~$\alpha < \gamma/2 -1/p$, uniformly in~$\ep>0$ and which satisfies~\eqref{eq:holderMDP}. Furthermore, Aldous theorem~\cite[Theorem 16.10]{Billingsey99} states that the sequence~$\{ \eta^{\ep,v^\ep}\}_{\ep>0}$ is tight.

%%%%%%%%%%%%%%%%%%%%%%%%%%%%%%%%%%%%%%%%%%%%%%%%%%%%%
%%%%%%%%%%%%%%%%%%%%%%%%%%%%%%%%%%%%%%%%%%%%%%%%%%%%%
\subsection{MDP weak convergence: Proof of Lemma~\ref{lemma:weakcvgMDP}}
\label{app:proof:lemma:weakcvgMDP}

We have shown in Lemma~\ref{lemma:tightMDP} that for any subsequence~$\{\ep_k\}_{k\in\bN}$,~$\{\eta^{\ep_k,v^{\ep_k}} \}_{k\in\bN}$ and $\{v^{\ep_k}\}_{k\in\bN}$ are tight as families of random variables with values in $\cW^d$ and $\cS_N$ respectively. 
By Skorohod representation theorem we can work with almost sure convergence for the purpose of identifying the limit.
Hence there exists a subsubsequence, denoted hereafter~$\big\{\eta^k,v^k \big\}$, that converges almost surely in the product topology on~$\cW^d\times \cS_N$ to some~$\cW^d\times \cS_N$-valued limit~$(\eta^0,v)$ in a possibly different probability space~$(\Omega^0,\cF^0,\bP^0)$ as~$n$ tends to~$+\infty$. We also denote~$\ep_k, b_k, \sigma_k, X_0^k, \Theta^k$ along this subsequence.

The convergence of the couple also takes place in distribution, and we follow the same method as in the LDP case which comes from \cite{Chiarini14}. For all~$t\in[0,T]$, let~$\Psi_t:\cS_N \times \cW^d \to \bR$ such that
\[
\Psi_t(f,\omega):=\abs{\omega_t - \int_0^t K(t-s) \left[\nabla b(s,\overbar{X}_s)\omega_s + \sigma(s,\overbar{X}_s)f_s\right] \D s}\wedge 1.
\]
Clearly,~$\Psi_t$ is bounded and one can show its continuity along the same lines as in the LDP proof but in a simpler way because of the linearity.
Therefore 
\[
\lim_{k\uparrow\infty} \bE\left[\Psi_t(v^k,\eta^k)\right] = \bE^0\left[\Psi_t(v,\eta^0)\right],
\]
and we prove that the left-hand side is actually equal to zero.
By {\bf H5} and Taylor's formula there exists a sequence of $\bR^d$-valued stochastic processes $\{R^\ep\}_{\ep>0}$ such that
\begin{align}
\label{eq:taylor}
    \frac{ b \big(s,\overbar{X}_s + \vartheta_\ep h_\ep   \eta^{\ep,v}_s \big)-b(s,\overbar{X}_s)}{\vartheta_\ep h_\ep  } = \nabla b\big(s, \overbar{X}_s \big) \eta^{\ep,v}_s + R^\ep(s),
\end{align}
and a constant $C_R>0$ such that
\[
\abs{R^\ep(s)} \le C_R \vartheta_\ep h_\ep \abs{\eta^{\ep,v}_s}^2.
\]
We recall that $\normbT{\nabla b(\cdot,\overbar{X})}$
is finite by {\bf H5} and observe that, by~\eqref{eq:boundveta}:
\begin{equation}
\bE \big[\abs{ R^\ep(u)}^p \big]\le \big(C_R \vartheta_\ep h_\ep \big)^p \bE \left[ \big\lvert\eta^{\ep,v^\ep}_s\big\lvert^{2p}\right]  <\infty.
\label{eq:taylorrest}
\end{equation}
Again starting from~\eqref{eq:etaepv},
we use~{\bf H8}, the Taylor estimate~\eqref{eq:taylor} and It\^o isometry to get
\begin{align*}
    \bE \left[\Psi_t(\eta^k,v^k)^2 \right]
    \le &  \frac{5\abs{X^k_0-x_0}^2}{\vartheta_{\ep_k}^2 h_{\ep_k}^2} \\
    & + 5 \left(\frac{\nu_{\ep_k}}{\vartheta_{\ep_k} h_{\ep_k}}\right)^2 \bE \abs{ \int_0^t K(t-s) \Xi\big(\Theta^{k}_s\big) \ds}^2 \\
    & + 5  \bE \abs{ \int_0^t K(t-s) R^\ep(s) \ds }^2 \\
    & + 5  \bE\abs{ \int_0^t  K(t-s)  \Big[ \sigma_k \big(s,\Theta^{k}_s \big)-\sigma\big(s,\overbar{X}_s\big)\Big] v^k_s \ds }^2 \\
    & + \frac{5}{h_{\ep_k}^2}  \bE\left[  \int_0^t \abs{ K(t-s)  \sigma_k \big(s,\Theta^{k}_s \big)}^2 \ds    \right] \\
     =: & 5 \big(\text{I}_k +\text{II}_k + \text{III}_k+ \text{IV}_k+ h_{\ep_k}^{-2}  \text{V}_k \big).
\end{align*}
{\bf H7} and {\bf H8} tell us that~$\text{I}_k +\text{II}_k$ tends to zero as~$\ep$ goes to zero while an application of Cauchy-Schwarz inequality and the bound~\eqref{eq:taylorrest} yields the same conclusion for $\text{III}_k$.

To deal with IV$_k$, recall that~$\Theta^{k}$ converges in distribution towards~$\overbar{X}$ as~$k$ tends to infinity. Since~$\overbar{X}$ is deterministic, the convergence actually takes place in probability, with respect to the topology of uniform convergence. 
Moreover~$\Theta^{k}$ is uniformly bounded in~$L^{r}(\Omega)$ for all~$r>p$, thus the family~$\{\abs{\Theta^{k}}^p\}_{k\ge0}$ is uniformly integrable. Therefore the convergence also occurs with respect to the~$L^p(\Omega)$-norm.

The modulus of continuity of~$\sigma$ is only available on compact sets of~$\bT\times\bR^d$ so we define a constant~$M> \normbT{\overbar{X}}$ and introduce the following sets, for each~$k\in\bN$:
\[
E_k := \bigg\{ \omega\in\Omega: \normbT{ \Theta^{k}(\omega) - \overbar{X}} \le M \bigg\},
\]
with the observation that~$\lim_{k\uparrow\infty}\bP ( E_k)=1$ thanks to the previous argument. Since~$\overbar{X}$ is uniformly bounded,~$\abs{\Theta^{k}_t (\omega)}\le 2M$ for all~$t\in\bT,\omega\in E_k$,~$k\ge0$. 
Therefore using Cauchy-Schwarz inequality,
\begin{align*}
    \bE \left[\text{IV}_k \right]
    & \le 2N  \bE\left[\int_0^t  \abs{K(t-s)}^2 \abs{ \sigma_k \big(s,\Theta^{k}_s \big) -\sigma \big(s,\Theta^{k}_s \big)}^2\D s
    (\one_{E_k} + \one_{E_k^c})\right] \\
    & + 2N \bE\left[\int_0^t \abs{ K(t-s)}^2 \abs{
    \sigma\big(s,\Theta^{k}_s\big) - \sigma(s,\overbar{X}_s) }^2 \ds
    (\one_{E_k} + \one_{E_k^c})\right],
\end{align*}
where we will use the localisation to obtain convergence in the first term and H\"older continuity in the second. Let us assume for the moment that linear growth {\bf H3a} holds. We use that $\Theta^k$ is uniformly bounded by $2M$ in $E_k$ and linear growth for both $\sigma_k$ and $\sigma$ to obtain
\begin{align}
    \sup_{s\in\bT} \bE \left[ \abs{
    \sigma_k\big(s,\Theta^{k}_s\big) - \sigma(s,\Theta^k_s) }^2 \right]
    = & \sup_{s\in\bT} \bE \left[ \abs{
    \sigma_k\big(s,\Theta^{k}_s\big) - \sigma(s,\Theta^k_s) }^2 \one_{E_k} \right]
    + \sup_{s\in\bT} \bE \left[ \abs{
    \sigma_k\big(s,\Theta^{k}_s\big) - \sigma(s,\Theta^k_s) }^2 \one_{E_k^c} \right]\nonumber \\
    \le & \normbT{\norm{\sigma_k - \sigma}_{2M}}^2 + \sup_{s\in\bT} \bE \left[ \one_{E_k^c} C_L^2 \big( 2 + 2 \lvert \Theta^{k}_s\lvert \big)^2 \right],
\label{eq:sigmacvg}
\end{align}
which tends to zero as~$k$ goes to infinity because of {\bf H2} for the first term 
and because $\bP(\Omega\setminus E_k)$ tends to zero for the second. 
Moreover, by~{\bf H6}, there exists~$\delta>0$ such that $\sigma$ is locally $\delta$-H\"older continuous thus there exist $C_{2M}>0$ such that
\[
\sup_{s\in\bT} \bE \left[ \one_{E_k} \abs{
    \sigma\big(s,\Theta^{k}_s\big) - \sigma(s,\overbar{X}_s) }^2 \right]
    \le  \sup_{s\in\bT} \bE \left[ \one_{E_k} C_{2M} \abs{\Theta^{k}_s - \overbar{X}_s}^{2\delta}
    \right],
\]
which tends to zero.
Finally, linear growth leads to
\begin{equation}
\sup_{s\in\bT} \bE \left[ \one_{E_k^c} \abs{
    \sigma\big(s,\Theta^{k}_s\big) - \sigma(s,\overbar{X}_s) }^2 \right]
    \le
    \bE \left[ \one_{E_k^c} C_L^2 \big( 2 + \lvert \Theta^{k}_s\lvert + \abs{\overbar{X}_s} \big)^2 \right],
\label{eq:sigmacty}
\end{equation}
which also tends to zero because $\bP(\Omega\setminus E_k)$ tends to zero.
If only {\bf H3b} with Assumption~\ref{assu:Gammabound} hold then a different bound depending on~\eqref{eq:UpsilonGrowth} would replace those in~\eqref{eq:sigmacvg} and~\eqref{eq:sigmacty}, by noticing that~$\overbar{X}=\cG^0(0)$. In both cases the above estimates tend to zero as~$k$ tends to infinity, hence~$\bE \left[\text{IV}_k \right]$ converges towards zero. Finally, $\{$V$_k\}_{k\in\bN}$ is uniformly bounded across~$k\ge0$ as~\eqref{eq:boundnoise} shows, thus~$h_{\ep_k}^{-p} \text{V}_k$ tends to zero.
We have proved that
\[
\lim_{k\uparrow\infty} \bE\left[\Psi_t(v^k,\eta^k)\right] = 0,
\]
and this entails that the limit~$\eta^0$ satisfies~\eqref{eq:MDPlimit}~$\bP^0$-almost surely, for all~$t\in\bT$. Since~$\eta^0$ has continuous paths, this holds for all~$t\in\bT$, $\bP^0$-almost surely and the solution is unique therefore we conclude that~$\eta^0=\psi$. Every subsequence has a subsequence for which this convergence holds therefore~$\eta^{\ep,v^\ep}$ converges weakly towards~$\psi$ as~$\ep$ goes to zero.

%%%%%%%%%%%%%%%%%%%%%%%%%%%%%%%%%%%%%%%%%%%%%%%%%%%%%
%%%%%%%%%%%%%%%%%%%%%%%%%%%%%%%%%%%%%%%%%%%%%%%%%%%%%

\bibliographystyle{abbrv}
\bibliography{bib}

\begin{thebibliography}{10}

\bibitem{ALP17}
E.~{Abi Jaber}, M.~Larsson, and S.~Pulido.
\newblock Affine {V}olterra processes.
\newblock {\em Annals of Applied Probability}, 29(5):3155--3200, 2017.

\bibitem{Alos07}
E.~Al\`os, J.~A. Le\'on, and J.~Vives.
\newblock On the short-time behavior of the implied volatility for
  jump-diffusion models with stochastic volatility.
\newblock {\em Finance and Stochastics}, 11(4):571--589, 2007.

\bibitem{BNPJ14}
O.~E. Barndorff-Nielsen, M.~S. Pakkanen, and J.~Schmiegel.
\newblock Assessing relative volatility/intermittency/energy dissipation.
\newblock {\em Electronic Journal of Statistics}, 8(2):1996--2021, 2014.

\bibitem{BNS09}
O.~E. Barndorff-Nielsen and J.~Schmiegel.
\newblock Brownian semistationary processes and volatility/intermittency.
\newblock {\em Radon Series on Computational and Applied Mathematics}, 8:1--26,
  2009.

\bibitem{BFGMS19}
C.~Bayer, P.~K. Friz, P.~Gassiat, J.~Martin, and B.~Stemper.
\newblock A regularity structure for rough volatility.
\newblock {\em Mathematical Finance}, pages 1--51, 2019.

\bibitem{BFG16}
C.~Bayer, P.~K. Friz, and J.~Gatheral.
\newblock Pricing under rough volatility.
\newblock {\em Quantitative Finance}, 16(6):887--904, 2016.

\bibitem{BFGHS17}
C.~Bayer, P.~K. Friz, A.~Gulisashvili, B.~Horvath, and B.~Stemper.
\newblock Short-time near-the-money skew in rough fractional volatility models.
\newblock {\em Quantitative Finance}, 19(5):779--798, 2019.

\bibitem{BLP15}
M.~Bennedsen, A.~Lunde, and M.~Pakkanen.
\newblock Hybrid scheme for {B}rownian semistationary processes.
\newblock {\em Finance and Stochastics}, 21:931--965, 2017.

\bibitem{BBDW19}
S.~Bhamidi, A.~Budhiraja, P.~Dupuis, and R.~Wu.
\newblock Rare event asymptotics for exploration processes for random graphs.
\newblock Preprint, \href{https://arxiv.org/pdf/1912.04714}{arXiv:1912.04714}.
  Forthcoming in \emph{The Annals of Applied Probability}, 2019.

\bibitem{Billingsey99}
P.~Billingsey.
\newblock {\em Convergence of probability measures}.
\newblock Wiley series in probability and statistics, 1999.

\bibitem{boue98}
M.~Bou\'e and P.~Dupuis.
\newblock A variational representation for certain functionals of {B}rownian
  motion.
\newblock {\em Annals of Probability}, 26(4):1641--1659, 1998.

\bibitem{BCD13}
A.~Budhiraja, J.~Chen, and P.~Dupuis.
\newblock Large deviations for stochastic partial differential equations driven
  by a {P}oisson random measure.
\newblock {\em Stochastics Processes and Applications}, 123(2):523--560, 2013.

\bibitem{BD01}
A.~Budhiraja and P.~Dupuis.
\newblock A variational representation for positive functionals of infinite
  dimensional {B}rownian motion.
\newblock {\em Probability and Mathematical Statistics}, 20(1), 2001.

\bibitem{BD19}
A.~Budhiraja and P.~Dupuis.
\newblock {\em Analysis and Approximation of Rare Events Representations and
  Weak Convergence Methods}.
\newblock Springer, 2019.

\bibitem{BDG16}
A.~Budhiraja, P.~Dupuis, and A.~Ganguly.
\newblock Moderate deviations principles for stochastic differential equations
  with jumps.
\newblock {\em Annals of Probability}, 44(3):1723--1775, 2016.

\bibitem{BDG18}
A.~Budhiraja, P.~Dupuis, and A.~Ganguly.
\newblock Large deviations for small noise diffusions in a fast {M}arkovian
  environment.
\newblock {\em Electronic Journal of Probability}, 23(112):1--33, 2018.

\bibitem{BDM08}
A.~Budhiraja, P.~Dupuis, and V.~Maroulas.
\newblock Large deviations for infinite dimensional stochastic dynamical
  systems.
\newblock {\em Annals of Probability}, 36(4):1390--1420, 2008.

\bibitem{BDM11}
A.~Budhiraja, P.~Dupuis, and V.~Maroulas.
\newblock Variational representations for continuous time processes.
\newblock {\em Annals of the Institute Henri Poincaré Probability and
  Statistics}, 47(3):725--747, 2011.

\bibitem{BFW20}
A.~Budhiraja, E.~Friedlander, and R.~Wu.
\newblock Many-server asymptotics for join-the-shortest-queue: Large deviations
  and rare events.
\newblock {\em The Annals of Applied Probability}, 31(5), 2021.

\bibitem{Cellupica19}
M.~Cellupica and B.~Pacchiarotti.
\newblock Pathwise asymptotics for {V}olterra type rough volatility models.
\newblock {\em Journal of Theoretical Probability}, 2020.

\bibitem{Chevillard17}
L.~Chevillard.
\newblock Regularized fractional {O}rnstein-{U}hlenbeck processes and their
  relevance to the modeling of fluid turbulence.
\newblock {\em Physical Review E}, 97, 2017.

\bibitem{Chiarini14}
F.~Chiarini and M.~Fischer.
\newblock On large deviations for small noise {I}t\^o processes.
\newblock {\em Advances in Applied Probability}, 46(4):1126--1147, 2014.

\bibitem{CV12a}
A.~Chronopoulou and F.~G. Viens.
\newblock Estimation and pricing under long-memory stochastic volatility.
\newblock {\em Annals of Finance}, 8:379--403, 2012.

\bibitem{CV12b}
A.~Chronopoulou and F.~G. Viens.
\newblock Stochastic volatility and option pricing with long-memory in discrete
  and continuous time.
\newblock {\em Quantitative Finance}, 12:635--649, 2012.

\bibitem{Comte12}
E.~Comte, L.~Coutin, and E.~Renault.
\newblock Affine fractional stochastic volatility models with application to
  option pricing.
\newblock {\em Annals of Finance}, 8:337--378, 2012.

\bibitem{Comte98}
E.~Comte and E.~Renault.
\newblock Long memory in continuous-time stochastic volatility models.
\newblock {\em Mathematical Finance}, 8:291--323, 1998.

\bibitem{CDMD15}
G.~Conforti, S.~De~Marco, and J.-D. Deuschel.
\newblock On small-noise equations with degenerate limiting system arising from
  volatility models.
\newblock {\em In: Friz P., Gatheral J., Gulisashvili A., Jacquier A.,
  Teichmann J. (eds) Large Deviations and Asymptotic Methods in Finance.
  Springer Proceedings in Mathematics \& Statistics}, 2015.

\bibitem{CHPP13}
J.~M. Corcuera, E.~Hedevang, M.~S. Pakkanen, and M.~Podolskij.
\newblock Asymptotic theory for {B}rownian semi-stationary processes with
  application to turbulence.
\newblock {\em Stochastic Processes and their Applications}, 123(7):2552--2574,
  2017.

\bibitem{CD99}
L.~Coutin and L.~Decreusefond.
\newblock Abstract nonlinear filtering theory in the presence of fractional
  {B}rownian motion.
\newblock {\em Annals of Applied Probability}, 9(4):1058–1090, 1999.

\bibitem{CD00}
L.~Coutin and L.~Decreusefond.
\newblock {V}olterra differential equations with singular kernels.
\newblock {\em Proceedings of the Workshop on Mathematical Physics and
  Stochastic Analysis}, 2000.

\bibitem{Cuchiero18}
J.~Cuchiero and J.~Teichmann.
\newblock Generalized {F}eller processes and {M}arkovian lifts of stochastic
  {V}olterra processes: the affine case.
\newblock {\em Journal of Evolution Equations}, 2020.

\bibitem{Decreusefond02}
L.~Decreusefond.
\newblock Regularity properties of some stochastic {V}olterra integrals with
  singular kernels.
\newblock {\em Potential Analysis}, 16:139--149, 2002.

\bibitem{DZ10}
A.~Dembo and O.~Zeitouni.
\newblock {\em Large Deviations Techniques and Applications}.
\newblock Springer-Verlag Berlin Heidelberg, 1998.

\bibitem{DS89}
J.-D. Deuschel and D.~W. Stroock.
\newblock {\em Large Deviations}.
\newblock Academic Press Inc., 1989.

\bibitem{Donati04}
C.~Donati-Martin, A.~Rouault, M.~Yor, and M.~Zani.
\newblock Large deviations for squares of {B}essel and {O}rnstein-{U}hlenbeck
  processes.
\newblock {\em Probability Theory and Related Fields}, 129:261--289, 2004.

\bibitem{DE97}
P.~Dupuis and R.~S. Ellis.
\newblock {\em A Weak Convergence Approach to the Theory of Large Deviations}.
\newblock Wiley, 1997.

\bibitem{DS12}
P.~Dupuis and K.~Spiliopoulos.
\newblock Large deviations for multiscale diffusions via weak convergence
  methods.
\newblock {\em Stochastic Processes and their Applications}, 122(4):1947--1987,
  2012.

\bibitem{DSW12}
P.~Dupuis, K.~Spiliopoulos, and H.~Wang.
\newblock Importance sampling for multiscale diffusions.
\newblock {\em SIAM Journal on Multiscale Modeling and Simulation},
  10(1):1--27, 2012.

\bibitem{EFR18}
O.~El~Euch, M.~Fukasawa, and M.~Rosenbaum.
\newblock The microstructural foundations of leverage effect and rough
  volatility.
\newblock {\em Finance and Stochastics}, 22(2):241–280, 2018.

\bibitem{ER18}
O.~El~Euch and M.~Rosenbaum.
\newblock Perfect hedging in rough {H}eston models.
\newblock {\em Annals of Applied Probability}, 28(6):3813--3856, 2018.

\bibitem{ER19}
O.~El~Euch and M.~Rosenbaum.
\newblock The characteristic function of rough {H}eston models.
\newblock {\em Mathematical Finance}, 29(1):3--38, 2019.

\bibitem{FGS19}
M.~Forde, S.~Gerhold, and B.~Smith.
\newblock Small-time, large-time, and {$H\to0$} asymptotics for the rough
  {H}eston model.
\newblock {\em Mathematical Finance}, 31(1), 2020.

\bibitem{FZ15}
M.~Forde and H.~Zhang.
\newblock Asymptotics for rough stochastic volatility models.
\newblock {\em SIAM Journal on Financial Mathematics}, 8(1):114--145, 2017.

\bibitem{Freidlin84}
M.~I. Freidlin and A.~D. Wentzell.
\newblock {\em Random Perturbations of Dynamical Systems}.
\newblock Springer-Verlag New York, 1984.

\bibitem{FGaP18}
P.~K. Friz, P.~Gassiat, and P.~Pigato.
\newblock Precise asymptotics: robust stochastic volatility models.
\newblock {\em The Annals of Applied Probability}, 2(31), 2021.

\bibitem{FGeP18}
P.~K. Friz, S.~Gerhold, and A.~Pinter.
\newblock Option pricing in the moderate deviations regime.
\newblock {\em Mathematical Finance}, 28(3), 2018.

\bibitem{Fukasawa11}
M.~Fukasawa.
\newblock Asymptotic analysis for stochastic volatility: martingale expansion.
\newblock {\em Finance and Stochastics}, 15:635--654, 2011.

\bibitem{Fukasawa17}
M.~Fukasawa.
\newblock Short-time at-the-money skew and rough fractional volatility.
\newblock {\em Quantitative Finance}, 17(2):189--198, 2017.

\bibitem{GL14}
K.~Gao and R.~Lee.
\newblock Asymptotics of implied volatility to arbitrary order.
\newblock {\em Finance and Stochastics}, 18(2):349--392, 2014.

\bibitem{Gassiat19}
P.~Gassiat.
\newblock On the martingale property in the rough {B}ergomi model.
\newblock {\em Electronic Communications in Probability}, 24(33), 2019.

\bibitem{GJR14}
J.~Gatheral, T.~Jaisson, and M.~Rosenbaum.
\newblock Volatility is rough.
\newblock {\em Quantitative Finance}, 18(6):933--949, 2018.

\bibitem{GatheralKR}
J.~Gatheral and M.~Keller-Ressel.
\newblock Affine forward variance models.
\newblock {\em Finance and Stochastics}, 23:501--533, 2019.

\bibitem{Gripenberg90}
G.~Gripenberg, S.-O. Londen, and O.~Staffans.
\newblock {\em Volterra integral and functional equations}.
\newblock Encyclopedia of Mathematics and its Applications. Cambridge
  University Press, 1990.

\bibitem{GJRS18}
H.~Guennon, A.~Jacquier, P.~Roome, and F.~Shi.
\newblock Asymptotic behavior of the fractional {H}eston model.
\newblock {\em SIAM Journal on Financial Mathematics}, 9(3):1017--1045, 2018.

\bibitem{Gulisashvili18}
A.~Gulisashvili.
\newblock Large deviation principle for {V}olterra type fractional stochastic
  volatility models.
\newblock {\em SIAM Journal on Financial Mathematics}, 9(3):1102–1136, 2018.

\bibitem{HJL19}
B.~Horvath, A.~Jacquier, and C.~Lacombe.
\newblock Asymptotic behaviour of randomised fractional volatility models.
\newblock {\em Journal of Applied Probability}, 56(2):496--523, 2019.

\bibitem{HJM17}
B.~Horvath, A.~Jacquier, and A.~Muguruza.
\newblock Functional central limit theorems for rough volatility.
\newblock Preprint, \href{https://arxiv.org/abs/1711.03078}{arXiv:1711.03078},
  2017.

\bibitem{JPS18}
A.~Jacquier, M.~Pakkanen, and H.~Stone.
\newblock Pathwise large deviations for the rough {B}ergomi model.
\newblock {\em Journal of Applied Probability}, 55(4):1078--1092, 2018.

\bibitem{JK20}
A.~Jacquier and K.~Spiliopoulos.
\newblock Pathwise moderate deviations for option pricing.
\newblock {\em Mathematical Finance}, 30(2):426--463, 2020.

\bibitem{KS98}
I.~Karatzas and S.~Shreve.
\newblock {\em Brownian motion and stochastic calculus}.
\newblock Springer-Verlag New York, 1998.

\bibitem{LMS19}
C.~Lacombe, A.~Muguruza, and H.~Stone.
\newblock Asymptotics for volatility derivatives in multi-factor rough
  volatility models.
\newblock {\em Mathematics and Financial Economics}, 15, 2021.

\bibitem{Li17}
Y.~Li, R.~Wang, N.~Yao, and S.~Zhang.
\newblock A moderate deviation principle for stochastic {V}olterra equation.
\newblock {\em Statistics and Probability Letters}, 122, 2017.

\bibitem{Marie14}
N.~Marie.
\newblock A pathwise fractional one compartment intra-veinous bolus model.
\newblock {\em International Journal of Statistics and Probability}, 3, 2014.

\bibitem{MP18}
R.~McCrickerd and M.~Pakkanen.
\newblock Turbocharging {M}onte--{C}arlo pricing for the rough {B}ergomi model.
\newblock {\em Quantitative Finance}, 18(11):1877--1886, 2018.

\bibitem{Morse17}
M.~R. Morse and K.~Spiliopoulos.
\newblock Moderate deviations principle for systems of slow-fast diffusions.
\newblock {\em Asymptotic Analysis}, 105:97--135, 2017.

\bibitem{Morse18}
M.~R. Morse and K.~Spiliopoulos.
\newblock Importance sampling for slow-fast diffusions based on moderate
  deviations.
\newblock {\em SIAM Journal on Multiscale Modeling and Simulation},
  18(1):315–350, 2018.

\bibitem{MS15}
L.~Mytnik and T.~S. Salisbury.
\newblock Uniqueness for {V}olterra-type stochastic integral equations.
\newblock Preprint \href{https://arxiv.org/abs/1502.05513}{arXiv:1502.05513},
  2015.

\bibitem{Nualart00}
D.~Nualart and C.~Rovira.
\newblock Large deviations for stochastic {V}olterra equations.
\newblock {\em Bernoulli}, 6(2):339--355, 2000.

\bibitem{FGGJTBook}
F.~P., G.~J., G.~A., J.~A., and T.~J.
\newblock {\em Large Deviations and Asymptotic Methods in Finance}.
\newblock Springer Proceedings in Mathematics and Statistics, 2015.

\bibitem{Rob10}
S.~Robertson.
\newblock Sample path large deviations and optimal importance sampling for
  stochastic volatility models.
\newblock {\em Stochastic Processes and their Applications}, 120(1):66--83,
  2010.

\bibitem{Rovira00}
C.~Rovira and M.~Sanz-Sol\'e.
\newblock Large deviations for stochastic {V}olterra equations in the plane.
\newblock {\em Potential Analysis}, 12:359--383, 2000.

\bibitem{Kostas13}
K.~Spiliopoulos.
\newblock Large deviations and importance sampling for systems of slow-fast
  motion.
\newblock {\em Applied Mathematics and Optimization}, 67:123–161, 2013.

\bibitem{SteinStein}
E.~Stein and J.~Stein.
\newblock Stock price distributions with stochastic volatility - an analytic
  approach.
\newblock {\em Review of Financial studies}, 4:727--752, 1991.

\bibitem{VZ18}
F.~Viens and J.~Zhang.
\newblock A martingale approach for fractional {B}rownian motions and related
  path dependent {PDE}s.
\newblock {\em Annals of Applied Probability}, 29:3489--3540, 2019.

\bibitem{YW70}
T.~Yamada and S.~Watanabe.
\newblock On the uniqueness of solutions of stochastic differential equations.
\newblock {\em Journal of Mathematics of Kyoto University}, 11(1):155--167,
  1971.

\bibitem{Zhang08}
X.~Zhang.
\newblock Euler schemes and large deviations for stochastic {V}olterra
  equations with singular kernels.
\newblock {\em Journal of Differential Equations}, 244:2226--2250, 2008.

\end{thebibliography}

\end{document}